\documentclass[11pt]{amsart}
\usepackage{amscd,amsmath,amssymb,amsfonts,verbatim}
\usepackage[cmtip, all]{xy}
\usepackage{MnSymbol}
\usepackage{comment}

\newtheorem{thm}{Theorem}[section]
\newtheorem{prop}[thm]{Proposition}
\newtheorem{lem}[thm]{Lemma}
\newtheorem{cor}[thm]{Corollary}

\theoremstyle{definition}
\newtheorem{ques}[thm]{Question}

\newtheorem{defn}[thm]{Definition}
\theoremstyle{remark}
\newtheorem{remk}[thm]{Remark}
\newtheorem{remks}[thm]{Remarks}

\newtheorem{exm}[thm]{Example}
\newtheorem{exms}[thm]{Examples}
\newtheorem{notat}[thm]{Notation}
\numberwithin{equation}{section}

{\hfill$\square$\end{defn}}
{\hfill$\square$\end{remk}}
{\hfill$\square$\end{remks}}
{\hfill$\square$\end{exm}}
{\hfill$\square$\end{exms}}
{\hfill$\square$\end{notat}}

\newcommand{\thmref}{Theorem~\ref}
\newcommand{\propref}{Proposition~\ref}
\newcommand{\corref}{Corollary~\ref}

\newcommand{\lemref}{Lemma~\ref}

\newcommand{\remref}{Remark~\ref}

\newcommand{\sC}{{\mathcal C}}

\newcommand{\sI}{{\mathcal I}}

\newcommand{\sK}{{\mathcal K}}

\newcommand{\sO}{{\mathcal O}}

\newcommand{\A}{{\mathbb A}}

\newcommand{\F}{{\mathbb F}}
\newcommand{\G}{{\mathbb G}}

\newcommand{\N}{{\mathbb N}}
\renewcommand{\P}{{\mathbb P}}
\newcommand{\Q}{{\mathbb Q}}
\newcommand{\R}{{\mathbb R}}

\newcommand{\W}{{\mathbb W}}

\newcommand{\Z}{{\mathbb Z}}

\newcommand{\fm}{{\mathfrak m}}

\newcommand{\fp}{{\mathfrak p}}

\newcommand{\CH}{{\rm CH}}

\newcommand{\surj}{\twoheadrightarrow}
\newcommand{\inj}{\hookrightarrow}
\newcommand{\red}{{\rm red}}

\newcommand{\Hom}{{\rm Hom}}

\newcommand{\Spec}{{\rm Spec \,}}

\newcommand{\id}{{\operatorname{id}}}

\newcommand{\Sch}{{\operatorname{\mathbf{Sch}}}}

\newcommand{\Sm}{{\mathbf{Sm}}}

\newcommand{\sq}{\square}

\newcommand{\cyc}{{\operatorname{\rm cyc}}}

\newcommand{\ds}{{/\kern-3pt/}}

\renewcommand{\log}{{\operatorname{log}}}

\newcommand{\tr}{{\operatorname{tr}}}

\newcommand{\un}{\underline}
\newcommand{\ov}{\overline}

\renewcommand{\dim}{\text{\rm dim}}

\newcommand{\tuborg}{\left\{\begin{array}{ll}}
\newcommand{\sluttuborg}{\end{array}\right.}

\newcommand{\sfs}{{\rm sfs}}

\newcommand{\ord}{{\rm ord}}
\newcommand{\Tz}{{\rm Tz}}

\newcommand{\TCH}{{\rm TCH}}
\newcommand{\wt}{\widetilde}
\newcommand{\wh}{\widehat}

\newcounter{elno}

\newcounter{elno-abc}   

\newcounter{elno-abc-prime}

\begin{document}
\title{Zero-cycles with modulus and relative $K$-theory}
\author{Rahul Gupta, Amalendu Krishna}
\address{Fakult\"at f\"ur Mathematik, Universit\"at Regensburg, 
93040, Regensburg, Germany.}
\email{Rahul.Gupta@mathematik.uni-regensburg.de}
\address{School of Mathematics, Tata Institute of Fundamental Research,  
1 Homi Bhabha Road, Colaba, Mumbai, 400005, India.}
\email{amal@math.tifr.res.in}




\maketitle

\begin{quote}\emph{Abstract.}
  Let $D$ be an effective Cartier divisor on a regular quasi-projective
  scheme $X$ of dimension $d \ge 1$ over a field.
 For an  integer $n \ge 0$, we construct a
cycle class map from the higher Chow groups with modulus
$\{\CH^{n+d}(X|mD, n)\}_{m \ge 1}$ to the relative $K$-groups
$\{K_n(X,mD)\}_{m \ge 1}$ in the category of pro-abelian groups. We show that this
induces a pro-isomorphism between the additive higher Chow groups of
relative 0-cycles and the reduced algebraic $K$-groups of
truncated polynomial rings over a regular
semi-local ring which is essentially of finite type over a characteristic zero
field.

\end{quote}
\setcounter{tocdepth}{1}
\tableofcontents  


\section{Introduction}\label{sec:Intro}
The story of Chow groups with modulus began with the discovery of additive higher 0-cycles by
Bloch and Esnault in \cite{BE1} and \cite{BE2}. Their hope was that these additive 0-cycle groups would serve
as a guide in developing a theory of motivic cohomology with modulus which could describe 
the algebraic $K$-theory
of non-reduced schemes. Recall that Bloch's original higher Chow groups (equivalently, Voevodsky's 
motivic cohomology) overlook the 
difference  between non-reduced and reduced schemes.

Motivated by the work of Bloch and Esnault, a theory of motivic cohomology with modulus was proposed by
Binda and Saito \cite{BS} in the name of `higher Chow groups with modulus'
(recalled in \S~\ref{sec:HCGM}).
The expectation was that one would be able to describe relative algebraic $K$-theory in terms of these 
Chow groups.
The theory of Chow groups with modulus generalized the theory of
additive higher Chow groups defined by Bloch-Esnault and
further studied by R\"ulling \cite{R}, 
Krishna-Levine \cite{KLevine} and Park \cite{Park}. It also
generalized the theory of 0-cycles with modulus of Kerz-Saito \cite{KeS}
and the higher Chow groups of Bloch \cite{Bloch-1}.


Recall that one way to study the algebraic $K$-theory of a non-reduced
(or any singular) scheme is to embed it as a closed subscheme of 
a smooth scheme and study the 
resulting relative $K$-theory. Since there are motivic cohomology groups
which can completely describe the algebraic $K$-theory of a 
smooth scheme, what one needs is a theory of motivic cohomology 
to describe the relative $K$-theory.

The expectation that the higher Chow groups with modulus should be
the candidate for the motivic cohomology 
to describe the relative $K$-theory has generated a lot of
interest in them in past several years. In a recent work \cite{IK}, 
Iwasa and Kai constructed a theory of Chern classes from the relative $K$-theory
to a variant of the higher Chow groups with modulus. In another work
\cite{IK-2}, they proved a Riemann-Roch type theorem showing that
the relative group $K_0$-group  of an affine modulus pair is rationally
isomorphic to a direct sum of Chow groups with modulus. An integral version of
this isomorphism for all modulus pairs
in dimension up to two was earlier proven by Binda and Krishna \cite{BK}.
They also constructed a cycle class map for relative $K_0$-group in all dimensions.

The above results suggest strong connection 
between cycles with modulus and relative $K$-theory.
However, an explicit construction of cycle class maps in full generality or
Atiyah-Hirzebruch type spectral sequences, which may directly connect Chow groups 
with modulus to relative algebraic $K$-theory, remains a challenging 
problem today.

\subsection{Main results and consequences}\label{sec:MR}
The objective of this paper is to investigate the original question of
Bloch and Esnault \cite{BE2} in this subject. Namely,
can 0-cycles with modulus explicitly describe relative $K$-theory in terms of algebraic cycles?
We provide an answer to this question in this paper.
We prove two results.
The first is that there is indeed a
direct connection between 0-cycles with modulus
and relative $K$-theory in terms of an explicit cycle class map. The second is
that in many cases of interest, these 0-cycles with modulus
are strong enough to completely describe the relative algebraic $K$-theory.
More precisely, we prove the following. The terms and notations used
in the statements of these results are explained in the body of the text. 
In particular, we refer to \S~\ref{sec:Rel-K} for relative $K$-theory
and \S~\ref{sec:HCGM} for higher Chow groups with modulus.

\begin{thm}\label{thm:Intro-1}
Let $X$ be a regular quasi-projective variety of pure dimension $d \ge 1$ over a field $k$
and let $D \subset X$ be an effective Cartier divisor. Let $n \ge 0$ be an integer.
Then there is a cycle class map
\begin{equation}\label{eqn:Intro-1-0}
cyc_{X|D} \colon  \{\CH^{n+d}(X|mD,n)\}_m \to \{K_n(X,mD)\}_m
\end{equation}
between pro-abelian groups. This map is covariant functorial
for proper morphisms, and contravariant functorial for flat morphisms of 
relative dimension zero. 
\end{thm}

For those interested in the precise variation in the modulus in the definition
of $cyc_{X|D}$, we actually prove that for every pair of integers $m \ge 1$ and $n\geq 0$, there exists a
cycle class map 
$ \CH^{n+d}(X|(n+1)mD,n) \to K_n(X,mD)$ such that going to pro-abelian groups,  we get 
the cycle class map of \thmref{thm:Intro-1}.
For a general divisor $D \subset X$, we do not expect that 
the cycle class map that we construct in \thmref{thm:Intro-1} will exist
without increasing the modulus. 
However, if we use rational coefficients, then the 
usage of pro-abelian groups can indeed be avoided, as
the following result shows.
In this paper, we use this improved version in the proof of 
\thmref{thm:Intro-2}.

\begin{thm}\label{thm:Intro-1-Q}
Let $X$ be a regular quasi-projective variety of pure dimension $d \ge 1$ over a field $k$
and let $D \subset X$ be an effective Cartier divisor. Let $n \ge 0$ be an integer.
Then there is a cycle class map
\begin{equation}\label{eqn:Intro-1-0-Q}
cyc_{X|D} \colon \CH^{n+d}(X|D,n)_{\Q} \to K_n(X,D)_{\Q}.
\end{equation}

This map is covariant functorial for proper morphisms, and contravariant
functorial for flat morphisms of relative dimension zero. 
Furthermore, it coincides with the map  ~\eqref{eqn:Intro-1-0} on the generators
of $\CH^{n+d}(X|D,n)$. 
\end{thm}

We now address as to why the cycle class maps of
Theorems~\ref{thm:Intro-1} and ~\ref{thm:Intro-1-Q} 
should be non-trivial and what we expect of these maps.
Recall that the relative $K$-theory $K_n(X,mD)$ has Adams operations
(e.g., see \cite{Levine-5} for their construction).
From our construction, we expect the map ~\eqref{eqn:Intro-1-0-Q} to be 
injective in the pro-setting, with image $\{K_n(X,mD)_{\Q}^{(d+n)}\}_m$.
Here, $K_n(X,mD)_{\Q}^{(d+n)}$ is $(d+n)$-th eigen-space of the
Adams operations.
When $D = \emptyset$, the cycle class map $cyc_X := cyc_{X|\emptyset}$ 
is not new and it
was constructed by Levine \cite{Levine-1} by a different method. 
He also showed that
in this special case, $cyc_X$ is indeed injective with
image $K_n(X)_{\Q}^{(d+n)}$.

When $X = \Spec(k)$ and $D = \emptyset$, the cycle 
class map $cyc_X$ coincides with Totaro's map 
$\CH^n(k,n) \to K^M_n(k) \to K_n(k)$ \cite{Totaro}.
Totaro showed that the map $\CH^n(k,n) \to K^M_n(k)$ is an 
isomorphism and one knows that the canonical map
$K^M_n(k)_{\Q} \to K_n(k)^{(n)}_{\Q}$ is an isomorphism. 
The remaining part of this paper is devoted to showing that
$cyc_{X|D}$ is in fact an isomorphism with integral coefficients
for the modulus pair $(\A^1_R, \{0\})$, where $R$ is 
a regular semi-local ring essentially of finite type over a 
characteristic zero field.

We make some further remarks on the past works on the cycle class map for 
0-cycles with modulus. Following Levine's strategy,
Binda \cite{Binda} showed that there is a cycle class map to relative $K$-theory
provided one makes the following changes: replace the higher Chow group with modulus by a variant of it
(which imposes a stronger version of the modulus condition,
originally introduced in \cite{KP-1}), assume that 
$D_{\rm red}$ is a strict normal
crossing divisor, and assume rational coefficients.
\thmref{thm:Intro-1} imposes none of these conditions.
If $D \subset X$ is a regular divisor, a cycle class map was
defined in \cite[Theorem~1.5]{KP*} using the stable $\A^1$-homotopy theory.


\vskip .3cm

We now describe our results about the cycle class map of \thmref{thm:Intro-1}
for the modulus pair $(\A^1_R, \{0\})$.
Recall that in case of the higher $K$-theory of a smooth
scheme $X$, the cycle-class map $\CH^{n+d}(X,n) \to K_n(X)$
from the 0-cycle group can not be expected to describe all of 
$K_n(X)$ (even with rational coefficients). However, 
we show in our next result that the cycle class map
of \thmref{thm:Intro-1} is indeed enough to describe all of 
the (integral) relative $K$-theory of nilpotent extensions of smooth schemes,
if we work in the category of pro-abelian groups instead of the
usual category of abelian groups. This demonstrates a remarkable 
feature of relative $K$-theory which is absent in the usual
$K$-theory.

Before we state the precise result, recall that the additive higher Chow groups
are special cases of higher Chow groups with modulus. More precisely,
for an equi-dimensional scheme $X$, the additive higher Chow group $\TCH^p(X,n+1;m)$ is same thing
as the higher Chow group with modulus $\CH^p(X \times \A^1_k| X \times {(m+1)}\{0\},n)$ for $m,n, p \ge 0$.
To understand the reason for the shift in the value of $n$, we need to recall that the additive higher
Chow groups are supposed to compute the relative $K$-theory of truncated polynomial extensions
and one knows that the connecting homomorphism $\partial \colon K_{n+1}({X[t]}/{(t^{m+1})}, (x)) \to
K_n(X \times \A^1_k, X \times (m+1)\{0\})$ is an isomorphism when $X$ is regular. Under this dichotomy,
we shall use the notation $cyc_{X}$ for $cyc_{\A^1_X|(X \times \{0\})}$ whenever we use the language of
additive higher Chow groups. In particular, for a ring $R$, we shall write $cyc_R$ for $cyc_{\A^1_R|\{0\}}$ 
while using
additive higher Chow groups. 

Let $R$ now be a regular semi-local ring which is essentially of finite type over a characteristic zero field.
Recall that there is a canonical map $K^M_*(R) \to K_*(R)$ from
the Milnor to the Quillen $K$-theory of $R$.
For $n \ge 1$, the group $\TCH^n(R, n;m)$ is not a 0-cycle group if $\dim(R) \ge 1$.
Hence, \thmref{thm:Intro-1} does not give us a cycle class map for this group. However, using this theorem
for fields and various other deductions, we can in fact prove an improved version of
\thmref{thm:Intro-1}. Namely, we can avoid the usage of pro-abelian groups for the existence of the cycle 
class map with integral coefficients.

\begin{thm}\label{thm:Intro-2}
  Let $R$ be a regular semi-local ring which is essentially of finite type over a characteristic zero
  field.
Let $m \ge 0$ and $n \ge 1$ be two integers. Then the following hold. 
\begin{enumerate}
\item
  There exists a cycle class map
  \[
    cyc^M_{R} \colon \TCH^n(R, n;m) \to K^M_n({R[x]}/{(x^{m+1})}, (x)).
  \]
\item
  The composite map
  \[
    cyc_R \colon
\TCH^n(R, n;m) \xrightarrow{cyc^M_R} K^M_n({R[x]}/{(x^{m+1})}, (x)) \to K_n({R[x]}/{(x^{m+1})}, (x))
  \]
coincides with the map of \thmref{thm:Intro-1} when $R$ is a field.
\item
$cyc^M_R$ and $cyc_R$ are natural in $R$.
\item
  $cyc^M_R$ is an isomorphism.
\item
  The map 
  \[
    cyc_R \colon \{\TCH^{n}(R,n;m)\}_m \to \{K_n({R[x]}/{(x^{m+1})}, (x))\}_m
  \]
  of pro-abelian groups is an isomorphism.
\end{enumerate}
\end{thm}

In other words, \thmref{thm:Intro-2} (5) says that the relative $K$-theory of truncated polynomial rings can 
indeed
be completely described by the relative 0-cycles over $R$ (the cycles in $\TCH^n(R,n;m)$ have relative
dimension zero over $R$). This shows that the additive Chow groups defined by Bloch-Esnault \cite{BE2} and
R{\"u}lling \cite{R} are indeed the relative $K$-groups, at least in characteristic zero.
This was perhaps the main target of the introduction of additive higher Chow groups by
Bloch and Esnault.

By the works of several authors (see \cite{EM} and \cite{Kerz09} for regular semi-local rings and
\cite{NS} and \cite{Totaro} for fields), it is now well known that the motivic cohomology of a
regular semi-local ring in the equal bi-degree (the Milnor range) coincides with its Milnor $K$-theory.
\thmref{thm:Intro-2} (4) says that this isomorphism
also holds for truncated polynomial rings over such rings.
This provides a concrete evidence that if one could extend Voevodsky's theory of motives to the theory of `non-$\A^1$-invariant' motives over so-called
fat points (infinitesimal extensions of spectra of fields), then
the underlying motivic cohomology groups must be the additive higher Chow groups
(see \cite{KP-5}).

\vskip .3cm

It should be remarked that the objective of \thmref{thm:Intro-2} is not to compute the relative $K$-groups.
There are already known computations of these by many authors (e.g., see
\cite{Good} and \cite{Hesselholt-tower}).
Instead, the above result addresses the question whether these relative (Milnor or Quillen)
$K$-groups could be described by additive 0-cycles.

\vskip .3cm

\thmref{thm:Intro-2} has following consequences. The first corollary below is in fact part of our
proof of \thmref{thm:Intro-2}.

\begin{cor}\label{cor:Mil-Qui}
Let $R$ be a regular semi-local ring which is essentially of finite type over a characteristic zero
field.
Let $n \ge 0$ be an integer. Then the canonical map
\[
\{K^M_n({R[x]}/{(x^{m})}, (x))\}_m   \to \{K_n({R[x]}/{(x^{m})}, (x))\}_m
\]
of pro-abelian groups is an isomorphism.
In particular, $\{K_n({R[x]}/{(x^{m})}, (x))^{(p)}\}_m = 0$ for $p \neq n$.
\end{cor}

\vskip .3cm

Let $R$ be any regular semi-local ring containing $\Q$. Then the N{\'e}ron-Popescu desingularization
theorem says that $R$ is a direct limit of regular semi-local rings $\{R_i\}$, where each
$R_i$ is essentially of finite type over $\Q$ (see \cite[Theorem~1.1]{Swan}).
One knows from \cite[Lemma~1.17]{R} that if each $\{K^M_n({R_i[x]}/{(x^{m})}, (x))\}_{n, m \ge 1}$ is a restricted 
Witt complex  over $R_i$, then $\{K^M_n({R[x]}/{(x^{m})}, (x))\}_{n, m \ge 1}$ is a restricted
Witt-complex over ${\underset{i \ge 1}\varinjlim} \ R_i = R$ (see \cite[Definition~1.14]{R} for the definition
of a restricted Witt-complex).
On the other hand, it was shown in \cite[Theorem~1.2]{KP-4} that each
collection $\{\TCH^n(R_i,n;m)\}_{n, m \ge 1}$ is a
restricted Witt-complex over $R_i$.
We therefore obtain our next consequence of \thmref{thm:Intro-2}.

\begin{cor}\label{cor:Milnow-Witt}
 Let $R$ be a regular semi-local ring containing $\Q$.
 Then the relative Milnor $K$-theory $\{K^M_n({R[x]}/{(x^{m})}, (x))\}_{n, m \ge 1}$ is a restricted 
Witt-complex  over $R$.
\end{cor}

In \cite[Chapter~II]{Bloch-IHES}, Bloch had shown (without using the terminology of Witt-complex)
that if $R$ is a regular local
ring containing a field of characteristic $p > 2$, then 
the subgroup of the relative Quillen $K$-theory of truncated polynomial rings over $R$,
generated by Milnor symbols (the symbolic $K$-theory
in the language of Bloch), has the structure of a restricted Witt-complex.
The above corollary extends the result of Bloch to characteristic zero.

\vskip .3cm

The last consequence of \thmref{thm:Intro-2} is the following.
Park and {\"U}nver \cite{PU} proposed a definition of motivic cohomology of truncated polynomial ring
${k[x]}/{(x^{m})}$ over a field. They showed that these motivic cohomology in the Milnor range coincide with
the Milnor $K$-theory of ${k[x]}/{(x^{m})}$ when $k$ is a characteristic zero field.
\thmref{thm:Intro-2} implies  that the Milnor range (relative) motivic cohomology
of Park-{\"U}nver coincides with the additive higher Chow groups.

\subsection{Comments and questions}\label{sec:Comments}
We make a couple of remarks related to the above results.

(1) Since \thmref{thm:Intro-1} is characteristic-free, one would expect the
same to be true for \thmref{thm:Intro-2} and
\corref{cor:Mil-Qui} as well. 
Our remark is that \thmref{thm:Intro-2} and
\corref{cor:Mil-Qui} are indeed true in all characteristics $\neq 2$.
Since the techniques of our proofs in positive characteristics 
are different from the present paper, they are presented
in \cite{GK}.

(2) Our second remark is actually a question. Recall that Chow groups with modulus are supposed to be the
motivic cohomology to describe the relative $K$-theory, just as Bloch's higher Chow groups describe
$K$-theory. Analogous to Bloch's Chow groups, the ones with modulus exist in all bi-degrees.
However, as we explained earlier,
\thmref{thm:Intro-2} says that the 0-cycles groups with modulus are often enough to describe all
of relative $K$-theory in the setting of pro-abelian groups. One can therefore ask the following.

\begin{ques}\label{ques:Future}
  Let $R$ be a regular semi-local ring essentially of finite type over a perfect field. 
Let $n, p \ge 1$ be two
  integers such that $n \neq p$. Is $\{\TCH^p(R,n;m)\}_m = 0$ ?
\end{ques}

Note that this question is consistent with the second part of
\corref{cor:Mil-Qui}.
Note also that it is already shown in \cite{KP-3} that the answer to this 
question is yes when $p > n$. So the open case is when $p < n$.
We hope to address this question in a future work.
Reader may recall that when $p < n/2$, the additive version of the
deeper Beilinson-Soul{\'e} vanishing conjecture
says that $\TCH^p(R,n;m)$ should vanish for every $m \ge 1$.

\subsection{An outline of the paper}\label{sec:Outline}
We end this section with a brief outline of the layout of this text.
In sections \ref{sec:Prelim} and ~\ref{sec:Milnor-K}, we set up our notations,
recollect the main objects of study and prove some
intermediate results. In \S~\ref{sec:CCM}, we define the cycle class map on the group of generators of
0-cycles. Our definition of the cycle class map is {\sl a priori} completely different from the
one in \cite{Binda} and \cite{Levine-1}. The novelty of the new construction is that it is very
explicit in nature and, therefore, it becomes possible to check that it factors through the 
rational equivalence. We also prove in this section that the cycle class map is natural for suitable proper and flat morphisms. 
One can check that this map does coincide with more abstractly defined maps of 
\cite{Binda} and \cite{Levine-1}
on generators. But we do not discuss this in this paper (see however \S~\ref{sec:Coincide}
for a sketch of this).

We break the proof of \thmref{thm:Intro-1} into two steps. In \S~\ref{sec:curves}, we prove it for a very 
specific type of curves using the results of \S~\ref{sec:Prelim}.
This is the technical part of the proof of \thmref{thm:Intro-1}.
It turns out that the general case can be reduced to the above
type using the results of \S~\ref{sec:Proj-F}.
This is done in \S~\ref{sec:Prf-1}. The idea that we have to increase the modulus for factoring
the cycle class map through the rational equivalence is already evident in the technical results of
\S~\ref{sec:Pre-Milnor}.

Sections~\ref{sec:char-0} and ~\ref{sec:Rel-Milnor**} 
constitute the heart of the proof of \thmref{thm:Intro-2}. In \S~\ref{sec:char-0}, we
provide some strong relations between the additive 0-cycles, relative Milnor $K$-theory and
the big de Rham-Witt complex. In particular, we show that it suffices to know the image of
certain very specific 0-cycles under the cycle class map in order to show that it factors through the
relative Milnor $K$-theory of a truncated polynomial ring (see \lemref{lem:Milnor-surj}).
In \S~\ref{sec:Rel-Milnor**}, we give an explicit description of the relative Milnor
$K$-theory in terms of the module of K{\"a}hler differentials (see \lemref{lem:M-Q-C-0}).
This allows us to establish the isomorphism between the 
additive higher Chow groups of 0-cycles and the relative Milnor $K$-theory.

To pass to the Quillen $K$-theory, we prove a vanishing theorem (see \propref{prop:Pro-vanish})
using some results of
\cite{Krishna-0}.
This allows us to show that the additive 0-cycle
groups for fields are isomorphic to the relative $K$-theory in the setting of pro-abelian groups.
In \S~\ref{sec:local}, we extend the results of \S~\ref{sec:char-0-*} to regular semi-local rings
using the main results of \cite{KP-2}. 
The last section is the appendix which contains some auxiliary results on the relation between 
Milnor and Quillen $K$-theory of fields. These results are used in the main proofs.

\section{The relative $K$-theory and cycles with modulus}\label{sec:Prelim}
In this section, we fix our notations and prove some basic results in relative
algebraic $K$-theory. We shall also recall the definition of the higher Chow groups
with modulus.

\subsection{Notations}\label{sec:Notations}
We shall in general work with schemes over an arbitrary base field $k$. We shall specify further 
conditions on $k$ as and when it is required. We let $\Sch_k$ denote the category of separated
finite type schemes over $k$. Recall that $X \in \Sch_k$ is called regular if $\sO_{X,x}$ is a
regular local ring for all points $x \in X$. We let $\Sm_k$ denote the full subcategory of $\Sch_k$
consisting of regular schemes. 
For $X, Y \in \Sch_k$, we shall denote the product $X \times_k Y$ simply by $X \times Y$.
For any point $x \in X$, we shall let $k(x)$ denote the residue field of $x$.
For a reduced scheme $X \in \Sch_k$, we shall let $X^N$ denote the normalization of $X$.
For $p \ge 0$, we shall denote the set of codimension $p$ points of a scheme $X$ by $X^{(p)}$.
For an affine scheme $X \in \Sch_k$, we shall let $k[X]$ denote the coordinate ring of $X$.

We shall let $\ov{\square}$ denote the projective space $\P^1_k = {\rm Proj}(k[Y_0, Y_1])$ and let
$\square = \ov{\square} \setminus \{1\}$. We shall let $\A^n_k  = \Spec(k[y_1, \ldots , y_n])$ be the
open subset of $\ov{\square}^n$, where $(y_1, \ldots, y_n)$ denotes the coordinate system of 
$\ov{\square}^n$ with $y_j = {Y^j_1}/{Y^j_0}$.  
Given a rational map $f \colon X \dashedrightarrow \ov{\square}^n$ in $\Sch_k$ and a
point $x \in X$ lying in the
domain of definition of $f$, we shall let
$f_i(x) = (y_i \circ f)(x)$, where $y_i \colon \ov{\square}^n \to \ov{\square}$ is the $i$-th projection.
For any $1 \le i \le n$ and $t \in \ov{\square}(k)$, we let $F^t_{n,i}$ denote the closed subscheme of 
$\ov{\square}^n$ given by $\{y_i = t\}$. We let $F^t_n = \stackrel{n}{\underset{i =1}\sum} \ F^t_{n,i}$.

By a closed pair $(X,D)$ in $\Sch_k$, we shall mean a closed immersion $D \inj X$ in $\Sch_k$,
where $X$ is reduced and $D$ is an effective Cartier divisor on $X$.
We shall write $X \setminus D$ as $X^o$.
We shall say that $(X,D)$ is a modulus pair if $X^o \in \Sm_k$.
If $(X,D)$ is a closed pair, we shall let $mD \subset X$ be the closed subscheme defined by the sheaf of
ideals $\sI^m_D$, where $D$ is defined by the sheaf of ideals $\sI_D$.

All rings in this text will be commutative and Noetherian. For such a ring $R$ and an
integer $m \ge 0$, we shall let $R_m = {R[t]}/{(t^{m+1})}$ denote the truncated polynomial 
algebra over $R$. We shall write $\Spec(R[t_1, \ldots , t_n])$ as $\A^n_R$.
The tensor product $M \otimes_{\Z} N$ will be denoted simply as $M \otimes N$.
Tensor products over other bases will be explicitly indicated.

\subsection{The category of pro-objects}\label{sec:Pro}
By a pro-object in a category $\sC$, we shall mean a sequence of objects $\{A_m\}_{m \ge 0}$
together with a map $\alpha^A_m \colon A_{m+1} \to A_m$ for each $m \ge 0$. We shall write this object
often as $\{A_m\}$. We let ${\rm Pro}(\sC)$ denote the category of pro-objects in $\sC$ with the
morphism set given by
\begin{equation}\label{eqn:Pro-0}
\Hom_{{\rm pro}(\sC)}(\{A_m\}, \{B_m\}) = {\underset{n}\varprojlim} \ 
{\underset{m}\varinjlim}\ \Hom_{\sC}(A_m, B_n).
\end{equation}
   
In particular,  giving a morphism $f$ as above is equivalent to finding a function  
$\lambda: \N \to \N$, a map $f_n: A_{\lambda(n)} \to B_n$ for each $n \ge 0$ such that 
for each $n' \ge n$, there exists $l \ge \lambda(n), \lambda(n')$ so that the diagram
\begin{equation}\label{eqn:Pro-1}
\xymatrix@C.8pc{
A_l \ar[r] \ar[dr] & A_{\lambda(n')} \ar[r]^-{f_{n'}} & B_{n'} \ar[d] \\
& A_{\lambda(n)} \ar[r]^-{f_n} & B_n}
\end{equation}
is commutative, where the unmarked arrows are the structure maps of $\{A_m\}$ and $\{B_m\}$.
We shall say that $f$ is strict if $\lambda$ is the identity function.
If $\sC$ admits all sequential limits, we shall denote the limit of $\{A_m\}$ by 
${\underset{m}\varprojlim} \ A_m \in \sC$. If $\sC$ is an abelian category, then
so is ${\rm Pro}(\sC)$. We refer the reader to \cite[Appendix~4]{AM} for further details about
${\rm Pro}(\sC)$.

\subsection{The relative algebraic $K$-theory}\label{sec:Rel-K}
Given a closed pair $(X,D)$ in $\Sch_k$, we let $K(X,D)$ be the homotopy fiber of the
restriction map between the Thomason-Trobaugh non-connective algebraic $K$-theory spectra $K(X) \to K(D)$.
We shall let $K_i(X)$ denote the homotopy groups of $K(X)$ for $i \in \Z$.
We similarly define $K_i(X,D)$.
We shall let $K^D(X)$ denote the homotopy fiber of the restriction map
$K(X) \to K(X \setminus D)$. Note that $K^D(X)$ does not depend on the subscheme structure of
$D$ but $K(X,D)$ does.
Note also that if $D' \subset X$ is another  closed subscheme such that $D \cap D' = \emptyset$,
then $K^D(X)$ is canonically homotopy equivalent to the
homotopy fiber $K^D(X,D')$ of the restriction map $K(X,D') \to K(X \setminus D, D')$.

If $(X,D)$ is a closed pair, we have the canonical 
restriction map $K(X, (m+1)D) \to K(X,mD)$.
In particular, this gives rise a pro-spectrum $\{K(X,mD)\}$ and a level-wise homotopy fiber
sequence of pro-spectra
\begin{equation}\label{eqn:Pro-spectra}
\{K(X,mD)\} \to K(X) \to \{K(mD)\}.
\end{equation}

If $X = \Spec(R)$ is affine and $D = V(I)$, we shall often write $K(X,mD)$ as $K(R, I^m)$ and
$K(X)$ as $K(R)$. For a ring $R$, we shall let $\wt{K}(R_m)$ denote the reduced
$K$-theory of $R_m$, defined as the
homotopy fiber of the augmentation map $K(R_m) \to K(R)$. Observe that there exists a
canonical decomposition $K(R_m) \cong \wt{K}(R_m) \times K(R)$.

Suppose that $R$ is a regular semi-local ring. Let
$f(t) \in R[t]$ be a polynomial such that 
$f(0) \in R^{\times}$ and let $Z = V((f(t))) \subset \A^1_R$
be the closed subscheme defined by $f(t)$. Since $Z \cap \{0\} = \emptyset$, the composite
map $K^Z(\A^1_R) \to K(\A^1_R) \to K((m+1)\{0\})$ is null-homotopic for all $m \ge 0$.
Hence, there is a factorization $K^Z(\A^1_R) \to K(\A^1_R, (m+1)\{0\}) \to K(\A^1_R)$.
Let $[\sO_Z]$ denote the fundamental class of $Z$ in $K^Z_0(\A^1_R)$
(see \cite[Exercise~5.7]{TT}). Note that $Z$ may not be reduced or irreducible.
Let $\alpha_Z$ denote the image of $[\sO_Z]$ under the map
$K^Z_0(\A^1_R) \to K_0(\A^1_R, (m+1)\{0\})$.
Let $\partial_n \colon \wt{K}_n(R_m) \to K_{n-1}(\A^1_R, (m+1)\{0\})$ denote the connecting homomorphism
obtained by considering the long exact homotopy groups sequence associated to 
~\eqref{eqn:Pro-spectra}. The homotopy invariance of $K$-theory on $\Sm_k$ shows that this
map is an isomorphism. For $g(t) \in R[t]$, let $\ov{g(t)}$ denote its image in $R_m$.

\begin{lem}\label{lem:Elem-0}
Given $Z = V((f(t)))$ as above, we have
\[
\alpha_Z = \partial_1(\ov{(f(0))^{-1} f(t)}).
\]
\end{lem}
\begin{proof}
  Since $f(0) \in R^{\times}$, we note that $Z = V((f(0))^{-1} f(t))$.
  We let $g(t) =  (f(0))^{-1} f(t)$
so that $g(0) = 1$ and therefore $\ov{g(t)} \in \wt{K}_1(R_m)$.
We let $\Lambda = \{(a,b) \in R[t] \times R[t]| a - b \in (t^{m+1})\}$
be the double of $R[t]$ along
the ideal $(t^{m+1})$ as in \cite[Chapter~2]{Milnor}. Let $p_1: \Lambda \to R[t]$
be the first projection.
Then recall from \cite[Chapter~6]{Milnor} that $K_0(R[t], (t^{m+1})) \cong
{\rm Ker}((p_1)_{\#}: K_0(\Lambda) \to K_0(R[t]))$ and \cite[Chapter~3]{Milnor} shows
that $\partial_1(u) = [M(u)] - [\Lambda] \in K_0(R[t], (t^{m+1}))$, where
$M(u)$ is the rank one projective $\Lambda$-module given by
$M(u) = \{(x,y) \in R[t] \times R[t]| u \ov{x} = \ov{y} \ \mbox{in} \ R_m\}$ for any $u \in R^{\times}_m$.

Let $u =\ov{g(t)}$ and let $M= M(\ov{g(t)})$.  Let $p_2:M \to  {R[t]}/{(g(t))} $ denote the composition of the second projection $M \to R[t]$ with the surjection $R[t] \surj R[t]/(g(t))$.  It is then easy to see that the sequence
\[
0 \to \Lambda \xrightarrow{\theta} M \xrightarrow{p_2} {R[t]}/{(g(t))} \to 0
\]
is a short exact sequence of $\Lambda$-modules if we let
$\theta((a,b)) = (a, b g(t)) \in M$.
In particular, we get $[\sO_Z] = [V((g(t)))] = [M] - [\Lambda] = \partial_1(\ov{g(t)})$.
This proves the lemma. 
\end{proof}

\subsection{The projection formula for relative $K$-theory}\label{sec:Proj-F}
Let $(X,D)$ be a modulus pair in $\Sch_k$ and let $S_X$ be the double of $X$ along $D$.
Recall from \cite[\S~2.1]{BK} that $S_X$ is the pushout $X \amalg_D X$ of the diagram of
schemes $X \hookleftarrow D \inj X$. On each affine open subset $U \subset X$,
the double is the spectrum $S_U$ of the ring 
$\{(a, b) \in \sO_U(U) \times \sO_U(U) | a-b \in \sI_D(U)\}$, where $\sI_D \subset \sO_X$ is the ideal sheaf of $D$. 
We have two inclusions $\iota_\pm \colon X \inj S_X$ and a projection $p \colon S_X \to X$
such that $p \circ \iota_\pm = {\rm id}_X$.
In particular, there is a canonical decomposition $K(S_X) \cong
K(S_X, X_-) \times K(X)$.
There is an inclusion of modulus pairs $(X, D) \inj (S_X, X_-)$, with respect to the 
embedding $X_+ \inj S_X$. This yields the pull-back map
$\iota^*_+ \colon K(S_X, X_-) \to K(X,D)$. 

We now let $u \colon Z \inj X$ be a closed immersion such that $Z \cap D = \emptyset$.
This gives rise to a closed embedding $Z \inj X \xrightarrow{\iota_+} S_X$ such that
$Z \cap D = Z \cap X_- = \emptyset$.
Since $Z \cap D = \emptyset$, the push-forward map (which exists because
$Z \subset X_{\rm reg}$) $u_* \colon K(Z) \to K(X)$ composed with the restriction $K(X) \to K(D)$ is
null-homotopic. Hence, there is a canonical factorization
$K(Z) \to K(X,D) \to K(X)$ of the push-forward map.
We shall denote the map $K(Z) \to K(X,D)$ also by $u_*$. It is clearly functorial in $(X,D)$
and $Z$. Recall also that $K(Z)$ and $K(X,D)$ are module spectra over the ring spectrum $K(X)$
(e.g., see \cite[Chapter~3]{TT}). We shall need to know the following result about
the map $u_*$ in the proof of \lemref{lem:Factor}.

\begin{lem}\label{lem:Proj-rel}
  The push-forward map $u_* \colon K_*(Z) \to K_*(X,D)$ is $K_*(X)$-linear.
\end{lem}
\begin{proof}
  Since $Z \subset S_X \setminus X_-$, we also have the push-forward map
  $v_* \colon K(Z) \to K(S_X, X_-)$, where we let $v = \iota_+ \circ u$.
  Suppose we know that $u_* = \iota^*_+ \circ v_* \colon K(Z) \xrightarrow{v_*} K(S_X, X_-) 
  \xrightarrow{\iota^*_+} K(X,D)$ and the lemma holds for $v_*$.
  Then for any $\alpha \in K_*(X)$ and $\beta \in K_*(Z)$,
  we get
  \[
    u_*(u^*(\alpha) \beta) = \iota^*_+(v_*(v^*p^*(\alpha) \beta)) =
    \iota^*_+(p^*(\alpha) v_*(\beta)) = (p \circ \iota_+)^*(\alpha) u_*(\beta) = \alpha u_*(\beta).
    \]
    We thus have to show the following.
    \begin{enumerate}
    \item
      The lemma holds for the inclusion $Z \inj S_X$, and
    \item
      $u_* = \iota^*_+ \circ v_*$.
    \end{enumerate}
    
    To prove (1), we can use that the map $K_*(S_X, X_-) \to K_*(S_X)$ is a split inclusion
    (as we saw above). Using this and the fact that $K(S_X, X_-) \to K(S_X)$ is
    $K(S_X)$-linear, it suffices to prove (1) for the composite push-forward map
    $v_* \colon K(Z) \to K(S_X)$. But we already saw above that $K(Z)$ is a module spectrum over $K(S_X)$.
    
    We now prove (2).
    By the definition of the push-forward maps to the relative $K$-theory, we have
    factorizations
    \begin{equation}\label{eqn:Proj-rel-0}
      \xymatrix@C.8pc{
        K(Z) \ar[r] \ar@{=}[d] & K^Z(S_X, X_-) \ar[r] \ar[d]^-{\iota^*_+} &
        K(S_X, X_-) \ar[d]^-{\iota^*_+} \\
        K(Z) \ar[r] & K^Z(X, D) \ar[r] & K(X, D),}
    \end{equation}
    such that the square on the right is commutative and the top (resp. bottom) composite arrow
    is $v_*$ (resp. $u_*$). Hence, it suffices to show that the left square is commutative.

    For showing this, we use the diagram
\begin{equation}\label{eqn:Proj-rel-1}
      \xymatrix@C.8pc{   
K(Z) \ar[r] \ar@{=}[d] & K^Z(S_X, X_-) \ar[r]^-{\cong} \ar[d]^-{\iota^*_+} &
        K^Z(S_X \setminus X_-) \ar[d]^-{\iota^*_+} \\
        K(Z) \ar[r] & K^Z(X, D) \ar[r]^-{\cong} & K^Z(X \setminus D),}
    \end{equation}
    where the horizontal arrows on the right are the restriction maps induced by the
    open immersions of modulus pairs. In particular, the square on the right is
    commutative. The right horizontal arrows are homotopy equivalences by the
    excision theorem. Hence, it suffices to show that the composite square in
    ~\eqref{eqn:Proj-rel-1} commutes.

    To see this, we note that the composite horizontal arrows in
    ~\eqref{eqn:Proj-rel-1} have the factorizations:
    
\begin{equation}\label{eqn:Proj-rel-2}
      \xymatrix@C.8pc{  
K(Z) \ar[r] \ar@{=}[d] & G(Z) \ar[r] \ar@{=}[d] & K^Z(S_X \setminus X_-) \ar[d]^-{\iota^*_+}_-{\cong} \\
K(Z) \ar[r] & G(Z) \ar[r]  & K^Z(X \setminus D),}
\end{equation}
where $G(Z)$ is the $K$-theory of pseudo-coherent complexes on $Z$
(\cite[Chapter~3]{TT}) and $K(Z) \to G(Z)$ is the canonical map.
We are now done because the square on the right in ~\eqref{eqn:Proj-rel-2} clearly commutes.
\end{proof}

\subsection{The 0-cycles with modulus}\label{sec:HCGM}
Let $k$ be a field and let $(X,D)$ be an equi-dimensional closed pair 
in $\Sch_k$ of dimension $d \ge 1$. We recall the definition of
the higher Chow groups with modulus from \cite{BS} or \cite{KP}. For any integers 
$n, p \ge 0$, we let
$\un{z}^p(X|D, n)$ denote the free abelian group on the set of integral closed 
subschemes of $X \times \square^n$ of codimension $p$ satisfying the following.
\begin{enumerate}
\item
$Z$ intersects $X \times F$ properly for each face $F \subset \square^n$.
\item
If $\ov{Z}$ is the closure of $Z$ in $X \times \ov{\square}^n$ and 
$\nu: \ov{Z}^N \to X \times \ov{\square}^n$
is the canonical map from the normalization of $\ov{Z}$, then the inequality 
(called the modulus condition)
\[
\nu^*(D \times \ov{\square}^n) \le \nu^*(X \times F^1_n)
\] 
holds in the set of Weil divisors on $\ov{Z}^N$.
\end{enumerate}

An element of the group $\un{z}^p(X|D, n)$ will be called an admissible cycle.
It is known that $\{n \mapsto \un{z}^p(X|D, n)\}$ is a cubical abelian group (see \cite[\S~1]{KLevine}). 
We denote this by $\un{z}^p(X|D, *)$.
We let $z^p(X|D, *) = \frac{\un{z}^p(X|D, *)}{\un{z}^p_{\rm degn}(X|D, *)}$, where
$\un{z}^p_{\rm degn}(X|D, *)$ is the degenerate part of the cubical abelian group 
$\un{z}^p(X|D, *)$. For $n \ge 0$, we let 
\[
\CH^p(X|D, n) = H_n(z^p(X|D, *))
\]
and call them the higher Chow groups with modulus of $(X,D)$.
The direct sum 
\begin{equation}\label{eqn:0-cycles}
\CH_0(X|D, *): = {\underset{n \ge 0}\oplus} \ \CH^{d+n}(X|D, n) = 
{\underset{n \ge 0}\oplus} \ \CH_{-n}(X|D, n)
\end{equation}
is called the {\sl higher Chow group of 0-cycles with modulus}. 
The subject of this paper is
to study the relation between $\CH_0(X|D, *)$ and the relative $K$-theory $K_*(X,D)$.

We recall for the reader that the groups $\CH^p(X|D, *)$ satisfy the flat pull-back and 
the proper push-forward 
properties under certain conditions. We refer the reader to \cite{BS} or \cite{KP} 
for these and other properties of the Chow groups with modulus.

\section{The Milnor $K$-theory}\label{sec:Milnor-K}
Recall that for a semi-local ring $R$, the Milnor $K$-group $K^M_i(R)$ is defined to be 
the $i$-th graded piece of the graded Milnor $K$-theory $\Z$-algebra $K^M_*(R)$.
The latter is
defined to be the quotient of the tensor algebra $T_*(R^{\times})$ by the two-sided graded 
ideal generated by homogeneous elements $\{a \otimes (1-a)| a, 1-a \in R^{\times}\}$.
The image of an element $a_1 \otimes \cdots \otimes a_n \in T_n(R^{\times})$ in $K^M_n(R)$
is denoted by the Milnor symbol $\un{a} = \{a_1,  \ldots , a_n\}$.
If $I \subset R$ is an ideal, the relative Milnor $K$-theory $K^M_i(R,I)$ is defined to be
the kernel of the natural surjection $K^M_n(R) \to K^M_n(R/I)$.
It follows from \cite[Lemma~1.3.1]{Kato-Saito} that $K^M_n(R,I)$ is generated by Milnor symbols
$\{a_1, \ldots, a_n\}$, where $a_i \in {\rm Ker}(R^{\times} \to (R/I)^{\times})$ for some $1 \le i \le n$,
provided $R$ is a finite product of local rings.

The product structures on the Milnor and Quillen $K$-theories yield 
a natural graded ring homomorphism $\psi_R \colon K^M_*(R) \to K_*(R)$.
If $I \subset R$ is an ideal, we have a natural  
isomorphism $K^M_1(R,I) \cong \wh{K}_1(R,I)$, where
$\wh{K}_*(R,I)$ is the group ${\rm Ker}(K_*(R) \to K_*(R/I))$. Using the module structure on
$\wh{K}_*(R,I)$ over $K_*(R)$ and the ring homomorphism $K^M_*(R) \to K_*(R)$, we obtain a natural
graded $K^M_*(R)$-linear map $\psi_{R|I} \colon K^M_*(R,I) \to \wh{K}_*(R,I)$.
The cup product on Milnor $K$-theory yields maps $K^M_n(R) \otimes K^M_{n'}(R,I) \to K^M_{n+n'}(R,I)$.
In the sequel, we shall loosely denote the image of this map also by
$K^M_n(R)K^M_{n'}(R,I)$ (e.g, see \lemref{lem:Milnor-0}).

\subsection{The improved Milnor $K$-theory}\label{sec:IMK}
If $R$ is a semi-local ring whose residue fields are not infinite, 
then the Milnor $K$-theory $K^M_*(R)$ does
not have good properties. For example, the Gersten conjecture does not hold even if $R$ is a regular
local ring containing a field.
If $R$ is a finite product of local rings containing a field, Kerz \cite{kerz10} defined an improved 
version of Milnor $K$-theory,
which is denoted as $\wh{K}^M_*(R)$. This is a graded commutative ring
and there is natural map of graded commutative rings $\eta^R: K^M_*(R) \to \wh{K}^M_*(R)$.
For an ideal $I \subset R$, we let $\wh{K}^M_*(R,I) = {\rm Ker}(\wh{K}^M_*(R) \to \wh{K}^M_*(R/I))$.
We thus have a natural map $K^M_*(R,I) \to \wh{K}^M_*(R,I)$.
We state some basic facts about $\wh{K}^M_*(R)$  in the following result and
refer the reader to \cite{kerz10} for proofs.

\begin{prop}\label{prop:Kerz-finite}
Let $R$ be a finite product of local rings containing a field. Then the map
$\eta^R: K^M_*(R) \to \wh{K}^M_*(R)$ has following properties.
\begin{enumerate}
\item
$\eta^R$ is surjective.
\item
$\eta^R_n$ is an isomorphism for all $n \ge 0$ if $R$ is a field.
\item
$\eta^R_n$ is an isomorphism for $n \le 1$.
\item
$\eta^R_n$ is an isomorphism for all $n$ if each residue fields of $R$ are infinite.
\item
The natural map $K^M_n(R) \to K_n(R)$ factors through $\eta^R_n$.
\item
The map $\wh{K}^M_2(R) \to K_2(R)$ is an isomorphism.
\item
The Gersten conjecture holds for $\wh{K}^M_n(R)$.
\end{enumerate}
\end{prop}

We now let $R$ be a regular semi-local ring (not necessarily a product of local rings) 
containing a field. Let $F$ denote the total quotient ring (a product of fields) of $R$.
Recall from \cite[\S~1]{Kato} that there is a (Gersten) complex of abelian groups
\begin{equation}\label{eqn:Gersten}
  K^M_n(F) \to {\underset{{\rm ht}(\fp) = 1}\oplus} \ K^M_{n-1}(k(\fp)) \to \cdots \to
{\underset{{\rm ht}(\fp) = n-1}\oplus} \ K^M_{1}(k(\fp)) \to 
{\underset{{\rm ht}(\fp) = n}\oplus} 
\ K^M_{0}(k(\fp)).
\end{equation}

We let $\wh{K}^M_n(R)$ denote the kernel
of the boundary map
\[
 \partial: K^M_n(F) \to 
 {\underset{{\rm ht}(\fp) = 1}\oplus} \ K^M_{n-1}(k(\fp))
\]
in ~\eqref{eqn:Gersten}.
For any $X \in \Sch_k$, the  improved Milnor $K$-theory Zariski sheaf $\wh{\sK}^M_{n,X}$
was defined in \cite{kerz10} whose stalk at a point $x \in X$ is $\wh{K}^M_n(\sO_{X,x})$.
As ~\eqref{eqn:Gersten} gives rise to a resolution of $\wh{K}^M_n(R_\fp)$ for every prime ideal
$\fp \subset R$ by \propref{prop:Kerz-finite} (7), it follows that  $\wh{K}^M_n(R)$ coincides with the
group of global sections
of the sheaf $\wh{\sK}^M_{n,X}$ on $X = \Spec(R)$.

Since the composite map $K^M_n(R) \to K^M_n(F) \to K^M_{n-1}(k(\fp))$ is well known to be zero
for every height one prime ideal $\fp \subset R$, it follows from the definition of
$\wh{K}^M_n(R)$ and the Gersten resolution of Quillen $K$-theory that 
there are natural maps
\begin{equation}\label{eqn:Mil-Q}
  K^M_n(R) \to \wh{K}^M_n(R) \xrightarrow{\psi_R} K_n(R).
\end{equation}

Suppose now that $R$ is a regular semi-local integral domain of dimension one containing 
a field and 
$I \subset R$ is an ideal of height one. Then $R/I$ is a finite product of 
Artinian local rings.
In particular, the improved Milnor $K$-theory $\wh{K}^M_*(R/I)$ is defined.
We can write $R/I = \stackrel{r}{\underset{i =1}\prod} {R_{\fm_i}}/{IR_{\fm_i}}$,
where $\fm_1, \ldots , \fm_r$ are the minimal primes of $I$.
We thus have the canonical maps
\begin{equation}\label{eqn:Milnor-semi-local}
\wh{K}^M_n(R) \inj \stackrel{r}{\underset{i =1}\prod} \wh{K}^M_n(R_{\fm_i}) \surj
\stackrel{r}{\underset{i =1}\prod} \wh{K}^M_n({R_{\fm_i}}/{IR_{\fm_i}}) \xleftarrow{\cong}
\wh{K}^M_n(R/I),
\end{equation}
where the first arrow is induced from the definition of $\wh{K}^M_n(R)$ and the 
Gersten resolutions of the improved Milnor $K$-theory of the localizations  of $R$.
We define the relative improved Milnor $K$-group $\wh{K}^M_n(R,I)$ as the kernel of
the composite map. Note that this agrees with the relative improved Milnor 
$K$-groups defined earlier if $R$ is a product of local rings.

Note that ~\eqref{eqn:Milnor-semi-local} also shows that the diagram
\begin{equation}\label{eqn:Milnor-semi-local-0}
\xymatrix@C.8pc{
K^M_n(R) \ar[r] \ar[d] & \wh{K}^M_n(R) \ar[r] \ar[d] & K_n(R) \ar[d] \\
K^M_n(R/I) \ar[r] & \wh{K}^M_n(R/I) \ar[r] & K_n(R/I)}
\end{equation}
commutes.
We therefore get the 
canonical maps of relative $K$-theories
\begin{equation}\label{eqn:Milnor-semi-local-1}
K^M_*(R,I) \to \wh{K}^M_*(R,I) \to \wh{K}_*(R,I),
\end{equation}
where recall that $\wh{K}_*(R,I) = {\rm Ker}(K_*(R) \to K_*(R/I))$.

\subsection{Some results on Milnor-$K$-theory}\label{sec:Pre-Milnor}
We shall need few results on the Milnor $K$-theory of discrete valuation rings.
For a discrete valuation ring $R$ with field of fractions $F$, we shall let 
$\ord: F^{\times} \to \Z$ denote the valuation map. 
We begin with the following elementary computation in Milnor $K$-theory.
We shall use the additive notation for the group operation  of the Milnor $K$-theory. 

\begin{lem}\label{lem:Elem-Milnor}
  Let $R$ be a semi-local integral domain with field of fractions $F$.
  Let $a,b, s, t$ be non-zero elements of $R$ such that $1 + as, 1 + bt \neq 0$.
  Then we have the following identity in $K^M_2(F)$.
\begin{equation}\label{eqn:Milnor-1}
  \{1+as, 1+ bt\} =
  \left\{
    \begin{array}{ll}
      - \{1 + \frac{ab}{1 + as} st, - as(1 + bt)\} & \mbox{if $1 + (1+bt)as \neq 0$} \\
      0, & \mbox{otherwise}.
    \end{array}
    \right.
\end{equation}
\end{lem}
\begin{proof}
Suppose first that $1 + (1 + bt)as = 0$. Then we have
\[
  \{1+as, 1+ bt\} = \{1 + as, (-as)^{-1}\} = - \{1 + as, -as\} = 0.
  \]

Otherwise, we write
\begin{equation}\label{eqn:Milnor-2}
  \begin{array}{lll}
    \{1 + \frac{ab}{1 + as} st, - as(1 + bt)\} & = &
\{1 + \frac{ab}{1 + as} st, - as\} + \{1 + \frac{ab}{1 + as} st, (1 + bt)\} \\
& = & \{1 + \frac{ab}{1 + as} st, - as\} + \{1 + (1+ bt)as, 1 + bt\} \\
& & - \{1 + as, 1 + bt\} \\
& = &  \{1 + (1 + bt)as, - as\} - \{1 + as, -as\} \\
&  &  + \{1 + (1 + bt)as, 1 + bt\} - \{1 + as, 1 + bt\} \\
& = & \{1 + (1 + bt)as,  -(1 + bt)as\} - \{1 + as, -as\} \\
&   &  - \{1 + as, 1 + bt\} \\
& = & \{1 - u, u\} - \{1 - v, v\} - \{1 + as, 1 + bt\} \\
& = & - \{1 + as, 1 + bt\},
  \end{array}
\end{equation}
where we let $u = - (1 + bt)as$ and $v = -as$.
This proves the lemma.
\end{proof}

\begin{lem}\label{lem:Milnor-0}
Let $R$ be a discrete valuation ring with maximal ideal $\fm$ and field of fractions $F$.
For $m, n \ge 1$, let $K^M_n(F, m)$ denote the subgroup of $K^M_n(F)$ generated by
Milnor symbols $\{y_1, \ldots , y_n\} \in K^M_n(F)$ such that $\stackrel{n}{\underset{i =1}\sum} 
{\ord}(y_i - 1) \ge m$. Then for any $n \ge 0$, we have
\[
K^M_{n+1}(F,m) \subseteq (1 + \fm^m)K^M_n(F).
\]
\end{lem}
\begin{proof}
Note that for $n = 0$, we actually have $K^M_1(F, m) = 1 + \fm^m$ and this is obvious from the
definition of $K^M_1(F, m)$.
We shall prove $n \ge 1$ case by induction on $n$. We let $\pi$ denote a uniformizing parameter of $R$.
We can write $y_i = 1 + u_i \pi^{m_i}$ for some $u_i \in R^{\times}$ and $m_i \in \Z$ for $1 \le i \le n$.
We first observe that if $m_i \ge m$ for some $i \ge 1$, then $y_i = 1+ u_i\pi^{m_i} \in 1 + \fm^m$ and 
we are done.

We now assume that $n =1$. In this case, if some $m_i \le 0$, then we must have that some $m_j \ge m$ 
and we are done as above. We can therefore assume that $0 < m_1, m_2 < m$.
In this case, \lemref{lem:Elem-Milnor} says that $\{y_1, y_2\}$ is either zero or it is
$- \{1 + u_1u_2 y_1^{-1} \pi^{m_1 + m_2}, - u_1 y_2 \pi^{m_1}\}$.
Since $m_1 \ge 0$, we see that $y^{-1}_1 \in R^{\times}$. In particular,
$1 + u_1 u_2 y^{-1}_1 \pi^{m_1 + m_2} \in 1 + \fm^{m_1 + m_2}$. We therefore get
$\{y_1, y_2\} \in (1 + \fm^m)K^M_1(F)$.

If $n \ge 2$, we have must have $m_i \ge 0$ for some $1 \le i \le n$ as $m \ge 0$.
Since the permutation of coordinates of a Milnor symbol only changes its sign in the Milnor $K$-group,
we can assume that $m_1 \ge 0$ so that $y_1 \in R^{\times}$.
We can now write $\{y_1, \ldots , y_n\} = \{y_1, y_2\} \cdot \{y_3, \ldots , y_n\}$.
We have seen before that the term $\{y_1, y_2\}$ is either zero or we have
\[
\begin{array}{lll}
\{y_1, \ldots , y_n\} & = & 
\{1 + u_1\pi^{m_1}, 1 + u_2\pi^{m_2}\} \cdot \{y_3, \ldots, y_n\} \\
& = & - \{1 + u_1u_2 y_1^{-1} \pi^{m_1 + m_2}, - u_1 \pi^{m_1} y_2\}\cdot \{y_3, \ldots , y_n\} \\
& = & \{- u_1 y_2 \pi^{m_1}\} \cdot \{1 + u_1u_2 y_1^{-1} \pi^{m_1 + m_2}, y_3, \ldots, y_n\}.
\end{array}
\]
Since $y_1 \in R^{\times}$, it follows that $y'_2:= 1 + u_1u_2 y_1^{-1} \pi^{m_1 + m_2} \in 1 + \fm^{m_1 + m_2}$.
In particular, we see that $\ord(y'_2 - 1) + \stackrel{n}{\underset{i =3}\sum} \ord(y_i -1) \ge 
\stackrel{n}{\underset{i = 1}\sum} m_i \ge m$. Hence, the induction hypothesis implies that
$\{y'_2, y_3, \ldots , y_n\} \in (1 + \fm^m)K^M_{n-1}(F)$.
This implies that
$\{y_1, \ldots , y_n\} = \{- u_1 y_2 \pi^{m_1}\} \cdot \{y'_2, y_3, \ldots , y_n\} 
\in (1 + \fm^m)K^M_n(F)$.
This finishes the proof.
\end{proof}

\begin{lem}\label{lem:Milnor-1}
Let $R$ be a discrete valuation ring containing a field. Let $\fm$ and $F$ denote the
maximal ideal and the field of fractions of $R$, respectively. Then the following 
hold for every integer $n \ge 0$.
\begin{enumerate}
\item
$(1 + \fm)K^M_n(F) \subseteq \wh{K}^M_{n+1}(R)$.
\item
$(1 + \fm^{m +n})K^M_n(F) \subseteq (1 + \fm^m)\wh{K}^M_n(R)$ for all $m \ge 1$.
\end{enumerate}
\end{lem}
\begin{proof}
We shall prove the lemma by induction on $n$. As the base case $n = 0$ trivially follows, 
we shall assume
that $n \ge 1$.
Suppose we show that
\begin{equation}\label{eqn:Milnor-1-0}
(1 + \fm)K^M_1(F) \subseteq (1 + \fm) \wh{K}^M_1(R) \subset \wh{K}^M_2(R).
\end{equation} 
We will then have 
\[
(1 + \fm)K^M_n(F) \subseteq (1 + \fm) \wh{K}^M_1(R)K^M_{n-1}(F) =
\wh{K}^M_1(R) (1 + \fm)K^M_{n-1}(F)  \hspace*{4cm}
\]
\[
\hspace*{9cm}
{\subseteq}^{1}  \wh{K}^M_1(R) \wh{K}^M_{n}(R) \subseteq 
\wh{K}^M_{n+1}(R),
\]
where ${\subseteq}^1$ holds by induction on $n$. This will prove (1).
 
Similarly, suppose we show for all $m \ge 1$ that
\begin{equation}\label{eqn:Milnor-1-1}
(1 + \fm^{m+1})K^M_1(F) \subseteq (1 + \fm^{m}) \wh{K}^M_1(R) \subset \wh{K}^M_2(R).
\end{equation}
Then for any $m \ge 1$ and $n \ge 2$, we will have
\[
(1 + \fm^{m +n})K^M_n(F) \subseteq  (1 + \fm^{m+n-1}) \wh{K}^M_1(R) K^M_{n-1}(F) =
\wh{K}^M_1(R) (1 + \fm^{m+n-1}) K^M_{n-1}(F) \hspace*{4cm}
\]
\[  
\hspace*{6cm}
{\subseteq}^{1} \wh{K}^M_1(R) (1 + \fm^m) \wh{K}^M_{n-1}(R) \subseteq  
(1 + \fm^m)   \wh{K}^M_{n}(R),
\]
where ${\subseteq}^1$ holds by induction on $n$. This will prove (2).
We are therefore left with showing ~\eqref{eqn:Milnor-1-0} and ~\eqref{eqn:Milnor-1-1} 
in order to prove the lemma.

To prove ~\eqref{eqn:Milnor-1-0}, we let $\pi$ be a uniformizing parameter of $R$.
For $j \in \Z$ and $u, v \in R^{\times}$, we then have 
\[
\begin{array}{lll}
\{1 + u\pi, v\pi^j\} & = & \{1 + u\pi, v\} + \{1 + u\pi, \pi^j\} \\
& = & \{1 + u\pi, v\} + j\{1 + u\pi, \pi\} \\ 
& = &  \{1 + u\pi, v\} - j \{1 + u\pi, -u\},
\end{array}
\]
where the last equality holds because $\{1 + u\pi, - u\pi\} = 0$.
It follows that $\{1 + u\pi, v\pi^j\} \in (1 + \fm)\wh{K}^M_1(R)$.
If $i \ge 2$, then $\{1 + u\pi^i, v\pi^j\} \in (1 + \fm^{i-1})\wh{K}^M_1(R) 
\subseteq (1 + \fm)\wh{K}^M_1(R)$
by ~\eqref{eqn:Milnor-1-1}. It remains therefore to prove ~\eqref{eqn:Milnor-1-1}.

We now fix $m \ge 1, j \in \Z, a \in R$ and $ u \in R^{\times}$.
We consider the element $\{1 + a \pi^{m+1}, u \pi^j\} \in (1 + \fm^{m+1})K^M_1(F)$.
We set 
\[
t = - a \pi^m, \ v' = (1 + t(-1-\pi))^{-1} \ {\rm and} \ v'' = -1 - \pi.
\]
With these notations, it is clear that $1 + v't, 1 + v''t \in (1 + \fm^m)$ and
$(1 + v't)(1 + v''t) = 1 - \pi t$. In $K^M_2(F)$, we now compute
\[
\begin{array}{lll} 
\{1 + a \pi^{m+1}, u \pi^j\} & = & \{1  + a \pi^{m+1}, u\} +  
j \{1 + a \pi^{m+1}, \pi\} \\
& = &  \{1  + a \pi^{m+1}, u\} +  j \{1 + (a \pi^{m}) \pi, \pi\} \\
& = &  \{1  + a \pi^{m+1}, u\}  - j \{1 + (a \pi^{m}) \pi, - (a \pi^m)\} \\
& = & \{1  + a \pi^{m+1}, u\}  - j\{1 -\pi t, t\} \\
& = &  \{1  + a \pi^{m+1}, u\} - j\{(1 + v't)(1 + v''t), t\} \\
& = &  \{1  + a \pi^{m+1}, u\} - j\{1 + v't, t\} - j\{1 + v''t, t\} \\
& = & \{1  + a \pi^{m+1}, u\} + j\{1 + v't, -v'\} + j\{1 + v''t, -v''\} \\

& \in & (1 + \fm^m) \wh{K}^M_1(R).
\end{array}
\] 
This proves ~\eqref{eqn:Milnor-1-1} and completes the proof of the lemma.
\end{proof}

Let us now assume that $R$ is a regular semi-local integral domain of dimension one containing a field.
Let $\{\fm_1, \ldots , \fm_r\}$ be the set of maximal ideals of $R$ and let $F$ be the field of 
fractions of $R$. 
If $\un{m} = (m_1, \ldots , m_r)$ is an $r$-tuple of positive integers, 
we shall write the relative improved Milnor $K$-theory $\wh{K}^M_*(R, \fm^{m_1}_1 \cdots \fm^{m_r}_r)$
(see  ~\eqref{eqn:Milnor-semi-local}) as $\wh{K}^M_*(R, \un{m})$.
For $n \ge 1$, we let $K^M_n(F, \un{m}, {\rm sum}) =
\stackrel{r}{\underset{i =1}\cap} (1 + \fm^{m_i}_iR_{\fm_i})K^M_{n-1}(F)$.
We let $\un{m +n} = (m_1+n, \ldots , m_r + n)$.

\begin{lem}\label{lem:Milnor-2}
  For every integer $n \ge 1$, we have $K^M_n(F, \un{m+n}, {\rm sum}) 
\subseteq \wh{K}^M_n(R, \un{m})$.
\end{lem}
\begin{proof}
  If $R$ is local, then \lemref{lem:Milnor-1} says that
  $K^M_n(F, m+n, {\rm sum}) \subseteq (1 + \fm^m) \wh{K}^M_{n-1}(R) 
\subseteq \wh{K}^M_n(R, \fm^m)$.
  In particular, the lemma holds if $R$ is local.

  In general, let $y \in K^M_n(F, \un{m+n}, {\rm sum})$. Then it is clear that
  $y \in (1 + \fm^{m_i+n}R_{\fm_i})K^M_{n-1}(F)$ for each $1 \le i \le r$.
  It follows by \lemref{lem:Milnor-1} that $y \in (1 + \fm^m_iR_{\fm_i}) 
\wh{K}^M_{n-1}(R_{\fm_i}) 
\subseteq \wh{K}^M_n(R_{\fm_i})$
for each $i$. We conclude from ~\eqref{eqn:Gersten} and \propref{prop:Kerz-finite}
that $y \in \wh{K}^M_n(R)$.

  We now consider the diagram
  \begin{equation}\label{eqn:Milnor-2-0}
    \xymatrix@C.8pc{
      \wh{K}^M_n(R) \ar[rr] \ar[dd]_-{\Delta_R} \ar[dr]^-{\pi} & & K^M_n(F) 
\ar[dd]^-{\Delta_F} \\
      & \wh{K}^M_n({R}/{\fm^{m_1}_1 \cdots \fm^{m_r}_r}) \ar[dd]^>>>>>>>>>{\phi} & \\
      \stackrel{r}{\underset{i =1}\oplus} \wh{K}^M_n({R_{\fm_i}}) 
\ar[dr]_{\oplus_i \pi_i}  \ar[rr] & & 
      \stackrel{r}{\underset{i =1}\oplus} {K}^M_n(F) \\
      & \stackrel{r}{\underset{i =1}\oplus} 
\wh{K}^M_n({R_{\fm_i}}/{\fm^{m_i}_i}), & }
  \end{equation}
 where $\phi$ is the last in the sequence of arrows in ~\eqref{eqn:Milnor-semi-local}.

  This diagram is clearly commutative. It follows from the case of local 
rings shown above that
  $(\oplus_i \pi_i) \circ \Delta_R(y) = 0$. Since $\phi$ is an isomorphism, 
we conclude that $\pi(y) = 0$, which
  is what we wanted to show. This finishes the proof.
\end{proof}

\begin{lem}\label{lem:Milnor-2-Q-coeff}
  With notations as in \lemref{lem:Milnor-2}, we have $K^M_n(F, \un{m}, {\rm sum})_{\Q} 
\subseteq \wh{K}^M_n(R, \un{m})_{\Q}$.
\end{lem}
\begin{proof}
 The reduction from the semi-local ring $R$ to it being a local ring (dvr) goes through
exactly as in the proof of \lemref{lem:Milnor-2} without any change.
So the proof of the lemma is eventually reduced to showing the following improved version of
\lemref{lem:Milnor-1} (2) for every pair of integers $n \ge 0$ and $m \ge 1$:
\begin{equation}\label{eqn:Milnor-1-mod}
  (1 + \fm^m)K^M_n(F)_{\Q} \subseteq (1 + \fm^m) \wh{K}^M_n(R)_{\Q}.
\end{equation}

This inclusion is obvious for $n = 0$. To prove this for $n  \ge 1$, an easy induction
(see ~\eqref{eqn:Milnor-1-1} in the proof of \lemref{lem:Milnor-1})
reduces to the case $n =1$. We now let $\pi$ be a uniformizing parameter of $R$ and let
$\{1 + a\pi^m, u\pi^j\} \in (1 + \fm^m) K^M_1(F)_{\Q}$, where $u \in R^{\times}, a \in R$ and
$j \in \Z$. If $ a= 0$, there is nothing to show and so we can write
$a = u_0\pi^i$ with $u_0 \in R^{\times}$ and $i \ge 0$.
We then get the following in $K^M_2(F)_{\Q}$.
\[
  \begin{array}{lll}
    \{1 + a\pi^m, u\pi^j\} & = & \{1 + u_0\pi^{i+m}, u\pi^j\} \\
                           & = & \{1 + u_0\pi^{i+m}, u\} + j \{1 + u_0\pi^{i+m}, \pi\} \\
                           & = & \{1 + u_0\pi^{i+m}, u\} + \frac{j}{i + m} \{1 + u_0\pi^{i+m}, \pi^{i + m}\} \\
                           & = & \{1 + u_0\pi^{i+m}, u\} - \frac{j}{i + m} \{1 + u_0\pi^{i+m}, -u_0\}. \\
\end{array}
  \]
  Since the last term clearly belongs to $(1 + \fm^m) \wh{K}^M_1(R)_{\Q}$, we conclude the proof of
  ~\eqref{eqn:Milnor-1-mod} and hence of the lemma. 
\end{proof}

\section{The cycle class map}\label{sec:CCM}
In this section, we shall define the cycle 
class map on the
group of 0-cycles with modulus and prove a very special case of \thmref{thm:Intro-1}. 
The final proof of this theorem 
will be done by the end of the next section. We fix an arbitrary field $k$.

\subsection{The map $cyc_{X|D}$ on generators}\label{sec:Generators}
Let $(X,D)$ be a modulus pair in $\Sch_k$ of dimension $d \ge 1$. Let $n \ge 0$ 
be an integer. 
We begin by defining the cycle class map $cyc_{X|D}$ on the
group $z^{d+n}(X|D, n)$.
Let $Z \in X \times \square^n$ be an admissible closed point. Since $Z$ is a closed point, 
we have that
$Z = \Spec(k(Z))$. We let $p_{\ov{\square}^n} \colon X \times \ov{\square}^n \to \ov{\square}^n$ 
and $p_X \colon X \times \ov{\square}^n \to X$ denote the projection maps.
We let $f \colon Z \to X$ denote the projection map. It is clear that $f$ is a 
finite map
and its image is a closed point $x \in X$ which does not lie in $D$. We thus have a 
factorization $Z \to \Spec(k(x)) \to X^o \to X$ of $f$.
The latter is actually a map of modulus pairs 
$f \colon (Z, \emptyset) \to (X,D)$.
Hence, it induces the proper push-forward 
$f_* \colon \CH_q(Z, *) \to \CH_q(X|D, *)$, where
$\CH_q(Z, *)$ are Bloch's higher Chow groups of $Z$ \cite{Bloch-1}.

Now, the closed point $Z \in X \times \square^n$ defines a unique 
$k(Z)$-rational point
(which we also denote by $Z$) in $\square^n_{Z}$ such that the composite 
projection map
$Z \inj \square^n_{Z} \to Z$ is identity. Furthermore, $[Z] \in z^{d+n}(X|D, n)$ 
is the image of $[Z] \in z^n(Z, n)$ under the push-forward map $f_*$.
Since $Z \inj \square^n_{Z}$ does not meet any face of $\square^n$, it follows that 
$y_i(Z) \in k(Z)^{\times}$ for every $1 \le i \le n$, where $y_i \colon 
\ov{\square}^n_Z \to \ov{\square}_Z$ is the
projection to the $i$-th factor.
In particular, $\{y_1(Z), \ldots , y_n(Z)\}$ is a well-defined element of 
$K^M_n(k(Z))$. We let 
\begin{equation}\label{eqn:cycle-0}
cyc^M_Z([Z]) = \{y_1(Z), \ldots , y_n(Z)\} \in K^M_n(k(Z)) \ \ {\rm and} 
\end{equation}
\begin{equation}\label{eqn:cycle-1}
cyc_Z([Z]) = \psi_Z \circ cyc^M_Z([Z]) \in K_n(Z), 
\end{equation}
where recall that $\psi_Z \colon K^M_*(k(Z)) \to K_*(k(Z)) = K_*(Z)$ is
the canonical map from the Milnor to the Quillen $K$-theory.

We next recall from \S~\ref{sec:Proj-F} that as $x = f(Z) \in X^o$ (which is regular),
the finite map $f$ defines a map of spectra $f_* \colon K(Z) \to K(X,D)$ such that
the composite map $K(Z) \to K(X,D) \to K(X)$ is the usual push-forward map.
The same holds for the inclusion  $\iota^x: \Spec(k(x)) \inj X$.
We let
\begin{equation}\label{eqn:cycle-3} 
cyc_{X|D}([Z]) = f_* \circ cyc_Z([Z]) \in K_n(X,D).
\end{equation}
Extending this linearly, we obtain our cycle class map

\begin{equation}\label{eqn:cycle-4} 
cyc_{X|D} \colon z^{d+n}(X|D,n) \to K_n(X,D).
\end{equation}

If $Z$ is an admissible closed point as above and $x =  f(Z)$, 
then we have a commutative diagram
\begin{equation}\label{eqn:cycle-2}
\xymatrix@C.8pc{
K^M_n(k(Z)) \ar[d]_-{N_{Z/x}} \ar[r]^-{\psi_Z} & K_n(k(Z)) 
\ar[d]^-{T_{{k(Z)}/k}} \ar[dr]^-{f_*} & \\ 
K^M_n(k(x)) \ar[r]^-{\psi_x} & K_n(k(x)) \ar[r]^-{\iota^x_*} & K_n(X, D),}
\end{equation}
where $N_{Z/x}$ is the Norm map between the Milnor $K$-theory of fields \cite{BT} 
(see also \cite{Kerz09})
and the right vertical arrow is the transfer (push-forward) map between the 
Quillen $K$-theory of fields. The square on the left commutes by \lemref{lem:L6}.
We can therefore write 
\[
\begin{array}{lll}
cyc_{X|D}([Z]) & = & f_* \circ cyc_Z([Z]) \\
& =  & f_* \circ \psi_Z \circ cyc^M_Z([Z]) \\
& = & f_* \circ \psi_Z (\{y_1(Z), \ldots , y_n(Z)\}) \\
& = & \iota^x_* \circ \psi_x \circ N_{Z/{\{x\}}}(\{y_1(Z), \ldots , y_n(Z)\}).
\end{array}
\]

It is clear from the definition that for any integer $m \ge 1$, 
there is a commutative diagram
\begin{equation}\label{eqn:cycle-5}
\xymatrix@C.8pc{
z^{d+n}(X|(m+1)D,n) \ar[r]^-{cyc_{X|D}} \ar[d] & K_n(X,(m+1)D) \ar[d] \\
z^{d+n}(X|mD,n) \ar[r]^-{cyc_{X|D}} & K_n(X,mD),}
\end{equation}
where the vertical arrows are the canonical restriction maps.
We therefore have a strict map of pro-abelian groups
\begin{equation}\label{eqn:cycle-6} 
cyc_{X|D} \colon \{z^{d+n}(X|mD,n)\}_m \to \{K_n(X,mD)\}_m.
\end{equation}

\vskip .3cm

We next prove that the map $cyc_{X|D}$ is covariant with respect to proper morphisms of modulus pairs and contravariant for the flat morphisms of modulus pairs which are of relative dimension $0$. Note that these are the only general cases where the functoriality of the cycle class map makes sense.

\vskip .3cm

\subsection{Naturality for flat morphisms}

Let $(Y,E)$ and $(X,D)$ be modulus pairs.
Let $h : Y \to X$ be a flat morphisms of relative dimension $0$ such that $E = h^*(D)$. 
Recall from \cite[Proposition~2.12]{KP} that we have a pull-back map $h^*: z^{d+n}(X| D, n) \to z^{d+n}(Y|E, n)$ such that 
$h^*([Z]) = [W]$, where $W= (h\times {\rm id}_{\square^n})^{-1}(Z)$ and $Z$  a closed point in $ X \setminus D \times \sq^n$.

\begin{lem} \label{lem:flat-pull-back}
With notations as above, the following diagram commutes: 
\begin{equation}\label{eqn:Pull-Back-1}
\xymatrix@C.8pc{
z^{d+n}(X|D,n) \ar[r]^-{cyc_{X|D}} \ar[d]^{h^*} & K_n(X,D) \ar[d]^{h^*} \\
z^{d+n}(Y|E,n) \ar[r]^-{cyc_{Y|E}} & K_n(Y,E).}
\end{equation}
\end{lem}
\begin{proof}
Let $Z$ be a closed point in $z^{d+n}(X| D, n)$ and let $[W] = \sum_{i=1}^r m_i [W_i] \in z^{d+n}(Y|E, n)$, where $W_i$ are irreducible components of the inverse image scheme $W$ with multiplicities $m_i$. 
 Let $f^Z: Z \to X$, $f^W: W \to Y$ and $f^{W_i}: W_i \to Y$ denote the respective projections. Let
$y(Z) = cyc^M_Z([Z])$ and $y(W_i)= cyc^M_{W_i}([W_i])$ as in \eqref{eqn:cycle-0}.
We then have to show that 
\begin{equation}\label{eqn:Pull-Back-2}
h^* \circ f^Z_* \circ \psi_Z (y(Z)) = \sum_{i=1}^r m_i  (\ f^{W_i}_*\circ  \psi_{W_i} (y(W_i))). 
\end{equation}
Consider the following diagram: 
\begin{equation}\label{eqn:Pull-Back-3}
      \xymatrix@C.8pc{  
K(Z) \ar[r] \ar[d]^{h^*}& G(Z) \ar[r] \ar[d]^{h^*}  & K^Z(X\setminus D) \ar[d]^{h^*}  &  K^Z(X, D) \ar[l]_-{\cong} \ar[r] \ar[d]^{h^*}
& K(X, D) \ar[d]^{h^*}  \\
K(W) \ar[r] & G(W) \ar[r]  & K^W(Y \setminus E) &  K^W(Y, E) \ar[l]_-{\cong} \ar[r] & K(Y,E).}
\end{equation}
Since $h$ is flat, it follows that all the squares in  ~\eqref{eqn:Pull-Back-3} commute. Indeed, since the canonical map $K(-) \to G(-)$ respects flat pull-back, the left-most square in ~\eqref{eqn:Pull-Back-3} commutes. The middle left square commutes by \cite[Proposition~3.18]{TT} and the middle right square commutes because each map is a pull-back map. Lastly, the right-most square in ~\eqref{eqn:Pull-Back-3} commutes by the definition of the left arrow in the square. As discussed in \S~\ref{sec:Proj-F}, 
the composition of the top horizontal arrows is the push forward map $f^Z_*$ and the composition of the top horizontal arrows is the push forward map $f^W_*$.  

It then suffices to show that 
\begin{equation}\label{eqn:Pull-Back-4}
 f^W_* \circ h^* \circ \psi_Z (y(Z)) = \sum_{i=1}^r m_i  (\ f^{W_i}_*\circ  \psi_{W_i} (y(W_i))). 
\end{equation}
 Since $W = (h \times {\rm id}_{\square^n})^{-1}(Z)$, we have $y(W_i)= y(Z)$ for each $1\leq i \leq r$ under the injective map $k(Z)^{\times}  \inj k(W_i)^{\times}$. It then follows that 
 $\psi_{W_i} (y(W_i)) = h_i^*  \psi_Z (y(Z))  \in K_n(W_i)$, where $h_i: W_i \to Z$ is the induced map. Note that $h_i = h \circ g_i$, where $g_i : W_i \inj W$ denotes the inclusion of the irreducible component $W_i$ into $W$. We are therefore reduced to show that 
 \begin{equation}\label{eqn:Pull-Back-5}
 f^W_* \circ h^* \circ \psi_Z (y(Z)) = \sum_{i=1}^r m_i  (\ f^{W_i}_*\circ g_i^* \circ h^* \circ \psi_Z (y(Z))).
 \end{equation}
We shall actually show that  for all $a \in K_n(W)$, we have
 \begin{equation}\label{eqn:Pull-Back-6}
 f^W_* (a) = \sum_{i=1}^r m_i \ f^{W_i}_* \circ g_i^* (a).
 \end{equation}
 Observe that the equality ~\eqref{eqn:Pull-Back-5} follows from  ~\eqref{eqn:Pull-Back-6} with $ a= h^* \circ \psi_Z (y(Z))$. To show ~\eqref{eqn:Pull-Back-6}, consider the diagram:
\begin{equation}\label{eqn:Pull-Back-7}
      \xymatrix@C.8pc{  
 \coprod_{i=1}^r G(W_i) \ar[r]^-{\cong} \ar[rd]^{(g_{i *})_i}&  G(W_{\red}) \ar[r]  \ar[d]^{g_*} & K^W(Y \setminus E) \ar@{=}[d] &  K^W(Y, E) \ar[l]_-{\cong} \ar[r] \ar@{=}[d]& K(Y,E)  \ar@{=}[d]\\
K(W) \ar[r] & G(W) \ar[r]  & K^W(Y \setminus E) &  K^W(Y, E) \ar[l]_-{\cong} \ar[r] & K(Y,E).}
\end{equation}
 Since $W_{\red} = \coprod_i W_i$, it follows that the push-forward map $ \coprod_{i=1}^r G(W_i) \to G(W_{\red})$ is an isomorphism and the left triangle in ~\eqref{eqn:Pull-Back-7} commutes. 
Observe that the left-most square commutes because all arrow in the square are (compatible) push-forward maps. As before, the composition of the top horizontal arrows on $G(W_i)$ is the push-forward map $f^{W_i}_*$. Let $b \in G_n(W)$ be the image of $a\in K_n(W)$ under the map $K_n(W) \to G_n(W)$ induced by the bottom left arrow in ~\eqref{eqn:Pull-Back-7}.
The equality ~\eqref{eqn:Pull-Back-6} then follows if we show that 
\begin{equation}\label{eqn:Pull-Back-8}
b = g_* (\sum_{i=1}^r m_i \ g^*_i (b) ) \in G_n(W).
\end{equation}
The equality ~\eqref{eqn:Pull-Back-8} however
follows from the following calculation:
\begin{eqnarray*}
 g_* (\sum_{i=1}^r m_i \ g^*_i (b) ) &= & \sum_{i=1}^r m_i \ g_{i *}\circ g_i^* (b)\\
 & =^1 & \sum_{i=1}^r m_i \ g_{i *} ( g_i^* (b) [\sO_{W_i}])\\
 & =^2 & b \left( \sum_{i=1}^r m_i  \ g_{i *}( [\sO_{W_i}]) \right)\\
 & =^3 & b [\sO_W] = b,
\end{eqnarray*}
where $=^1$ follows because for each $i$, we have $[\sO_{W_i}] = 1 \in G_0(W_i)$, $=^2$ follows from the projection formula  \cite[Proposition~3.17]{TT}  for $G$-theory because $a \in K_n(W)$ and $=^3$ follows as $1 =  [\sO_W] = \sum_{i=1}^r m_i  \ g_{i *}( [\sO_{W_i}]) \in G_0(W)$. This completes the proof of the lemma. 
\end{proof}

\vskip .3cm

\subsection{Naturality for proper morphisms}
Let $(Y,E)$ and $(X,D)$ be modulus pairs such that $X$ and $Y$ are regular schemes over $k$ of pure dimension $d_X$ and $d_Y$. 
Let $h: Y \to X$ be a proper morphisms such that $E = h^*(D)$. By \cite[Proposition~2.10]{KP}, we have a proper push-forward map $h_*: z^{d_Y+n}(Y|E, n) \to z^{d_X +n}(X|D, n)$ such that $h_*([z])=[k(z):k(w)] [w]$, where $w = (h \times {\rm id}_{\square^n}) (z)$ and $z$ is a closed point in  $Y \setminus E \times \sq^n$. The existence of the push-forward map $h_* : K(Y, E) \to K(X,D)$ follows from the following lemma.

\begin{lem}\label{lem:Fin-tor-K-thy}
The map $h$ induces a proper push-forward map $h_* \colon K(Y, E) \to K(X,D)$ 
of relative  $K$-theory spectra.
\end{lem}
\begin{proof}
Since $Y$ is regular, we can assume with out loss of generality that $Y$ is integral. 
Since $X$ is regular, the map $h$ has finite tor-dimension. 
By \cite[3.16.4]{TT}, we have a push-forward map $f_* \colon K(Y) \to K(X)$. Observe that if $E= \emptyset$, then we have a push-forward map $f_* \colon K(Y) \to K(X,D)$. We can therefore assume that $E \neq \emptyset$. 
To prove the lemma, it suffices to show that $Y$ and $D$
are tor-independent over $X$. This will in particular imply that 
$E \to D$ also has finite tor-dimension.
For tor-independence, we note that  $D$ is an effective Cartier divisor.
Hence, the only 
possible non-trivial tor term can be ${\rm Tor}^1_{\sO_X}(\sO_{Y}, \sO_D)$.
But this is same as the $\sI_{D}$-torsion subsheaf of $\sO_{Y}$. Since
$Y$ is integral, this torsion subsheaf is non-zero if and only if
the ideal $\sI_E$ is zero. But this can not happen as $E$ is a proper
divisor on $Y$. This finishes the proof.
\end{proof}  

\begin{remk}\label{rem:Fin-tor-K-thy-1}
Observe that \lemref{lem:Fin-tor-K-thy} is true for a general integral scheme (may not be regular) $Y$ over a regular scheme $X$. 
\end{remk}

We now prove that the map $cyc_{X|D}$ in \eqref{eqn:cycle-4} commutes with the push-forward map.

\begin{lem}\label{lem:Push-forward}
For a cycle $\alpha \in z^{d_Y+n}(Y|E, n)$, we have 
$cyc_{X|D} \circ h_*(\alpha) = h_* \circ cyc_{Y|E}(\alpha)$.
\end{lem}
\begin{proof}
We can assume $\alpha$ is represented by a closed point 
$z \in z^{d_Y+n}(Y|E, n)$. We set 
$w = (h \times {\rm id}_{\square^n})(z) \in X \times \square^n$ and 
$x = p_X(w)$.
The compatibility between norm maps in the Milnor $K$-theory of fields and
push-forward maps in the Quillen $K$-theory (see \lemref{lem:L6}) yields a 
commutative diagram
\begin{equation}\label{eqn:0-C-map-4}
\xymatrix@C2pc{
K^M_n(k(z)) \ar[r]^{N_{z/w}} \ar[d]_-{\psi_z} \ar@/_3pc/[dd]_{{\iota_z}_*} & 
K^M_n(k(w)) \ar[d]^-{\psi_w} \ar@/^3pc/[dd]^{{\iota_w}_*} \\
K_n(k(z)) \ar[r]_{h_*} \ar[d]_{p_*} & K_n(k(w)) \ar[d]^-{p_*} \\
K_n(Y,E) \ar[r]_{h_*} & K_n(X,D).}
\end{equation} 

Using this commutative diagram, we get
\[
\begin{array}{lll}
cyc_{X|D} \circ h_*([z]) & = & cyc_{X|D}([k(z): k(w)] [w]) \\
& = & {\iota_w}_*(\{y_1(w), \cdots , y_n(w)\}^{[k(z): k(w)]}) \\
& {=}^1 & {\iota_w}_* \circ {\iota_{z/w}}_* \circ \iota^*_{z/w} 
(\{y_1(w), \cdots , y_n(w)\}) \\
& {=}^2 & {\iota_w}_* \circ N_{z/w} (\{y_1(z), \cdots , y_n(z)\}) \\
& = & h_* \circ {\iota_z}_* (\{y_1(z), \cdots , y_n(z)\}) \\
& = & h_* \circ cyc_{{Y}|E}([z]).
\end{array} 
\]

In this set of equalities, recall our notation (preceding \lemref{lem:Curve-boundary})  that
$y_i(z)$ is the image of $z$ under the $i$-th projection $\Spec(k(z)) \to \square_{k(z)}$.
In particular, $y_i(z) \in k(z)^{\times}$ for all $1 \le i \le n$.
The coordinates $y_i(w) \in k(w)^{\times}$ have similar meaning.
The map $\iota_{z/w}: \Spec(k(z)) \to \Spec(k(w))$ is the projection.
The equality ${=}^1$ is a consequence of the projection formula for the Milnor $K$-theory
associated to the resulting inclusion $k(w) \inj k(z)$. The equality ${=}^2$ follows from
the fact that $y_i(z) = y_i(w) = \iota^*_{z/w}(y_i(w))$ via the inclusion $k(w) \inj k(z)$
for all $1 \le i \le n$. 
This proves the lemma.
\end{proof}

\vskip .3cm

We end this section with some comments below
on $cyc_{X|D}$ in the non-modulus case.

\subsection{Agreement with Levine's map}\label{sec:Coincide}
We had mentioned in \S~\ref{sec:Intro} that for Bloch's higher Chow groups of 0-cycles,
a cycle class map to the ordinary $K$-theory of a regular variety was constructed by
Levine \cite{Levine-1} with rational coefficients. Binda \cite{Binda} constructed 
such a map in the modulus setting and his map is identical to that of Levine by definition.
It is not hard to check that the map
$cyc_{X|D}$ coincides with Levine's map when $D = \emptyset$ (note that
$z^{d+n}(X|D,n) \subseteq z^{d+n}(X,n)$).
In particular, it turns out that Levine's cycle class map for the ordinary 0-cycles exists
with integral coefficients. We give a sketch of this agreement and leave the
details for the reader.

Let $n \ge 0$ be an integer. Since $X$ is regular, the homotopy invariance implies that 
the multi-relative $K$-theory exact sequence (see \cite[\S~1]{Levine-1}) yields an isomorphism
\begin{equation}\label{eqn:Levine*-0}
\theta_n \colon K_n(X) \xrightarrow{\cong} K_0(X \times \square^n, \partial \square^n),
\end{equation}
where $K(X \times \square^n, \partial \square^n)$ is the iterated
multi-relative $K$-theory of $X \times \square^n$
relative to all codimension one faces.

Let $Z \subset X \times \square^n$ be a closed point. Then Levine's cycle class map $cyc^L_X([Z])$
is the image of $1 \in K_0(Z)$ under the composition
\begin{equation}\label{eqn:Levine*-1}
K_0(Z) \cong K^Z_0(X \times \square^n, \partial \square^n)
\to K_0(X \times \square^n, \partial \square^n) \xleftarrow{\theta_n} K_n(X).
\end{equation}

Let $f \colon Z \to X$ denote the projection map (which is finite). We then have
a finite map of multi-closed pairs $f \colon (Z \times \square^n, \partial \square^n) \to
(X \times \square^n, \partial \square^n)$. Since the relative $K$-theory fiber sequence
commutes with finite push-forward for regular schemes (see 
\lemref{lem:Fin-tor-K-thy}),
and since $cyc^L_X([Z])$ is the image of
$[Z] \in  K_0(Z \times \square^n, \partial \square^n)$ under the push-forward map $f_*$,
we can assume that $X = \Spec(k)$ and $f \colon Z \to \Spec(k)$ is the identity map.

When $n = 0$, the agreement of $cyc_k([Z])$ and $cyc^L_k([Z])$ is immediate.
When $n \ge 1$ and if we follow our notation of ~\eqref{eqn:cycle-0}, then
we see that $y_i(Z) = a_i \in k^{\times}$ for each $1 \le i \le n$.
In particular, we have $cyc_k([Z]) = \{a_1, \ldots , a_n\} \in K_n(k)$.
One therefore has to show that
if $z = (a_1, \ldots, a_n) \in (k^{\times})^n$ is a $k$-rational point in $\square^n$,
then the class $[k(z)] \in K_0(\square^n, \partial \square^n)$ coincides 
with $\theta_n(\{a_1, \dots , a_n\})$. But this is an elementary exercise in $K$-theory using
repeated application of relative $K$-theory exact sequence. For $n =1$, it already follows from
a straightforward generalization of \lemref{lem:Elem-0} 
(where we replace $(t^{m+1})$ by any ideal of $R[t]$) with identical proof.
We leave it to the reader to check the details for $n \ge 2$.

\vskip .3cm

\section{The case of regular curves}\label{sec:curves}

The goal now is to show that $cyc_{X|D}$ kills the rational equivalence
if we allow the modulus to vary along $\{mD\}_{m \ge 1}$. In this section, we shall prove a very 
special case of this. The proof of \thmref{thm:Intro-1} will be reduced to this case 
in the next section. 
We consider the following situation. Let $n \ge 0$ be an integer.
We let $X$ be a regular connected curve over $k$ and let $D \subset X$ be
an effective Cartier divisor. Let $\ov{W} \subset X \times \ov{\square}^{n+1}$ be a closed subscheme
such that the following hold.
\begin{enumerate}
\item
The composite map $\ov{W} \xrightarrow{\nu} X \times \ov{\square}^{n+1} \xrightarrow{p_X} X$ is an 
isomorphism. 
\item
$W = \ov{W} \cap (X \times \square^{n+1})$ is an admissible cycle on $X \times \square^{n+1}$ with 
modulus $(n+1)D$. That is, $[W] \in z^{n+1}(X|(n+1)D, n+1)$.
\end{enumerate}

Let $\partial = \stackrel{n+1}{\underset{i =1}\sum} (-1)^i(\partial^{\infty}_i - \partial^0_i):
z^{n+1}(X|E, n+1) \to z^{n+1}(X|E,n)$ be the boundary map in the cycle complex with modulus
for an effective divisor $E \subset X$.
We want to prove the following result in this subsection.

\begin{prop}\label{prop:Curve-vanish}
The class $cyc_{X|D}(\partial W)$ dies in $K_n(X,D)$ under the 
cycle class map
\[
z^{n+1}(X|D, n) \xrightarrow{cyc_{X|D}} K_n(X,D).
\]
\end{prop}

\vskip .3cm

We shall prove this proposition in several steps. We begin with the following description of the
cycle class map on various boundaries of $W$. Let $F$ denote the function field of $X$.
Let $g$ denote the Milnor symbol $\{g_1, \ldots , g_{n+1}\} \in K^M_{n+1}(F)$, where
$g_i: \ov{W} \to \ov{\square}$ is the $i$-th projection for $1 \le i \le n+1$.
Note that this symbol is well defined because no $g_i$ can be identically zero by the admissibility
of $W$. For any closed point $z \in \ov{W}$, let $\ord_z: F^{\times} \to \Z$ denote the valuation 
associated
to the discrete valuation ring $\sO_{\ov{W}, z}$. We let
$\partial^M_z: K^M_{i+1}(F) \to K^M_i(k(z))$ denote a boundary map in the Gersten complex 
~\eqref{eqn:Gersten}. 
We let $\fm_z$ denote the maximal ideal of the local ring $\sO_{\ov{W}, z}$.
A symbol $\{a_1, \ldots , \wh{a_i}, \ldots , a_n\}$ will mean the one obtained from
$\{a_1, \ldots , a_n\}$ by omitting $a_i$. 
For any point $z \in X \times \square^{n+1}$, let $y_i \colon \square^{n+1}_{k(z)} \to \square_{k(z)}$
denote the projection map to $i$-th factor.
Let $f^z \colon \Spec(k(z)) \to X$ denote the projection to $X$.

\begin{lem}\label{lem:Curve-boundary}
For $1 \le i \le n+1$, we have 
\[
cyc_{X|D}(\partial^0_i W) = {\underset{z \in \partial^0_i W}\sum} 
\ord_z(g_i) f^z_* \circ \psi_z(\{y_1(z), \ldots , \wh{y_i(z)}, \ldots , y_{n+1}(z)\}),
\]
\[
cyc_{X|D}(\partial^{\infty}_i W) = {\underset{z \in \partial^\infty_i W}\sum} 
\ord_z({1}/{g_i}) f^z_* \circ \psi_z(\{y_1(z), \ldots , \wh{y_i(z)}, \ldots , y_{n+1}(z)\}).
\]
\end{lem}
\begin{proof}
We should first observe that the admissibility of $W$ implies that if $z \in \partial^{t}_i W$
for $t \in \{0, \infty\}$, then we must have $y_j(z) \neq 0$ for all $j \neq i$. In particular, 
the element $\{y_1(z), \ldots , \wh{y_i(z)}, \ldots , y_{n+1}(z)\} \in K^M_n(k(z))$ is well defined.
By the definition of $cyc_{X|D}$, it suffices to show that for $1 \le i \le n+1$, we have
\begin{equation}\label{eqn:Curve-boundary-0}
\partial^0_i W  = {\underset{z \in \partial^0_i W}\sum}  \ord_z(g_i) [z] \ \ {\rm and} \ \
\partial^\infty_i W  = {\underset{z \in \partial^\infty_i W}\sum}  \ord_z({1}/{g_i}) [z].
\end{equation}
But this is an immediate consequence of the definition of the intersection product of an integral 
cycle with the faces of $X \times \square^n$.
\end{proof}

\begin{lem}\label{lem:Curve-Milnor}
Suppose that $n \ge 1$ and $\partial^M_X(g) \in {\underset{z \in X^o}\oplus} K^M_n(k(z))$
under the Milnor boundary map $\partial^M_X \colon K^M_{n+1}(F) \to 
{\underset{z \in X^{(1)}}\oplus} K^M_n(k(z))$.
Then
\begin{equation}\label{eqn:Curve-Milnor-ex}
{\underset{z \in X^o}\sum} f^z_* \circ \psi_z \circ \partial^M_{X} (g) = 0.
\end{equation}
\end{lem}
\begin{proof}
  Suppose first that $D = \emptyset$. In this case, we have a diagram
  \begin{equation}\label{eqn:CV-ex}
    \xymatrix@C.8pc{
      K^M_{n+1}(F) \ar[r]^-{\partial^M_X} \ar[d]_-{\psi_F} & 
{\underset{z \in X^{(1)}}\oplus} K^M_n(k(z))
      \ar[d]^-{(\psi_z)_z}
      & \\
      K_{n+1}(F) \ar[r]^-{\partial^Q_X}  & {\underset{z \in X^{(1)}}\oplus} K_n(k(z))
      \ar[r]^-{(f^z_*)_z} & K_n(X).}
    \end{equation}

    The Gersten complex for Milnor $K$-theory canonically maps to the Gersten 
complex for the
    Quillen $K$-theory by \lemref{lem:L1}. 
In particular, the square in the above diagram is commutative. The bottom 
row is
    exact by Quillen's localization sequence and a limit argument. The lemma follows 
immediately from this diagram.

    We now let $D = m_1x_1 + \cdots + m_r x_r$, where $x_1, \ldots , x_r$ are distinct closed points of 
$X$ and $m_1, \ldots , m_r$ are positive integers. Let $A = \sO_{X,D}$ be the 
semi-local ring of $X$ at $D$ and let
    $I$ denote the ideal of $D$ inside $\Spec(A)$. The localization and
    relativization sequences give us the commutative diagram of homotopy fiber sequences
    \begin{equation}\label{eqn:Curve-Milnor-0}
      \xymatrix@C.8pc{
        {\underset{z \in X^o}\amalg} K(k(z)) \ar@{=}[r] \ar[d] & 
{\underset{z \in X^o}\amalg} K(k(z)) 
\ar[d] & \\
        K(X,D) \ar[r] \ar[d] & K(X) \ar[r] \ar[d] & K(D) \ar@{=}[d] \\
        K(A, I) \ar[r] & K(A) \ar[r] & K(D).}
      \end{equation}

      The associated homotopy groups long exact sequences yield the commutative diagram of 
exact sequences of abelian groups
      \begin{equation}\label{eqn:Curve-Milnor-1}
        \xymatrix@C.8pc{
          K_{n+1}(A,I) \ar[r]^-{\partial^Q_A} \ar[d] & 
{\underset{z \in X^o}\oplus} K_n(k(z)) \ar@{=}[d]
          \ar[r]^-{(f^z_*)_z} & K_n(X,D) \ar[d] \\
          K_{n+1}(A) \ar[r]^-{\partial^Q_A} \ar[d]_-{u^*} & 
{\underset{z \in X^o}\oplus} K_n(k(z)) 
          \ar[r]^-{(f^z_*)_z} & K_n(X) \\
          K_{n+1}(D), & & }
        \end{equation}
        where $u \colon D \inj \Spec(A)$ is the inclusion map.
        It follows from this diagram that there is an exact sequence
        \begin{equation}\label{eqn:Curve-Milnor-2}
          \wh{K}_{n+1}(A,I) \xrightarrow{\partial^Q_{A}} 
{\underset{z \in X^o}\oplus} K_n(k(z)) \xrightarrow{(f^z_*)_z} K_n(X,D).
        \end{equation}

        In order to compare this with the Milnor $K$-theory, we consider the diagram
        \begin{equation}\label{eqn:Curve-Milnor-3}
        \xymatrix@C.8pc{
          \wh{K}^M_{n+1}(A) \ar[dr] \ar[dd]_-{\psi_A} \ar[rr]^-{\partial^M_A} & &
          {\underset{z \in X^o}\oplus} K^M_n(k(z))
          \ar[dd]^->>>>>{(\psi_z)_z} \ar@{^{(}->}[dr] & \\
          & K^M_{n+1}(F) \ar[dd]^->>>>>>>{\psi_F} 
\ar[rr]^>>>>>>>>>{\partial^M_X}
          & &  {\underset{z \in X}\oplus} K^M_n(k(z)) \ar[dd]^-{(\psi_z)_z} \\
          K_{n+1}(A) \ar[dr] \ar[rr]^>>>>>>{\partial^Q_A}  & & 
{\underset{z \in X^o}\oplus} 
K_n(k(z)) \ar@{^{(}->}[dr] & \\
          & K_{n+1}(F) \ar[rr]^{\partial^Q_X} & & {\underset{z \in X}\oplus} K_n(k(z)).}
    \end{equation}

    The map $\psi_A$ comes from ~\eqref{eqn:Mil-Q}. 
    By the definition of $\wh{K}^M_{n+1}(A)$, we know that the composite map
    $\wh{K}^M_{n+1}(A) \to K^M_{n+1}(F) \xrightarrow{\partial^M_z} K^M_n(k(z))$ is 
zero for all closed points
    $z \in D$. Hence, the composite $\wh{K}^M_{n+1}(A) \to K^M_{n+1}(F) 
\xrightarrow{\partial^M_X}
    {\underset{z \in X}\oplus} K^M_n(k(z))$ factors through the map denoted by $\partial^M_A$ 
in the above diagram. In particular, the top face of ~\eqref{eqn:Curve-Milnor-3}
commutes. Exactly the same reason shows that
    the bottom face also commutes. Furthermore, $\partial^Q_A$ is same as the boundary map 
    in the bottom row of ~\eqref{eqn:Curve-Milnor-1}. The left and the right faces clearly commute and
    so does the front face by \lemref{lem:L1}. A diagram chase shows that the back face of 
~\eqref{eqn:Curve-Milnor-3} commutes too.
    
In order to show ~\eqref{eqn:Curve-Milnor-ex},
we consider our final diagram
\begin{equation}\label{eqn:Curve-Milnor-4}  
\xymatrix@C.8pc{
& \wh{K}^M_{n+1}(A) \ar[dr]^-{\partial^M_A} \ar[dd]^>>>>>>>>{\psi_A} & & \\
\wh{K}^M_{n+1}(A,I) \ar@{^{(}->}[ur] \ar[rr]^>>>>>>{\partial^M_A} 
\ar[dd]_-{\psi_{(A,I)}} & & 
{\underset{z \in X^o}\oplus} K^M_n(k(z)) \ar[dd]^-{(\psi_z)_z} \\
&  K_{n+1}(A) \ar[dr]^-{\partial^Q_A} & & \\
\wh{K}_{n+1}(A,I) \ar@{^{(}->}[ur] \ar[rr]^-{\partial^Q_A} & &
{\underset{z \in X^o}\oplus} K_n(k(z)) \ar[r]^-{(f^z_*)_z} & K_n(X,D).}
\end{equation}
The left face of this diagram commutes by 
~\eqref{eqn:Milnor-semi-local-1} and we just showed above that the right face commutes. 
In particular, the bottom face also commutes.
Furthermore, the bottom row is same as the exact sequence ~\eqref{eqn:Curve-Milnor-2}.
Our assertion will therefore follow if we can show that $g \in \wh{K}^M_{n+1}(A,I)$
provided $W \in z^{n+1}(X|(n+1)D,n+1)$.

Suppose now that $W \in z^{n+1}(X|(n+1)D,n+1)$. By the definition of the modulus 
condition (see \S~\ref{sec:HCGM}), it means that 
$\stackrel{n+1}{\underset{j = 1}\sum} \ord_{x_i}(g_j - 1) \ge (n+1)m_i$ for each
$1 \le i \le r$. 
Since $n, m_i \ge 1$, we must have $(n+1)m_i \ge m_i + n$ for each $1 \le i \le r$. 
This implies that $g = \{g_1, \ldots , g_{n+1}\} \in K^M_{n+1}(F, m_i + n)$
for every $1 \le i \le r$ in the notations of \lemref{lem:Milnor-0}. 
If we now apply \lemref{lem:Milnor-0} with $R = A_{\fm_i}$, it follows that
$g \in (1 + \fm_i^{m_i +n})K^M_n(F)$ for every $1 \le i \le r$. 
In other words, $g$ lies in the intersection 
$\stackrel{r}{\underset{i =1} \cap} (1 + \fm_i^{m_i + n})K^M_n(F)$ as an element 
of $K^M_{n+1}(F)$. It follows from \lemref{lem:Milnor-2} that 
$g \in \wh{K}^M_{n+1}(A, \un{m}) = \wh{K}^M_{n+1}(A,I)$.
This finishes the proof of the lemma.
\end{proof}

{\bf{Proof of \propref{prop:Curve-vanish}:}} 
We assume first that $n = 0$. In this case, we will show the stronger assertion that 
$cyc_{X|D}(\partial W)$ dies in $K_n(X,D)$ if $W \in z^{n+1}(X|D, n+1)$.
Let $A$ be the semi-local ring and $I \subset A$ the ideal as in the proof of
\lemref{lem:Curve-Milnor}.
By ~\eqref{eqn:Curve-Milnor-0}, we have an exact sequence
\begin{equation}\label{eqn:mod-k}
K_1(A,I) \xrightarrow{\partial^Q_A} {\underset{z \in X^o}\oplus} K_0(k(z))
\xrightarrow{(f^z_*)_z} K_0(X,D) \to 0.
\end{equation}
Comparing this with the exact sequence
\[
K_1(A) \xrightarrow{\partial^Q_A} {\underset{z \in X^o}\oplus} K_0(k(z))
\xrightarrow{(f^z_*)_z} K_0(X),
\]
we see that we can replace $K_1(A,I)$ by $\wh{K}_1(A,I)$ in ~\eqref{eqn:mod-k}.  
But then, it is same as the exact sequence
\[
(1 + I)^{\times} \xrightarrow{\partial^M_A} z^1(X|D, 0) 
\xrightarrow{cyc_{X|D}}  K_0(X,D) \to 0.
\]
Moreover, one knows that ${\rm Coker}(\partial^M_A) \cong \CH^1(X|D, 0)$
(e.g, see \cite[\S~2]{Krishna-1}). We therefore showed
that $cyc_{X|D}: \CH^1(X|D,0) \to K_0(X,D)$ is actually an isomorphism.

We now assume for the remaining part of the proof that $n \ge 1$.
As before, let $g_i: \ov{W} \to \ov{\square}$ denote the projections and let 
$f = p_X \circ \nu \colon \ov{W}
\to X$ be the projection to $X$. We also recall the element $g = \{g_1, \ldots , g_{n+1}\} 
\in K^M_{n+1}(F)$.
Our task is to show that
\begin{equation}\label{eqn:Curve-vanish-0}
cyc_{X|D}(\partial W) = \stackrel{n+1}{\underset{i =1}\sum}
(-1)^i cyc_{X|D}(\partial^{\infty}_i W - \partial^0_i W) = 0 \ \ {\rm in} \ \ K_n(X,D)
\end{equation}
if $W$ satisfies the modulus condition for $(n+1)D$.

Our idea is to compute the cycle class of $\partial W$ in terms of the cycle class
of the Milnor boundary $\partial^M_X(g)$. In order to do this, we consider in general
a closed point $z \in \ov{W}$. 
If $z \in \ov{W} \cap F^1_n$, then we must have $g_i(z) = 1$ for some $1 \le i \le n+1$.
This means that $g_i -1 \in \fm_z$. \lemref{lem:Milnor-1} then implies that $\partial^M_z(g) = 0$.
If $z \in W \setminus \partial^{\{0, \infty\}}_i W$ for all $1 \le i \le n+1$, then we must have
$g_i \in \sO^{\times}_{\ov{W}, z}$ for all $1 \le i \le n+1$ and hence $\partial^M_z(g) = 0$.
If $z \in \partial^0_i W$ for some $1 \le i \le n+1$, then $z \notin \partial^t_j W$ unless 
$(t, j) = (0,i)$.
This implies that $g_j \in \sO^{\times}_{\ov{W}, z}$ for all $j \neq i$.
Furthermore, the image of $g_j$ under the map $\sO^{\times}_{\ov{W}, z} \surj k(z)^{\times}$ 
is simply $y_j(z)$. By the definition of the boundary map in the Gersten complex for the 
Milnor $K$-theory (e.g., see \cite{BT}), we therefore have
\[
\partial^M_z(g) = (-1)^i \ord_z(g_i) \{y_1(z), \ldots , \wh{y_i(z)}, 
\ldots , y_{n+1}(z)\} \in K^M_n(k(z)).
\]
We have the same expression for $\partial^M_z(g)$ if $z \in \partial^\infty_i W$ for some 
$1 \le i \le n+1$.

If we identify $\ov{W}$ with $X$ via $f$ so that $W \subseteq X^o = X \setminus D$,
it follows from the above computation of
$\{\partial^M_z(g)| z \in \ov{W}^{(1)}\}$ and the comparison of ~\eqref{eqn:Curve-vanish-0} and
\lemref{lem:Curve-boundary} that the two things hold. Namely,
\begin{enumerate}
\item
The image of $g$ under the Milnor boundary 
$K^M_{n+1}(F) \xrightarrow{\partial^M_{X}} {\underset{z \in X}\oplus} K^M_n(k(z))$
lies in the subgroup ${\underset{z \in \partial W}\oplus} K^M_n(k(z)) \subset
{\underset{z \in X^o}\oplus} K^M_n(k(z))$.
\item  
The element $\partial^M_X(g)$ maps to $cyc_{X|D}(\partial W)$ under the composition of maps
\begin{equation}\label{eqn:Curve-vanish-1}
{\underset{z \in X^o}\oplus} K^M_n(k(z)) \xrightarrow{(\psi_z)_z}
{\underset{z \in X^o}\oplus} K_n(k(z)) \xrightarrow{(f^z_*)_z} K_n(X,D).
\end{equation}
\end{enumerate}
The proposition is therefore reduced to showing that 
${\underset{z \in X^o}\sum} f^z_* \circ \psi_z \circ \partial^M_{X} (g) = 0$
if $W$ lies in $z^{n+1}(X|(n+1)D, n+1)$.
But this follows at once from \lemref{lem:Curve-Milnor}.
$\hfill\square$

\section{Proof of \thmref{thm:Intro-1}}\label{sec:Prf-1}
In this section, we shall prove \thmref{thm:Intro-1} using the case of 
regular curves.
So let $k$ be any field. Let $X$ be a regular quasi-projective variety of 
pure dimension $d \ge 1$
over $k$ and let $D \subset X$ be an effective Cartier divisor.
We fix an integer $n \ge 0$.
In ~\S~\ref{sec:Generators}, we constructed the cycle class map
\[
cyc_{X|D} \colon z^{d+n}(X|D, n) \to K_n(X,D).
\]
The naturality statements in \thmref{thm:Intro-1} follow from  \lemref{lem:Push-forward} and \lemref{lem:flat-pull-back}. To prove \thmref{thm:Intro-1}, it therefore suffices to show the following.

\begin{prop}\label{prop:Zero-boundary}
Let $W \subset z^{d+n}(X|(n+1)D, n+1)$ be an integral cycle.
Then the image of $W$ under the composition 
\[
z^{d+n}(X|(n+1)D, n+1) \inj z^{d+n}(X|D, n+1)  \xrightarrow{\partial} z^{d+n}(X|D, n) 
\xrightarrow{cyc_{X|D}} K_n(X,D)
\]
is zero.
\end{prop}

\vskip .3cm

We shall prove this proposition in several steps.
Let $\ov{W} \subset X \times \ov{\square}^{n+1}$ be the closure of $W$ and let 
$\nu \colon \ov{W}^N \to X \times \ov{\square}^{n+1}$ be the map induced on the 
normalization of $\ov{W}$.       
We begin with a direct proof of one easy case of the proposition 
as a motivating step.

\begin{lem}\label{lem:Easy-case}
Suppose that $W$ lies over a closed point of $X$. 
Then the assertion of \propref{prop:Zero-boundary} holds.
\end{lem}
\begin{proof}
In this case,  the modulus condition implies that
such a closed point must lie in $X^o$. In other words, there is a closed point
$x \in X^o$ such that $W \in z^n(\Spec(k(x)), n+1) \subset z^{d+n}(X|D, n+1)$.
Using the commutative diagram (see the construction of $cyc_{X|D}$ in \S~\ref{sec:Generators})
\begin{equation}\label{eqn:Zero-boundary-0}
\xymatrix@C.8pc{
z^n(\Spec(k(x)), n+1) \ar[r]^-{\partial} \ar[d] & 
z^n(\Spec(k(x)), n) \ar[r]^-{cyc_x} \ar[d] & K_n(k(x)) \ar[d] \\
z^{d+n}(X|D, n+1) \ar[r]^-{\partial} & z^{d+n}(X|D, n) \ar[r]^-{cyc_{X|D}} & K_n(X,D),}
\end{equation}
it suffices to show that $cyc_x(\partial W) = 0$ in $K_n(k(x))$.
We can thus assume that $X = \Spec(k)$ and $D = \emptyset$.

We let $F$ denote the function field of $W$ and let 
$g = \{g_1, \ldots , g_{n+1}\} \in K^M_{n+1}(F)$
denote the Milnor symbol given by the projection maps 
$g_i \colon \ov{W} \to \ov{\square}$.
Following the proof of \propref{prop:Curve-vanish}, our assertion is equivalent to 
showing that the composition
\[
K^M_{n+1}(F) \xrightarrow{\partial^M_{\ov{W}}} 
{\underset{z \in \ov{W}^{(1)}}\oplus} K^M_n(k(z))
\to {\underset{z \in \ov{W}^{(1)}}\oplus} K_n(k(z)) \xrightarrow{(f^z_*)_z} K_n(k)
\]
kills $g$. Arguing as in ~\eqref{eqn:CV-ex}, it suffices to show that
the composite map
\[
K_{n+1}(F) \xrightarrow{\partial^Q_{\ov{W}}} 
 {\underset{z \in \ov{W}^{(1)}}\oplus} K_n(k(z)) \xrightarrow{(f^z_*)_z} K_n(k)
\]
kills $g$. Since this map is same as the composite map
\[
K_{n+1}(F) \xrightarrow{\partial^Q_{\ov{W}^N}} 
 {\underset{z \in {(\ov{W}^N)}^{(1)}}\oplus} K_n(k(z)) \xrightarrow{(f^z_*)_z} K_n(k),
\]
we need to show that $g$ dies under this map. But this is the well known 
Weil reciprocity theorem in algebraic $K$-theory 
(e.g., see \cite[Chapter~IV, Theorem~6.12.1]{Weibel-1}).
\end{proof}

We now proceed to the proof of the general case of \propref{prop:Zero-boundary}.
In view of \lemref{lem:Easy-case}, we can assume that the projection map
$f: \ov{W}^N \to X$ is finite.

This gives rise to a Cartesian square
\begin{equation}\label{eqn:Cyc-mod-*2}
\xymatrix@C.8pc{
\ov{W}^N \ar@{^{(}->}[dr]^{\phi} \ar@/_2pc/[ddr]_{{\rm id}} \ar@/^1pc/[drr] & & \\
& \ov{W}^N \times \ov{\square}^{n+1} \ar[r] \ar[d]_{p'} &
X \times \ov{\square}^{n+1} \ar[d]^{p_X} \\
& \ov{W}^N \ar[r]_{f} & X.}
\end{equation}

Note that $\phi$ is a closed immersion. Using the finiteness of $f$ and 
admissibility of $W$,
it is evident that that 
$W^N = \phi(\ov{W}^N) \cap (\ov{W}^N \times \square^{n+1})$
is an admissible cycle on $\ov{W}^N \times \square^{n+1}$.
In other words, it intersects the faces of $\ov{W}^N \times \square^{n+1}$ properly,
and satisfies the modulus $(n+1)E$, if we let $E = f^*(D) \subsetneq \ov{W}^N$. 
Notice that a consequence
of the modulus condition for $W$ is that $E$ is a proper Cartier divisor on $\ov{W}^N$.
Since $f$ is finite and $E = f^*(D)$, we have a push-forward map
$f_* \colon z^{n+1}(\ov{W}^N|(n+1)E, *) \to z^{n+d}(X|(n+1)D, *)$ 
(see \cite[Proposition~2.10]{KP}).
Since $(f \times {\rm id}_{\square^n})$ takes $W^N$ to $W$ and since $W^N \to W$ 
is the normalization map, we see that $f_*([W^N]) = [W] \in z^{n+d}(X|(n+1)D, n+1)$.
In particular, we get
\begin{equation}\label{eqn:PF-boundary}
  f_*(\partial W^N) = \partial (f_*([W^N])) = \partial W.
  \end{equation}

\vskip.3cm

{\bf{Proof of \propref{prop:Zero-boundary}:}}
In view of \lemref{lem:Easy-case}, we can assume that $f \colon \ov{W}^N \to X$ is finite.
Using ~\eqref{eqn:PF-boundary} and \lemref{lem:Push-forward}, we have that 
\[
cyc_{X|D} (\partial W) = cyc_{X|D} \circ   f_*(\partial W^N)  = f_* \circ cyc_{\ov{W}^N|E}(\partial W^N).
\]

We can therefore  assume that $X$ is a regular curve 
and $W \in z^{n+1}(X|(n+1)D, n+1)$ is an integral cycle such that the map $p_X: \ov{W} \to X$ is an 
isomorphism.
We can now apply \propref{prop:Curve-vanish} to finish the proof.
$\hfill\square$

\begin{remk}\label{remk:Fin-tor-0}
  We remark that throughout the proof of \thmref{thm:Intro-1}, it is only in 
\lemref{lem:Fin-tor-K-thy} where
  we need to assume that $X$ is regular everywhere (see \remref{rem:Fin-tor-K-thy-1}). One would like to 
believe that
  for a proper map of modulus pairs $f \colon (Y, f^*(D)) \to (X,D)$, 
there exists a push-forward map
  $f_* \colon K_*(Y, f^*(D)) \to K(X,D)$. But we do not know how to prove it.
  \end{remk}

\subsection{The cycle class map with rational coefficients}\label{sec:Ratl}
If we work with rational coefficients, we can prove the following improved version of
\thmref{thm:Intro-1}. This may not be very useful in positive characteristic.
However, one expects it to have many consequences in characteristic zero. The reason for
this is that the relative algebraic $K$-groups of nilpotent ideals are known be
$\Q$-vector spaces in characteristic zero.
The proofs of \thmref{thm:Intro-2} and its corollaries in this paper are crucially based on
this improved version.

\begin{thm}\label{thm:Intro-1-ratl}
Let $X$ be a regular quasi-projective variety of pure dimension $d \ge 1$ over a field $k$
and let $D \subset X$ be an effective Cartier divisor. Let $n \ge 0$ be an integer.
Then there is a cycle class map
\begin{equation}\label{eqn:Intro-1-ratl-0}
cyc_{X|D} \colon  \CH^{n+d}(X|D,n)_{\Q} \to K_n(X, D)_{\Q}.
\end{equation}
\end{thm}
\begin{proof}
  We shall only indicate where do we use rational coefficients in the proof of
  \thmref{thm:Intro-1} to achieve this improvement as rest of the proof is
  just a repetition. Since we work with rational coefficients, we shall
  ignore the subscript $A_{\Q}$ in an abelian group $A$ in this proof and treat $A$ as a
  $\Q$-vector space.

  As we did before, we need to prove \propref{prop:Zero-boundary} with
  $W \in z^{d+n}(X|D, n+1)$. 
We can again reduce the proof of this proposition to the case when $X$ is a regular curve 
and $W \in z^{n+1}(X|D, n+1)$ is an integral cycle such that the map $p_X: \ov{W} \to X$ is an 
isomorphism. We thus have to prove \propref{prop:Curve-vanish} with
$W \in z^{n+1}(X|D, n+1)$. In turn, this is reduced to proving \lemref{lem:Curve-Milnor}
when $W \in z^{n+1}(X|D, n+1)$. However, a close inspection shows that the proof of
\lemref{lem:Curve-Milnor} works in the present case too with no change until its
last step where we need to use \lemref{lem:Milnor-2-Q-coeff} instead of \lemref{lem:Milnor-2}.
\end{proof}

\subsection{Chow groups and $K$-theory with compact support}\label{sec:Comact-supp}
Let $X$ be a quasi-projective scheme of pure dimension $d$ over a field $k$ 
and let $\ov{X}$ be
a proper compactification of $X$ such that $\ov{X} \setminus X$ is supported 
on an effective Cartier divisor $D$. Recall from \cite[Lemma~2.9]{BS} that
$\CH^p(X,n)_c := {\underset{m \ge 1}\varprojlim} \CH^p(\ov{X}|mD,n)$ is independent 
of the choice of  $\ov{X}$
and is called the higher Chow group of $X$ with compact support.
One can similarly define the algebraic $K$-theory with compact support by
$K_n(X)_c := {\underset{m \ge 1}\varprojlim} K_n(\ov{X}, mD)$.
It follows from \cite[Theorem~A]{KST} that this is also independent of the choice 
of $\ov{X}$. As a consequence of \thmref{thm:Intro-1}, we get

\begin{cor}\label{cor:Compact-cycle}
  Let $X$ be a regular quasi-projective scheme of pure dimension $d$ over a field 
admitting resolution of
  singularities. Then there exists a cycle class map
  \[
    cyc_X \colon \CH^{n+d}(X, n)_c \to K_n(X)_c.
  \]
  \end{cor}

  We remark that even if $\CH^{n+d}(X, n)_c$ is defined without resolution of
  singularities, this condition is needed in \corref{cor:Compact-cycle} because
  the usage of \thmref{thm:Intro-1} requires that $X$ admits regular compactifications
  (see \remref{remk:Fin-tor-0}).

  \section{Milnor $K$-theory, 0-cycles and
  de Rham-Witt complex}\label{sec:char-0}
Our next goal is to show that the cycle class map of \thmref{thm:Intro-1} 
completely describes the relative $K$-theory of truncated polynomial
rings in terms of additive 0-cycles in characteristic zero.
We shall give a precise formulation of our main result for fields in
\S~\ref{sec:char-0-*} and for semi-local rings in \S~\ref{sec:local}.
In this section, we prove some results on the
connection between {\sl a priori} three different objects: the additive 0-cycles, 
the relative Milnor $K$-theory and the de Rham-Witt forms.
These results will form one of the two keys steps in showing that
the cycle class map $cyc_k$ (see ~\eqref{eqn:Add-1}) factors through the
Milnor $K$-theory in characteristic zero.

\subsection{The additive 0-cycles}\label{sec:Add-0}
To set up the notations, let $k$ be a field of any characteristic. Let $m \ge 0$ 
be an integer.
Recall that the additive higher Chow groups $\TCH^p(X, *;m)$ 
of $X \in \Sch_k$ with modulus $m$ are defined so that there are
canonical isomorphisms
$\Tz^p(X, n+1;m) \ {\cong}  \ z^p(X \times \A^1_k|X \times (m+1)\{0\}, n)$ and

\begin{equation}\label{eqn:Add-0-*}
\TCH^p(X, n+1;m) \xrightarrow{\cong} \CH^p(X \times \A^1_k|X \times (m+1)\{0\}, n),
\end{equation}
where the term on the right are the Chow groups with modulus defined in 
\S~\ref{sec:HCGM}.
Using a similar isomorphism between the relative $K$-groups,  
\thmref{thm:Intro-1} provides a commutative diagram of pro-abelian groups
\begin{equation}\label{eqn:Add-1}
\xymatrix@C.8pc{
\{\TCH^{d+n+1}(X, n+1;m)\}_m \ar[r]^-{cyc_X} \ar[d]_-{\cong} &  
\{K_{n+1}({X[t]}/{(t^{m+1})}, (t))\}_m 
\ar[d]^-{\cong} \ar[d]_-{\partial} \\
\{\CH^{d+1+n}(X \times \A^1_k|X \times (m+1)\{0\}, n)\}_m 
\ar[r]^-{cyc_{{\A^1_X}|X \times \{0\}}} &
\{K_n(X \times \A^1_k, X \times (m+1)\{0\})\}_m}
\end{equation}
for an equi-dimensional regular scheme $X$ of dimension $d$ and integer $n \ge 0$.

\subsection{Connection with de Rham-Witt complex}\label{sec:DRham}
Let $k$ be a field with ${\rm char}(k) \neq 2$.
Let $R$ be a regular semi-local ring which is essentially of finite type
over $k$. Let $m, n \ge 1$ be two integers. Let $\W_m\Omega^*_R$ be the big 
de Rham-Witt complex of Hesselholt and Madsen (see \cite[\S~1]{R}).
We shall let $\un{a} = (a_1, \ldots , a_m)$ denote a general element of $\W_m(R)$.
Recall from \cite[Appendix]{R} that there is an isomorphism of
abelian groups $\gamma \colon \W(R) \xrightarrow{\cong} (1 + tR[[t]])^{\times}$
(with respect to addition in $\W(R)$ and multiplication in $R[[t]]$)
such that $\gamma(\un{a}) = \gamma((a_1, \ldots )) =
\stackrel{\infty}{\underset{i = 1}\prod} (1- a_it^i)$.
This map sends ${\rm Ker}(\W(R) \surj \W_m(R))$ isomorphically onto the subgroup
$(1 + t^{m+1}k[[t]])^{\times}$ and hence there is a canonical isomorphism of abelian groups
\begin{equation}\label{eqn:Witt-ps}
  \gamma_m \colon \W_m(R) \xrightarrow{\cong} \frac{(1 + tR[[t]])^{\times}}{(1 + t^{m+1}R[[t]])^{\times}}.
\end{equation}
Under this isomorphism, the Verschiebung map $V_r \colon \W_m(R) \to \W_{mr +r -1}(R)$
corresponds to the map on the unit groups induced by the $R$-algebra homomorphism
${R[[t]]}/{(t^{m+1})} \to {R[[t]]}/{(t^{rm+r})}$ under which $t \mapsto t^r$.
Recall also that there is a restriction ring homomorphism
$\xi^R_0 \colon \W_{m+1}(R) \surj \W_m(R)$ as part of the Witt-complex structure.
We shall often use the notation $V_r$ also for the composition
$\W_m(R) \xrightarrow{V_r} \W_{mr + r-1}(R) \stackrel{(\xi^R_0)^{r-1}}{\surj} \W_{mr}(R)$.
With this interpretation of the Verschiebung map, every element
$\un{a} \in \W_m(R)$ has a unique presentation
\begin{equation}\label{eqn:Sum-Witt}
  \un{a} = \stackrel{m}{\underset{i =1}\sum} V_i([a_i]_{ \lfloor{m/i}\rfloor}),
  \end{equation}
  where for a real number $x \in \R_{\ge 0}$, one writes
    $\lfloor x \rfloor$ for the greatest integer not bigger than $x$ 
  and $[ \ ]_{\lfloor x \rfloor} \colon R \to \W_{\lfloor x \rfloor}(R)$ for the Teichm{\"u}ller map
  $[a]_{\lfloor x \rfloor} = (a, 0, \ldots , 0)$.

It was shown in \cite[Theorem~7.10]{KP-4} that $\{\TCH^*(R,*;m)\}_{m \ge 1}$
is a pro-differential graded algebra which has the structure of a
restricted Witt-complex over $R$ in the sense of \cite[Definition~1.14]{R}.
Using the universal property of $\{\W_m\Omega^*_R\}_{m \ge 1}$ as the 
universal restricted Witt-complex over $R$, one gets a functorial morphism of
restricted Witt-complexes
\begin{equation}\label{eqn:Witt-Chow}
  \tau^R_{n,m} \colon \W_m \Omega^{n-1}_R \to \TCH^n(R, n ; m).
\end{equation}

It was shown in \cite[Theorem~1.0.2]{KP-2} that this map is an isomorphism.
When $R$ is a field, this isomorphism was shown earlier by R{\"u}lling \cite{R}.
We shall use this isomorphism throughout the remaining part of this paper
and consequently, will usually make no distinction between the source and target of this map.

\subsection{Connection with Milnor  $K$-theory}\label{sec:Milnor*}
Continuing with the above notations, we have another set of maps
\begin{equation}\label{eqn:Milnor-Chow-TCH}
  \xymatrix@C.8pc{
& \TCH^1(R,1;m)  \otimes \CH^{n-1}(R, n-1) \ar[dr]^-{\psi^R_{n,m}} & \\   
\W_m(R) \otimes K^M_{n-1}(R) \ar[ur]^-{\tau^R_{1,m} \otimes \nu^R_{n-1}} & & 
\TCH^n(R,n;m).}
\end{equation}

Here,
$\nu^R_n \colon K^M_{n}(R) \to \CH^n(R,n)$ is the semi-local ring 
analog of the Milnor-Chow homomorphism of Totaro \cite{Totaro}.
It takes a Milnor symbol $\{b_1, \ldots, b_{n-1}\}$ to the graph of
the function $(b_1, \ldots , b_{n-1}) \colon \Spec(R) \to \square^{n-1}$.
A combination of the main results \cite{EM} and \cite{Kerz09} implies that this
map is an isomorphism.
The map $\psi^R_{n,m}$ is given by the action of higher Chow groups on the 
additive higher Chow groups, shown in \cite{KLevine}. It takes cycles
$\alpha \in \TCH^i(R,n;m)$ and $\beta \in \CH^{j}(R, n')$ to
$\Delta^*_R(\alpha \times \beta) \in \TCH^{i+j}(R, n+n';m)$, where
$\Delta_R \colon \Spec(R) \times \A^1_k \times \square^{n+n'-1} \to \Spec(R) \times \Spec(R)
\times \A^1_k\times \square^{n+n'-1}$
is the diagonal on $\Spec(R)$ and identity on $\A^1_k \times \square^{n+n'-1}$.

\begin{lem}\label{lem:Milnor-Chow-TCH-0}
  For $\un{a} = (a_1, \ldots , a_m) \in \W_m(R)$ and
  $\un{b} = \{b_1, \ldots , b_{n-1}\} \in K^M_{n-1}(R)$, one has
$\psi^R_{n,m} \circ (\tau^R_{1,m} \otimes \nu^R_{n-1})(\un{a} \otimes \un{b}) = [Z]$,
where $Z \subset \Spec(R) \times \A^1_k \times \square^{n-1} \cong \A^1_R \times_R \square^{n-1}_R$
is the closed subscheme given by
\begin{equation}\label{eqn:Milnor-Chow-TCH-1}
Z = \{(t, y_1, \ldots, y_n)| \stackrel{m}{\underset{i =1}\prod} 
(1-a_it^i) = y_1-b_1 = \cdots = y_{n-1} - b_{n-1} = 0\}.
\end{equation}
\end{lem}
\begin{proof}
  We let $f(t) = \stackrel{m}{\underset{i =1}\prod} (1-a_it^i)$.
  In view of the description of $\psi^R_{n,m}$, we only have to show that
  $\tau^R_{1,m}(\un{a}) = V(f(t))$.
  But this is a part of the definition of the restricted
  Witt-complex structure on $\{\TCH^*(R,*;m)\}_{m \ge 1}$
over $R$  (see \cite[Proposition~7.6]{KP-4}).
  \end{proof}

Let $\phi^R_{n,m} \colon \W_m(R) \otimes K^M_{n-1}(R) \to \W_m\Omega^{n-1}_R$ be the
  unique map such that $\tau^R_{n,m} \circ \phi^R_{n,m} =
  \psi^R_{n,m} \circ (\tau^R_{1,m} \otimes \nu^R_{n-1})$.
  The following lemma describes the map $\phi^R_{n,m}$.

 \begin{lem}\label{lem:Milnor-Witt-desc}
   For any  $\un{a} \in \W_m(R)$ and $\un{b} = \{b_1, \ldots , b_{n-1}\} \in K^M_{n-1}(R)$, we have
   \[
     \phi^R_{n,m}(\un{a} \otimes \un{b}) = \un{a} d\log([b_1]) \wedge \cdots \wedge d\log([b_{n-1}]).
   \]
 \end{lem}
 \begin{proof}
For an ideal $I = (f_1, \ldots , f_r) \subset R[t, y_1, \ldots , y_{n-1}]$,
we let $Z(f_1, \ldots, f_r)$ denote the closed subscheme of
$\Spec(R[t, y_1, \ldots , y_{n-1}])$ defined by $I$.
We let $\un{a} \in \W_m(R)$ and $\un{b} = \{b_1, \ldots , b_{n-1}\} \in K^M_{n-1}(R)$.
We write $b = b_1 \cdots b_{n-1} \in R^{\times}$.
Then we have by ~\eqref{eqn:Sum-Witt},
\begin{equation}\label{eqn:Milnor-surj*-1}
  \tau^R_{n,m}(\un{a} d\log([b_1]) \wedge \cdots \wedge d\log([b_{n-1}]))  = \hspace*{7cm}
\end{equation}
\[
  \hspace*{6cm}
\stackrel{m}{\underset{i = 1}\sum}  \tau^R_{n,m}(V_i([a_i]_{\lfloor m/i \rfloor})
d\log([b_1]) \wedge \cdots \wedge d\log([b_{n-1}])).
\]

We now recall that $\tau^R_{n,m}$ is a part of the morphism of restricted Witt-complexes.
In particular, we have for each $1 \le i \le m$
\begin{equation}\label{eqn:Milnor-surj*-2}
  \tau^R_{n,m}(V_i([a_i]_{\lfloor m/i \rfloor})d\log([b_1]) \wedge \cdots \wedge d\log([b_{n-1}]))
  = \hspace*{5cm}
  \end{equation}
  \[
    \hspace*{4cm}  V_i(Z(1 - a_i t))(\stackrel{n-1}{\underset{j =1}\prod} Z(1-b^{-1}_it))
(d(Z(1 - b_1t)) \wedge \cdots \wedge d(Z(1 - b_{n-1}t)))
\]
\[
\hspace*{4cm} {=}^1 Z(1 - a_i t^i)(\stackrel{n-1}{\underset{j =1}\prod} Z(1-b^{-1}_it))
(d(Z(1 - b_1t)) \wedge \cdots \wedge d(Z(1 - b_{n-1}t))),
\]
\[
  \hspace*{4cm} {=}^2 Z(1 - a_i t^i)Z(1 - b^{-1}t)
  (d(Z(1 - b_1t)) \wedge \cdots \wedge d(Z(1 - b_{n-1}t))),
  \]
where ${=}^1$ follows from the fact the Verschiebung map on the
additive higher Chow groups
is induced by the pull-back through the power map
$\pi_r \colon \A^1_R \to \A^1_R$, given
by $\pi_r(t) = t^r$ (see \cite[\S~6]{KP-4}).
The equality ${=}^2$ follows from the
fact that the product in $\TCH^1(R, 1;m)$ is induced by the
multiplication map $\mu \colon \A^1_R \times_R \A^1_R \to \A^1_R$
(see \cite[\S~6]{KP-4}).

Since the differential of the additive higher Chow groups is induced by the
anti-diagonal map $(t, \un{y}) \mapsto (t, t^{-1}, \un{y})$, we see that
$d(Z(1 - b_it)) = Z(1 -b_it, y_i - b_i)$.
In particular, we get
\[
  d(Z(1 - b_1t)) \wedge \cdots \wedge d(Z(1 - b_{n-1}t)) =
  Z(1 - bt, y_1 - b_1, \ldots , y_{n-1} - b_{n-1}).
\]
As $Z(1 - b^{-1}t) \cdot Z(1 - bt) = Z(1-t) = \tau^R_{1,m}(1)$
is the identity element for the differential graded algebra structure on $\TCH^*(R,*;m)$, 
we therefore get
\[
  \tau^R_{n,m}(V_i([a_i]_{\lfloor m/i \rfloor})d\log([b_1]) \wedge \cdots \wedge d\log([b_{n-1}])) 
= Z(1 - a_i t^i, y_1 - b_1, \ldots , y_{n-1} - b_{n-1}).
\]
Combining this with ~\eqref{eqn:Milnor-surj*-1}, we get
\begin{equation}\label{eqn:Milnor-surj*-3}  
 \tau^R_{n,m}(\un{a} d\log([b_1]) \wedge \cdots \wedge d\log([b_{n-1}]))
  = \hspace*{7cm}
  \end{equation}
  \[
    \hspace*{7cm}                                                                                    
    Z(\stackrel{m}{\underset{i =1}\prod} (1 - a_it^i), y_1 - b_1, \ldots , y_{n-1} - b_{n-1}).
  \]
Since $\tau^R_{n,m}$ is an isomorphism, we now conclude the proof by applying
\lemref{lem:Milnor-Chow-TCH-0}.
\end{proof}

  \subsection{Additive 0-cycles in characteristic zero}\label{sec:Gen-0}
 We  shall assume in this subsection that the
  base field $k$ has characteristic zero.
  As above, we let $R$ be a regular semi-local ring  which is essentially of finite type
  over $k$ and $m,n \ge 1$ two integers. For any $\un{b} = \{b_1, \ldots , b_{n-1}\}
  \in K^M_{n-1}(R)$, we let $b = b_1 \cdots b_{n-1} \in R^{\times}$.
  Under our assumption on ${\rm char}(k)$, we can prove the following result
  which is the first key step for showing that
the cycle class map $cyc_k$ (see ~\eqref{eqn:Add-1}) factors through the
Milnor $K$-theory in characteristic zero.

  \begin{lem}\label{lem:Milnor-surj}
    The map
    \[
      \psi^R_{n,m} \circ (\tau^R_{1,m} \otimes \nu^R_{n-1}) \colon
    \W_m(R) \otimes K^M_{n-1}(R) \to \TCH^n(R,n;m)
  \]
  is surjective. In particular, $\psi^R_{n,m}$ is surjective.
\end{lem}
\begin{proof}
  Let $\phi^R_{n,m} \colon \W_m(R) \otimes K^M_{n-1}(R) \to \W_m\Omega^{n-1}_R$ be the
  unique map such that $\tau^R_{n,m} \circ \phi^R_{n,m} =
  \psi^R_{n,m} \circ (\tau^R_{1,m} \otimes \nu^R_{n-1})$.
  The lemma is then equivalent to showing that $\phi^R_{n,m}$ is surjective.
  
  For $\un{a} \in \W_m(R)$ and $\un{b} = \{b_1, \cdots , b_{n-1}\} \in K^M_{n-1}(R)$,
 it follows from \lemref{lem:Milnor-Witt-desc} that
${\phi}^R_{n,m}(\un{a} \otimes \un{b}) = \un{a} d\log([b_1])
\wedge \cdots \wedge d\log([b_{n-1}])$.
We shall now use that the ground field has characteristic zero.
Let $p >m$ be a prime. It then  follows from
\cite[Theorem~1.11, Remark~1.12]{R} that there is a canonical (Ghost)
isomorphism
\begin{equation}\label{eqn:Milnor-surj*-22}
  \zeta^R_{r,m} \colon \W_m\Omega^{r}_R \xrightarrow{\cong} \stackrel{m}{\underset{i =1}\prod}
  \Omega^r_R
\end{equation}
such that 
\[
 \zeta^R_{r,m}(x dy_1\cdots dy_r) = \left(\frac{1}{j^{r}}F_j(x) dF_j (y_1) \cdots dF_j(y_r)\right)_{1\leq j \leq m},
\]
where $x, y_i \in \W_m(R)$ and $F_j(y_i)$ means its restriction to $\W_1\Omega^r_R$ via the restriction
map of the de Rham-Witt complex. In particular, $\zeta^R_{0,m}(x) = ({\underset{d|j}\sum}
\ dx^{j/d}_d)_{1 \le j \le m}$ is the classical Ghost map, where $x = (x_1, \ldots , x_m) \in \W_m(R)$.

It follows that the following diagram commutes: 
\begin{equation}
\xymatrix@C.8pc{
\W_m(R) \otimes K^M_{n-1}(R) \ar[r]^-{\phi^R_{n,m}} \ar[d]^-{\cong}_-{\zeta^R_{0,m} \otimes \id} & \W_m\Omega^{n-1}_R \ar[d]_-{\cong}^-{ \zeta^R_{n-1,m}}\\
(\stackrel{m}{\underset{i =1}\prod}
 R) \otimes K^M_{n-1}(R) \ar[r] &  \stackrel{m}{\underset{i =1}\prod} \Omega^{n-1}_R,
}
\end{equation}
where the bottom arrow is defined component-wise so that
 $(a_i )_{i} \otimes \{b_1, \dots, b_{n-1}\}$ maps to 
 $(\frac{1}{i^{n-1}}  a_i d\log(b_1) \wedge \cdots \wedge d\log(b_{n-1}))_i $. 
Since we are working with characteristic zero field, it then suffices to  show that the map $R \otimes K^M_{n-1}(R) \to \Omega^{n-1}_R$, given by
  $a \otimes \un{b} \mapsto ad\log(b_1) \wedge \cdots \wedge d\log(b_{n-1})$, is surjective.
By an iterative procedure, it suffices to prove this surjectivity when $n =2$.

We now let $a, b \in R$. By \lemref{lem:Local*} below, we can write $b = b_1 + b_2$, 
where $b_1, b_2 \in R^{\times}$.
We then get $a db = adb_1 + adb_2 = ab_1 d\log (b_1) + ab_2 d\log(b_2)$. 
Since $\Omega^1_R$ is generated by the universal derivations of the elements of $R$ as an $R$-module, 
we are done.

\end{proof}

\begin{lem}\label{lem:Local*}
  Let $R$ be a semi-local ring which contains an infinite field $k$. 
Then every element $a \in R$ can be
  written as $a = u_1 + u_2$, where $u_1\in k^{\times}$ and  $u_2 \in R^{\times}$.
\end{lem}
\begin{proof}
Let $M = \{\mathfrak{m}_1, \cdots , \mathfrak{m}_r\}$ denote the set of all 
maximal ideals of $R$.
Fix $a \in R$. Suppose that there exists $u \in k^{\times}
\subseteq R^{\times}$ such that  $a + u \in R^{\times}$. Then we are done.
Otherwise, every element $u \in k^{\times}$ has the property that
$a+u \in \mathfrak{m}_i$ for some $i$. Since $k$ is infinite and $M$ is
finite, there are two distinct elements $u_1, u_2 \in k^{\times}$ such that
$a+ u_1$ and $a+u_2$ both belong to a maximal ideal $\mathfrak{m}_j$.
Then we get $u_1 - u_2 \in \mathfrak{m}_j$. But $u_1 \neq u_2$ in $k$
implies that $u_1 - u_2 \in k^{\times}$ and this forces $\mathfrak{m}_j = R$,
a contradiction. We conclude that there must exist $u \in k^{\times}$ 
such that $a+u \in R^{\times}$.
\end{proof}

\section{The relative Milnor $K$-theory}\label{sec:Rel-Milnor**} 
In this section, we shall prove our second key step (see \lemref{lem:M-Q-C-0})
to show the factorization of the cycle class map through the relative Milnor
$K$-theory and prove the isomorphism of the resulting map.

\subsection{Recollection of relative Hochschild and cyclic homology}\label{sec:HC-*}
In the next two sections, we shall use Hochschild, Andr{\'e}-Quillen and cyclic homology of commutative 
rings as our tools. We refer to \cite{Loday} for their definitions and some properties that
we shall use. While using a specific result from \cite{Loday}, we shall mention the
exact reference.

Let $R$ be a commutative ring. For an integer $m \ge 0$, recall from \S~\ref{sec:Notations}
that the truncated polynomial algebra ${R[t]}/{(t^{m+1})}$ is denoted by $R_m$. 
Throughout our discussion of truncated polynomial algebras, we shall
make no distinction between the variable $t \in R[t]$ and its image in $R_m$.

If $k \subset R$ is a subring, we shall use the notation $HH_*(R|k)$  for
the Hochschild homology of $R$ over $k$. Similarly, $D^{(q)}_*(R|k)$ for $q \ge 0$ and $HC_*(R|k)$
will denote the Andr{\'e}-Quillen and cyclic homology of $R$ over $k$, respectively. When $k = \Z$,
we shall write $HH_*(R|k)$ simply as $HH_*(R)$. Similar notations will be used for
$D^{(q)}_*(R|\Z)$ and $HC_*(R|\Z)$. Note that $HH_*(R) \cong HH_*(R|\Q), \ 
D^{(q)}_*(R) \cong D^{(q)}_*(R|\Q)$ and
$HC_*(R) \cong HC_*(R|\Q)$ if $R$ contains $\Q$.
We also have $\Omega^q_{R}:= \Omega^q_{{R}/{\Z}} \cong \Omega^q_{{R}/{\Q}}$ for $q \ge 0$.

Recall that for an ideal $I \subset R$, the relative
Hochschild homology $HH_*((R,I)|k)$ is defined as the homology of the
complex ${\rm Ker}(HH(R|k) \surj HH(R/I|k))$, where $k \subset R$ is a subring and
$HH(R|k)$ is the Hochschild complex of $R$ over $k$.
The relative cyclic homology $HC_*((R,I)|k)$ is defined to be the
homology of the complex ${\rm Ker}(CC(R|k) \surj CC(R/I|k))$, where
$CC(R|k)$ is the total cyclic complex of $R$ over $k$. We refer to 
\cite[1.1.16, 2.1.15]{Loday} for these definitions. One defines
$D^{(q)}_*((R,I)|k)$ similarly. If $R$ is a commutative ring, we shall write
$HH_*((R_m, (t))|\Z), \ D^{(q)}_*((R_m, (t))|\Z)$ and $HC_*((R_m, (t))|\Z)$
simply as $\wt{HH}_*(R_m), \ \wt{D}^{(q)}_*(R_m)$ and $\wt{HC}_*(R_m)$, respectively.
We let $\wt{\Omega}^n_{R_m} = {\rm Ker}(\Omega^n_{R_m} \surj \Omega^n_R)$.
Recall from \S~\ref{sec:Rel-K} that $\wt{K}_*(R_m)$ denotes the
relative $K$-theory $K_*(R_m, (t))$. 
Suppose now that $R$ contains $\Q$.
For $x \in tR_m$, we shall write $\exp(x) = {\underset{i \ge 0} \sum} {x^i}/{i!}$
and $\log(1 + x) = {\underset{i \ge 1} \sum} (-1)^{i-1}{x^i}/{i}$.
Note that these are finite sums and define homomorphisms
\begin{equation}\label{eqn:exp-log}
  tR_m \xrightarrow{\exp} \wt{K}_1(R_m) \xrightarrow{\log} tR_m
\end{equation}
which are inverses to each other.

\subsection{Relative Milnor $K$-theory of truncated 
polynomial rings}\label{sec:Rel-K-trun}
Let $R$ be a semi-local ring and let $m \ge 0, n \ge 1$ be two integers.
We shall write the relative Milnor $K$-groups 
(see \S~\ref{sec:Milnor-K}) $K^M_*(R_m, (t))$ as $\wt{K}^M_*(R_m)$. 
Since $\wt{K}_*(R_m)$ is same as the kernel of the augmentation map
$K_*(R_m) \surj K_*(R)$, we have the canonical map
$\psi_{R_m} \colon  \wt{K}^M_*(R_m) \to \wt{K}_*(R_m)$.

Let us now assume that $R$ is a semi-local ring containing $\Q$.
Recall that there is a Dennis trace map ${\tr}^R_{m,n} \colon K_n(R_m) \to HH_n(R_m)$
which restricts to the dlog map on the Milnor $K$-theory
(e.g., see \cite[Example~2.1]{GW-1}).
Equivalently, there is a commutative diagram 
\begin{equation}\label{eqn:DTrace}
  \xymatrix@C.8pc{
    K^M_n(R_m) \ar[r]^-{d\log} \ar[d]_-{\psi_{R_m,n}} & \Omega^n_{R_m} \ar[d]^{\epsilon_n} \\
  K_n(R_m) \ar[r]^-{{\tr}^R_{m,n}} & HH_n(R_m),}
\end{equation}
where $\epsilon_n$ is the canonical anti-symmetrization map from K{\"a}hler differentials to
Hochschild homology (see \cite[\S~1.3.4]{Loday}).

A very well known result of Goodwillie \cite{Good} 
says that the relativization of the Dennis trace map with respect to a nilpotent ideal 
factors  through a trace map $\wt{K}_n(R_m) \to \wt{HC}_{n-1}(R_m)$ such that
its composition with the canonical Connes' periodicity map
$B \colon \wt{HC}_{n-1}(R_m) \to \wt{HH}_n(R_m)$ is the relative Dennis trace map.
This factorization is easily seen on $\wt{K}_1(R_m)$ via the chain of maps
\begin{equation}\label{eqn:trace-K1}
\wt{K}_1(R_m) \xrightarrow{\log} tR_m \cong \wt{HC}_0(R_m) \xrightarrow{d}
\wt{HH}_1(R_m) \cong \wt{\Omega}^1_{R_m}.
\end{equation}
Recall that Connes' periodicity map $B$ on $HC_0(R_m)$
coincides with the differential $d \colon R_m \to \Omega^1_{R_m}$ under the
isomorphisms $R_m \cong HC_0(R_m)$ and $\Omega^1_{R_m} \cong HH_1(R_m)$.
Goodwillie showed that his factorization $\wt{K}_n(R_m) \to \wt{HC}_{n-1}(R_m)$ is an isomorphism of
$\Q$-vector spaces. We shall denote this Goodwillie's isomorphism also by ${\tr}^R_{m,n}$.

Going further, Cathelineau showed (see \cite[Theorem~1]{Cath}) that the  
$K$-group $\wt{K}_n(R_m)$ and the relative cyclic homology group 
$\wt{HC}_{n-1}(R_m)$ are $\lambda$-rings. Furthermore, Goodwillie's
map is an isomorphism of $\lambda$-rings, thanks to \cite[Theorem~6.5.1]{CW}.
In particular, it induces an isomorphism between the Adams graded pieces
${\tr}^R_{m,n} \colon {\wt{K}^{(q)}_n(R_m)} \xrightarrow{\cong} {\wt{HC}^{(q-1)}_{n-1}(R_m)}$ 
for every $1 \le q \le n$.
As a corollary of this isomorphism and \cite[Theorem~4.6.8]{Loday}, 
we get the following.

\begin{lem}\label{lem:Adams}
The Dennis trace map induces an isomorphism of $\Q$-vector spaces
\[
{\tr}^R_{m,n} \colon {\wt{K}^{(n)}_n(R_m)} \xrightarrow{\cong}
\frac{\wt{\Omega}^{n-1}_{R_m}}{d \wt{\Omega}^{n-2}_{R_m}}.
\]
\end{lem}

In order to relate these groups with the Milnor $K$-theory, we first 
observe that
as the map $K^M_1(R_m) \to K_1(R_m)$ is an isomorphism, it follows from the 
properties of
$\gamma$-filtration associated to the $\lambda$-ring structure on $K$-theory that the 
canonical map $K^M_n(R_m) \to K_n(R_m)$ factors through 
$K^M_n(R_m) \to F^{n}_{\gamma}K_n(R_m)$, where $F^{\bullet}_{\gamma}K_n(R_m)$ denotes the 
$\gamma$-filtration.
If we consider the induced map on the relative $K$-groups, it follows that the
canonical map $\wt{K}^M_n(R_m) \to \wt{K}_n(R_m)$ factors as
\[
\psi_{R_m, n} \colon \wt{K}^M_n(R_m) \to {\wt{K}^{(n)}_n(R_m)} =  
F^{n}_{\gamma}\wt{K}_n(R_m) \inj \wt{K}_n(R_m).
\]

It follows from \cite[Theorem~12.3]{Steinstra} that ${\rm Ker}(\psi_{R_m,n})$ is a 
torsion group.
On the other hand, it follows from \cite[Proposition~5.4]{GT} that
$\wt{K}^M_n(R_m)$ is a $\Q$-vector space. If we now apply Soul{\'e}'s computation of
$F^{n}_{\gamma}\wt{K}_n(R_m)$ in \cite[Th{\'e}or{\'e}me~2]{Soule},
we conclude that the map $\psi_{R_m, n}$ is in fact an isomorphism. We have thus 
shown the following.

\begin{lem}\label{lem:M-Q-C}
The maps
\[
\wt{K}^M_n(R_m) \xrightarrow{\psi_{R_m, n}} {\wt{K}^{(n)}_n(R_m)} 
\xrightarrow{{\tr}^R_{m,n}} 
\frac{\wt{\Omega}^{n-1}_{R_m}}{d \wt{\Omega}^{n-2}_{R_m}} 
\]
are all isomorphisms of $\Q$-vector spaces.
\end{lem}

One knows from \cite[3.4.4]{GW-2} that the map
$d = B \colon \frac{\wt{\Omega}^{n-1}_{R_m}}{d \wt{\Omega}^{n-2}_{R_m}} \to \wt{\Omega}^n_{R_m}$
is injective.
Using the fact that ${\tr}^R_{m,n} \colon \wt{K}^M_n(R_m) \to \wt{\Omega}^n_{R_m}$ is 
multiplicative (e.g., see \cite[Property~1.3]{Kont} or \cite[8.4.12]{Loday})
and it is the usual logarithm on $\wt{K}^M_1(R_m)$ (see ~\eqref{eqn:trace-K1}),
it follows from ~\eqref{eqn:DTrace} that modulo $d \wt{\Omega}^{n-2}_{R_m}$, 
the composite map ${\tr}^R_{n,m} \circ \psi_{R_m, n}$ is 
the dlog map: 
\begin{equation}\label{eqn:dlog-map}
d\log(\{1- tf(t), b_1, \ldots , b_{n-1}\}) = \log(1 - tf(t)) d\log(b_1) \wedge 
\cdots \wedge d\log(b_{n-1}).
\end{equation}

\subsection{More refined structure on $\wt{K}^M_n(R_m)$}
\label{sec:R-mod}
We shall further simplify the presentation of $\wt{K}^M_n(R_m)$ in the next result.
This will be our second key step in factoring the cycle class map through the relative
Milnor $K$-theory and showing that the resulting map is an isomorphism.
For this, we assume $m \ge 1$ and look at the diagram
\begin{equation}\label{eqn:Milnor-omega}
\xymatrix@C.8pc{
tR_m \otimes K^M_{n-1}(R)  \ar[r] \ar[d]_-{{\rm id} \otimes d\log }
 & \wt{K}^M_n(R_m) \ar[dd]^-{{\tr}^R_{m,n} \circ \psi_{R_m, n}} \\
tR_m \otimes \Omega^{n-1}_R \ar@{->>}[d] & \\
 tR_m \otimes_R \Omega^{n-1}_R  \ar[r]^-{\theta_{R_m}^{n-1}} \ar@{.>}[ruu] & 
\frac{\wt{\Omega}^{n-1}_{R_m}}{d\wt{\Omega}^{n-2}_{R_m}},}
\end{equation}
where the top horizontal arrow is the product 
$a \otimes \{b_1, \ldots , b_{n-1}\} \mapsto \{\exp(a), b_1, \ldots , b_{n-1}\}$. 

The map $\theta_{R_m}^{n-1}$ in ~\eqref{eqn:Milnor-omega}
is the composition of the canonical map 
$tR_m \otimes_R \Omega^{n-1}_R  \to \wt{\Omega}^{n-1}_{R_m}$ (sending 
$t^i \otimes \omega$ to $t^i \omega$) with the surjection
$\wt{\Omega}^{n-1}_{R_m} \surj \frac{\wt{\Omega}^{n-1}_{R_m}}{d\wt{\Omega}^{n-2}_{R_m}}$.
Using ~\eqref{eqn:dlog-map}, it is easy to check that ~\eqref{eqn:Milnor-omega} is 
commutative.
It follows from \lemref{lem:M-Q-C} that $\theta_{R_m}^{n-1}$ factors through a 
unique map $\wt{\theta}^{n}_{R_m} \colon tR_m \otimes_R \Omega^{n-1}_R  \to \wt{K}^M_n(R_m)$. 
We want to show that this map is an isomorphism. Equivalently,
$\theta_{R_m}^{n-1}$ an isomorphism. Since the proof of this is a bit long, we 
prove that it is surjective and injective in separate lemmas.

\begin{lem}\label{lem:Surj-omega}
The map $\theta_{R_m}^{n}$ is surjective for all $n \ge 0$.
\end{lem}
\begin{proof}
This is obvious for $n =0$ and so we assume $n \ge 1$. We now consider the  
exact sequence
  \[
    R_m \otimes_R \Omega^1_R  \to \Omega^1_{R_m} \to \Omega^1_{{R_m}/R} \to 0.
  \]
  
  We claim that the first arrow in this sequence is split injective.
  For this, we consider the map
  $d': R_m \to R_m \otimes_R \Omega^1_R$, given by 
$d'(\stackrel{m}{\underset{i = 0}\sum} a_it^i) = 
  \stackrel{m}{\underset{i = 0}\sum} t^i \otimes d(a_i)$.
  The computation
  \[
    \begin{array}{lll}
      d'((\sum_i a_i t^i)(\sum_j b_j t^j)) & = & d'(\sum_{i,j} a_i b_j t^{i+j}) \\
& = & \sum_{i,j} t^{i+j} d(a_i b_j) \\
& = & \sum_{i,j} b_j t^{i+j} \otimes d(a_i) + \sum_{i, j} a_i t^{i+j} \otimes d(b_j)  \\
      & = &  (\sum_j b_j t^j) \otimes d'(\sum_i a_i t^i) + 
(\sum_i a_i t^i) \otimes d'(\sum_j b_j t^{j})
\end{array}
  \]
  shows that $d'$ is a $\Q$-linear derivation on the $R_m$-module 
$\Omega^1_R \otimes_R R_m$.
  Hence, it induces an $R_m$-linear map 
$u: \Omega^1_{R_m} \to  R_m \otimes_R \Omega^1_{R_m}$.
  Moreover, it is clear that the composite 
$R_m \otimes_R \Omega^1_R  \to \Omega^1_{R_m} \xrightarrow{u}
  R_m \otimes_R \Omega^1_R$ is identity. This proves the claim.
  We thus get a direct sum decomposition of $R_m$-modules
  $\Omega^1_{R_m} = (R_m \otimes_R \Omega^1_R) \oplus \Omega^1_{{R_m}/R}$.
  As $\Omega^n_{{R_m}/R} = 0$ for $n \ge 2$, we get
  $\Omega^n_{R_m} = (R_m \otimes \Omega^n_R) \oplus (\Omega^1_{{R_m}/R} \otimes_R \Omega^{n-1}_R)$ 
for any $n \ge 1$.
  This implies that 
  \begin{equation}\label{eqn:M-Q-C-0-0}
    \wt{\Omega}^n_{R_m} = (tR_m \otimes_R \Omega^n_R) \oplus 
(\Omega^1_{{R_m}/R} \otimes_R \Omega^{n-1}_R).
  \end{equation}
  
The other thing we need to observe is that the exact sequence
\[
{(t^{m+1})}/{(t^{2{m+2}})} \xrightarrow{d} \Omega^1_{{R[t]}/R} \otimes_{R[t]} R_m 
\to \Omega^1_{{R_m}/R} \to 0
\]
implies that $\Omega^1_{{R_m}/R} \cong {R_mdt}/{(t^{m})dt}$. 
In particular, $\Omega^1_{{R_m}/R} \otimes_R \Omega^{n-1}_R$ is generated as an 
$R$-module by elements of
the form $(\stackrel{m-1}{\underset{i =0}\sum} a_i t^i dt) \otimes \omega$,
where $a_i \in R$.

We now let $\omega \in \Omega^{n-1}_{R}$. We then get
  \begin{equation}\label{eqn:M-Q-C-0-1}
    \begin{array}{lll}
d(\stackrel{m-1}{\underset{i =0}\sum} \frac{a_i}{i+1} t^{i+1} \otimes \omega) & = &
\stackrel{m-1}{\underset{i =0}\sum} \frac{a_i}{i+1} t^{i+1} \otimes d \omega
+ \stackrel{m-1}{\underset{i =0}\sum} \frac{da_i}{i+1} t^{i+1} \otimes \omega+
\stackrel{m-1}{\underset{i =0}\sum} a_i t^i dt  \otimes \omega\\
& = & \stackrel{m-1}{\underset{i =0}\sum} t^{i+1} \otimes \frac{a_i d\omega + da_i \wedge \omega}{i+1} 
+ \stackrel{m-1}{\underset{i =0}\sum} a_i t^i dt \otimes \omega. 
\end{array}
  \end{equation}
 It follows from this that the composite map
$tR_m \otimes_R \Omega^n_R \inj  \wt{\Omega}^n_{R_m} \to 
\frac{\wt{\Omega}^{n}_{R_m}}{d\wt{\Omega}^{n-1}_{R_m}}$ is surjective.
\end{proof}

\begin{lem}\label{lem:M-Q-C-0}
For $m, n \ge 1$, the map
\[
\wt{\theta}^{n}_{R_m} \colon tR_m \otimes_R \Omega^{n-1}_R  \to \wt{K}^M_n(R_m); 
\]
\[
\wt{\theta}^{n}_{R_m}(a \otimes d\log(b_1) \wedge \cdots \wedge d\log(b_{n-1}))
=  \{\exp(a), b_1, \ldots , b_{n-1}\},
\]
is an isomorphism.
\end{lem}
\begin{proof}
 In view of Lemmas~\ref{lem:M-Q-C} and ~\ref{lem:Surj-omega}, 
 the assertion that $\wt{\theta}^{n}_{R_m}$ is an isomorphism for all $m, n \geq 1$ is equivalent to
 showing that the map
$\theta^n_{R_m} \colon tR_m \otimes_R \Omega^{n}_R \to 
\frac{\wt{\Omega}^{n}_{R_m}}{d\wt{\Omega}^{n-1}_{R_m}}$ (see ~\eqref{eqn:Milnor-omega}) is injective for all
$n \ge 0$ and $m \ge 1$.
We can again assume that $n \ge 1$. We shall prove this by induction on $m \ge 1$.
Assume first that $m =1$. In this case, we want to show that the map
$tR \otimes_R \Omega^n_R \to \frac{\wt{\Omega}^{n}_{R_1}}{d\wt{\Omega}^{n-1}_{R_1}}$ is injective.
To show this, we should first observe that every element of $tR \otimes_R \Omega_R^{n-1}$ 
must be of the form $t \otimes \omega$ and every element of $Rdt \otimes_R \Omega^{n-2}_R$
must be of the form $dt \otimes \omega'$.
In this case, we get
\begin{equation}\label{eqn:M-Q-C-0-2}
  \begin{array}{lll}
d(t \otimes \omega + dt \otimes \omega') & = & dt \wedge \omega + t d \omega - dt \wedge d \omega' \\
& = & t \otimes d \omega + dt \otimes (\omega - d \omega' ).
\end{array}
\end{equation}

If $d(t \otimes \omega + dt \otimes \omega') $ has to lie in $Rt \otimes_R \Omega^n_R$,
then we must have $dt \otimes ( \omega - d \omega') = 0$. But this implies that
$\omega = d\omega'$. Putting this in ~\eqref{eqn:M-Q-C-0-2},
we get $t \otimes d\omega  = 0$. This shows that  
$d\wt{\Omega}^{n-1}_{R_1} \cap (Rt \otimes_R \Omega^n_R) = 0$. But this is 
equivalent to the desired injectivity.

\vskip .2cm

To prove the $m \ge 2$ case by induction, we shall need the following

{\bf Claim:} The restriction map 
${\rm Ker}(\wt{\Omega}^{n-1}_{R_{m+1}} \xrightarrow{d} \wt{\Omega}^n_{R_{m+1}}) \to
{\rm Ker}(\wt{\Omega}^{n-1}_{R_{m}} \xrightarrow{d} \wt{\Omega}^n_{R_{m}})$ is surjective for 
every $m \ge 1$.

To prove the claim, recall from ~\eqref{eqn:M-Q-C-0-0} that
\begin{equation}\label{eqn:M-Q-C-3}
  \begin{array}{lll}
  \wt{\Omega}^{n-1}_{R_m} & \cong & 
(tR_m \otimes_R \Omega^{n-1}_R) \oplus ({R_m dt}/{(t^{m})dt} \otimes_R \Omega^{n-1}_R) \\
  & \cong & (\stackrel{m-1}{\underset{i =0}\oplus} Rt^{i+1} \otimes_R \Omega^{n-1}_R  \oplus
            (\stackrel{m-1}{\underset{i =0}\oplus} Rt^i dt \otimes_R \Omega^{n-2}_R).
            \end{array}
  \end{equation}

  Suppose now that $\omega \in \wt{\Omega}^{n-1}_{R_m}$ is such that $d \omega = 0$.
  It follows from ~\eqref{eqn:M-Q-C-3} that we can write
  $\omega = \stackrel{m-1}{\underset{i =0}\sum} (i+1)^{-1} t^{i+1}\omega_i +
  \stackrel{m-1}{\underset{i =0}\sum} t^i dt \wedge \omega'_i$.
  Therefore, we have
  \begin{equation}\label{eqn:M-Q-C-4}
    \begin{array}{lll}
      d \omega 
      & = & \stackrel{m-1}{\underset{i =0}\sum} (i+1)^{-1} t^{i+1} d \omega_i+
                    \stackrel{m-1}{\underset{i =0}\sum} t^{i} dt \wedge d \omega_i -
                    \stackrel{m-1}{\underset{i =0}\sum} t^i dt \wedge d \omega'_i\\
               & = & \stackrel{m-1}{\underset{i =0}\sum} (i+1)^{-1} t^{i+1} d \omega_i +
      \stackrel{m-1}{\underset{i =0}\sum}t^i dt \wedge  (\omega_i - d \omega'_i). \\
    \end{array}
    \end{equation}

    Since the left hand term is zero, it implies from the decomposition
    ~\eqref{eqn:M-Q-C-3} (for $\wt{\Omega}^n_{R_m}$) that
    $d\omega_i = 0 = \omega_i - d\omega'_i$ for $0 \le i \le m-1$.
    If we now let $\omega' =  \stackrel{m-1}{\underset{i =0}\sum} (i+1)^{-1} t^{i+1} \omega'_i  
\in \wt{\Omega}^{n-2}_{R_m}$, we get
 \begin{equation}\label{eqn:M-Q-C-5}
d \omega' =  \stackrel{m-1}{\underset{i =0}\sum} (i+1)^{-1} t^{i+1} d \omega'_i  +
\stackrel{m-1}{\underset{i =0} \sum} t^i dt \wedge \omega'_i = \omega.
\end{equation}

The claim now follows from the surjectivity of the map
$\wt{\Omega}^{n-2}_{R_{m+1}} \to \wt{\Omega}^{n-2}_{R_m}$. Indeed, this implies that
some $\wh{\omega} \in \wt{\Omega}^{n-2}_{R_{m+1}}$ maps onto $\omega'$
and therefore $d \wh{\omega} \in \wt{\Omega}^{n-1}_{R_{m+1}}$ maps onto $\omega$.
Since $d \wh{\omega}$ is clearly a closed form, we are done.

In order to use the above claim,  we let
$F^n_m = (Rt^m \otimes_R \Omega^n_R) \oplus (R t^{m-1} dt \otimes_R \Omega^{n-1}_R)$.
We then have an exact sequence of $R_{m+1}$-modules
\begin{equation}\label{eqn:M-Q-C-6}
0 \to F^n_{m+1} \to \wt{\Omega}^n_{R_{m+1}} \to \wt{\Omega}^n_{R_{m}} \to 0.
\end{equation}

Letting $m \ge 1$ and taking the quotient of this short exact sequence by the 
similar exact sequence for $n-1$ via
the differential map, we obtain a commutative diagram 
\begin{equation}\label{eqn:M-Q-C-0-4}
\xymatrix@C.8pc{
0 \ar[r] & Rt^{m+1} \otimes_R \Omega^n_R \ar[r] \ar[d] & tR_{m+1} \otimes_R \Omega^n_R  
\ar[r] \ar[d]^-{\theta^n_{R_{m+1}}} & tR_m \otimes_R \Omega^n_R 
\ar[r] \ar[d]^-{\theta^n_{R_m}} & 0 \\
0 \ar[r] & {F^n_{m+1}}/{d F^{n-1}_{m+1}} \ar[r] & 
\frac{\wt{\Omega}^{n}_{R_{m+1}}}{d\wt{\Omega}^{n-1}_{R_{m+1}}} \ar[r] &
\frac{\wt{\Omega}^{n}_{R_{m}}}{d\wt{\Omega}^{n-1}_{R_{m}}} \ar[r] & 0.}
\end{equation}

The top row is clearly exact and the above claim precisely says that the bottom 
sequence is exact.
The right vertical arrow is injective by induction. It suffices therefore to show 
that the map $Rt^{m+1} \otimes_R \Omega^n_R \otimes_R \to {F^n_{m+1}}/{d F^{n-1}_{m+1}}$
is injective. The proof of this is almost identical to that of $m =1$ case.
Indeed, it is easy to check that every element of $F^{n-1}_{m+1}$ must be of the form
$t^{m+1} \otimes \omega + t^m dt \otimes \omega'$.
We therefore get
\begin{equation}\label{eqn:M-Q-C-0-5}
  \begin{array}{lll}
d(t^{m+1} \otimes \omega  + t^m dt \otimes \omega') & = & 
t^{m+1} d \omega   + (m+1)t^{m} dt \wedge \omega -
t^{m}dt \wedge d \omega' \\
 & = & t^{m+1} d \omega + t^m dt \wedge ((m+1) \omega - d\omega') \\
 & = &  t^{m+1} d \omega + t^m dt \otimes ((m+1) \omega - d\omega').
\end{array}
\end{equation}

If $d(t^{m+1} \otimes \omega + t^m dt \otimes \omega')$ lies in 
$Rt^{m+1} \otimes_R \Omega^n_R$,
then we must have $t^m dt \otimes ((m+1) \omega - d \omega'_i) = 0$. But this 
implies that $\omega = d((m+1)^{-1}\omega')$. Putting this in ~\eqref{eqn:M-Q-C-0-5},
we get $t^{m+1} \otimes d\omega  = 0$. This shows that  
$dF^{n-1}_{m+1} \cap (Rt^{m+1} \otimes_R \Omega^n_R) = 0$. Equivalently, the left 
vertical arrow in ~\eqref{eqn:M-Q-C-4} is injective. This proves that $\theta^n_{R_m}$ 
is injective  for all $m \ge 1$ and completes the proof of the lemma.
\end{proof}

We shall also need the following related result later on.

\begin{lem}\label{lem:Inj-*}
Let $R$ be as above. Let $n \ge 0$ and $m \ge 1$ be integers. Then the map
$d \colon tR_m \otimes_R \Omega^{n}_R \to  \wt{\Omega}^{n+1}_{R_m}$ is injective.
\end{lem}
\begin{proof}
Suppose $\omega = \stackrel{m-1}{\underset{i=0}\sum} t^{i+1} \otimes \omega_i$ and 
$d(\omega) = 0$.
That is,  $\stackrel{m-1}{\underset{i=0}\sum} t^{i+1} \otimes d \omega_i  +
\stackrel{m-1}{\underset{i=0}\sum} (i+1) t^i dt \otimes \omega_i = 0$.
But this implies by ~\eqref{eqn:M-Q-C-0-0} that
$\stackrel{m-1}{\underset{i=0}\sum} (i+1) t^idt \otimes \omega_i  = 0$.
Since $\Omega^1_{{R_m}/R} \otimes_R \Omega^n_R  \cong \Omega^n_R dt \oplus \cdots 
\oplus \Omega^n_R t^{m-1}dt \cong (\Omega^n_R)^{m}$ as an $R$-module,
we must have $\omega_i = 0$ for each $i$. In particular, we have $\omega = 0$.
\end{proof}

\section{The cycle class map in characteristic zero}\label{sec:char-0-*}
In this section, we shall show that the cycle class map for the additive 0-cycles 
completely describes the relative $K$-theory of the truncated polynomial rings over a 
characteristic zero field in terms of additive 0-cycles.
This was perhaps the main target for the introduction of
the additive higher Chow groups by Bloch and Esnault \cite{BE2}.
We formulate our precise result as follows.

Let $k$ be a characteristic zero field. Let $R$ be a regular semi-local ring which is 
essentially of finite type over $k$. Let $m \ge 0, n \ge 1$ be two integers.
Recall from \S~\ref{sec:Rel-K-trun} that the canonical map from the Milnor to the 
Quillen $K$-theory
induces a map $\psi_{R_m, n} \colon  \wt{K}^M_n(R_m) \to \wt{K}_n(R_m)$. 
Since this map is clearly compatible
with change in $m \ge 0$, we have a strict map of pro-abelian groups
\begin{equation}\label{eqn:Add-2}
\psi_{R,n} \colon \{\wt{K}^M_n(R_m)\}_m \to \{\wt{K}_n(R_m)\}_m.
\end{equation}

In this section, we shall restrict our attention to the case when $R$ is the base field $k$ itself. 
We shall prove a general result for regular semi-local rings in the next section. 
In the case of the field $k$, every integer $n \ge 1$ has associated to it a diagram of
pro-abelian groups:
\begin{equation}\label{eqn:Add-3}
\xymatrix@C.8pc{
& \{\wt{K}^M_n(k_m)\}_m \ar[d]^-{\psi_{k}} \\
\{\TCH^{n}(k, n; m)\}_m \ar[r]^-{cyc_k} \ar@{.>}[ur] & \{\wt{K}_n(k_m)\}_m.}
\end{equation}

The following is our main result.

\begin{thm}\label{thm:Add-main}
If $k$ is a field of characteristic zero, then all maps in ~\eqref{eqn:Add-3} are 
isomorphisms.
\end{thm}

The proof of this theorem will be done by combining the results of
\S~\ref{sec:char-0} and \S~\ref{sec:Rel-Milnor**} with a series of new steps.

\subsection{Factorization of $cyc_k$ into Milnor $K$-theory}\label{sec:F-M}
We follow two step strategy  for proving \thmref{thm:Add-main}.
We shall first show that $cyc_k$ factors through the Milnor $K$-theory and 
the resulting map is an isomorphism. The second step will be to compare the
Milnor and Quillen $K$-groups in the pro-setting.
Apart from showing factorization through the Milnor $K$-theory,
the following result also improves \thmref{thm:Intro-1} in that it tells us that
for additive higher Chow groups of 0-cycles, the cycle class map is  a
strict morphism of pro-abelian groups (see \S~\ref{sec:Pro}).

\begin{lem}\label{lem:Factor}
  Let $m \ge 0, n \ge 1$ be two integers. Then the following hold.
  \begin{enumerate}
  \item
  The map $cyc_k \colon \Tz^{n}(k, n; m) \to
  \wt{K}_n(k_{m})$ descends to a group homomorphism
  \[
    cyc_k \colon \TCH^{n}(k, n; m) \to \wt{K}_n(k_{m}).
  \]
\item
  The map $cyc_k$ has a factorization
  \[
    \TCH^{n}(k, n; m) \xrightarrow{cyc^M_k}  \wt{K}^M_n(k_{m})
    \xrightarrow{\psi_{k_{m},n}} \wt{K}_n(k_{m}).
    \]
  \end{enumerate}
\end{lem}
\begin{proof}
  Since ${\rm char}(k) =0$, we know by the main result of \cite{R}
  (see also \cite[Theorem~1.2]{KP-4}) that each $\TCH^{n}(k, n; m)$ is a $k$-vector space.
  Similarly, $\wt{K}_n(k_{m})$ is a $\Q$-vector space because it is isomorphic to 
  $\wt{HC}_{n-1}(k_{m})$ by \cite{Good}.
  The first part of the lemma therefore follows directly from \thmref{thm:Intro-1-ratl}. 

We shall now prove the second part.
Since $\psi_{k_m,n}$ is injective for each $m \ge 0, n \ge 1$ by \lemref{lem:M-Q-C}, 
we only need to show
that $cyc_k$ takes a set of generators of the group $\TCH^{n}(k,n;m)$ to $\wt{K}^M_n(k_{m})$.


We first assume that $n = 1$. Let $z \in \A^1_k$ be a closed point. We can write
$z = \Spec({k[t]}/{(f(t)))}$, where $f(t)$ is an irreducible polynomial.
The modulus condition for $z$ implies that $f(0) \in k^{\times}$.
If we let  $g(t) = (f(0))^{-1} f(t)$ and let $\ov{g(t)}$ denote the image of 
$g(t)$ in $k_{m}$, then we see that $\ov{g(t)} \in \wt{K}^M_1(k_{m})$. 
By the definition of the cycle class map in ~\eqref{eqn:cycle-3}, we have that
$cyc_k(z) = [k(z)] \in K_0(\A^1_k, (m+1)\{0\})$. But \lemref{lem:Elem-0} says that
$[k(z)] = \partial (\ov{g(t)})$ under the isomorphism
$\wt{K}^M_1(k_{m}) = \wt{K}_1(k_{m}) 
\xrightarrow{\partial} K_0(k[t], (t^{m+1}))$.
So we are done.

Suppose now that $n \ge 2$. In this case, \lemref{lem:Milnor-surj} says that
$\TCH^{n}(k, n; m)$ is generated by closed points 
$z \in \A^1_k \times \square^{n-1}$
which lie in $\A^1_k \times \G_m^{n-1} \subset \A^{n}_k$. Furthermore,  $z \in \A^n_k$ 
is defined by an ideal
$I \subset k[t, y_1, \ldots , y_{n-1}]$ of the type
$I =(f(t), y_1 -b_1, \ldots , y_{n-1} - b_{n-1})$, where $b_i \in k^{\times}$ for each 
$1 \le i \le n-1$.
Since $z$ is a closed point, $f(t)$ must be an irreducible polynomial in $k[t]$.
Moreover,
$f(t)$ defines an element of $\W_{m}(k) \cong \wt{K}^M_1(k_{m})$. In particular,
$f(0) \in k^{\times}$.

We next note that the push-forward map $K_*(k(z)) \to K_*(k(z'))$
is $K_*(k)$-linear, where $z' = \Spec({k[t]}/{(f(t))})$ (see \cite[Chapter~3]{TT}).
The map $K_*(k(z')) \to K_*(k[t], (t^{m+1}))$ is $K_*(k)$-linear
by \lemref{lem:Proj-rel}. It follows that the composition
$K_*(k(z)) \to K_*(k[t], (t^{m+1}))$ is $K_*(k)$-linear.

It follows therefore from the definition of the cycle class map in ~\eqref{eqn:cycle-3} that
under the map
$cyc_k \colon \Tz^{n}(k, n; m) \to {K}_{n-1}(k[t], (t^{m+1}))$, 
we have 
\[
cyc_k([z]) = \{b_1, \ldots , b_{n-1}\} \cdot [k(z')] \in K^M_{n-1}(k) \cdot 
{K}_{0}(k[t], (t^{m+1})) \subseteq {K}_{n-1}(k[t], (t^{m+1})),
\]
where $\{b_1, \ldots , b_{n-1}\} \in K^M_{n-1}(k)$. 
We let $g(t) = (f(0))^{-1} f(t)$ and let $\ov{g(t)}$ be the image of $g(t)$
in $k_{m}$ via the surjection $k[t] \surj k_{m}$.
 
Since $\partial \colon \wt{K}_n(k_{m}) \xrightarrow{\cong}
{K}_{n-1}(k[t], (t^{m+1}))$ is $K_*(k)$-linear,
we see that
\[
\begin{array}{lll}
\{b_1, \ldots , b_{n-1}\} \cdot [k(z')] & = & \{b_1, \ldots , b_{n-1}\} \cdot \partial 
(\ov{g(t)}) \\
& = & \partial (\{b_1, \ldots , b_{n-1}\} \cdot \ov{g(t)}).
\end{array}
\]
We are now done because 
$\{b_1, \ldots , b_{n-1}\} \cdot \ov{g(t)} \in K^M_{n-1}(k) \cdot 
\wt{K}^M_1(k_{m}) \subseteq \wt{K}^M_n(k_{m})$.
We have thus shown that $cyc_k([z])$ lies in the image of $\wt{K}^M_n(k_{m})$
under the map $\partial$. This finishes the proof.
\end{proof}

\subsection{The main result for $cyc^M_k$}\label{sec:Chow-M-*}
Let $k$ be a characteristic zero field as before.
We shall now show that the cycle class map $cyc^M_k$ that we obtained in
\lemref{lem:Factor} is an isomorphism.
Since the map $\tau^k_{n,m} \colon \W_m \Omega^{n-1}_k \to \TCH^n(k,n;m)$ is an 
isomorphism, we shall make no distinction between $\W_m \Omega^{n-1}_k$ and 
$\TCH^n(k,n;m)$ throughout our discussion of the proof of 
\thmref{thm:Add-main}. Furthermore, we shall denote $cyc_k \circ \tau^k_{n,m}$ 
also by $cyc_k$ in what follows.

\begin{thm}\label{thm:Chow-Milnor}
For every pair of integers $m \ge 0, n \ge 1$, the map 
\[
cyc^M_k \colon \TCH^{n}(k, n; m) \to  \wt{K}^M_n(k_{m})
\]
is an isomorphism.
\end{thm}
\begin{proof}
  When $m =0$, the group on the right of $cyc_k$ is zero by definition and
  the one on the left is zero by \cite[Theorem~6.3]{KP-6}. We can therefore assume that
  $m \ge 1$.
 
  We can replace $\TCH^{n}(k, n; m)$ by $\W_m \Omega^{n-1}_k$. Accordingly, we can identify
  $cyc^M_k$ with $cyc^M_k \circ \tau^k_{n,m}$. 
  Let
  $\eta^k_m \colon \Omega^{n-1}_k \to t^mk_{m} \otimes_k \Omega^{n-1}_k$ denote the ($k$-linear) map
  $\eta^k_m(a \omega) = - at^m \otimes \omega$.
  This is clearly an isomorphism of $k$-vector spaces.

We shall prove the theorem by induction on $m \ge 1$.
Suppose first that $m =1$. In this case, it follows from Lemmas~\ref{lem:Milnor-Chow-TCH-0} and
~\ref{lem:Milnor-Witt-desc} that
\[
  \begin{array}{lll}
cyc^M_k(ad\log(b_1) \wedge \cdots \wedge d\log(b_{n-1}))
    & = & cyc^M_k(Z(1-at, y_1 - b_1, \ldots , y_{n-1} - b_{n-1})) \\
    & = &  \{1 -at, b_1, \ldots , b_{n-1}\} \\
    & = & \wt{\theta}^n_{k_1}((-at) \otimes d\log(b_1) \wedge \cdots \wedge d\log(b_{n-1}))\\
    & = & \wt{\theta}^n_{k_1} \circ \eta^k_1(a d\log(b_1) \wedge \cdots \wedge d\log(b_{n-1})).
  \end{array}
\]
It follows from this that $cyc^M_k =  \wt{\theta}^n_{k_1} \circ \eta^k_1$.
We now apply \lemref{lem:M-Q-C-0} to conclude that $cyc^M_k$ is an isomorphism.

Suppose now that $m \ge 2$. Let $F^n_m$ denote the kernel  of the restriction map
$\wt{K}^M_n(k_{m}) \surj \wt{K}^M_n(k_{m-1})$. It is easy to see that the isomorphism
$\wt{\theta}^n_{k_{m}}$ of  \lemref{lem:M-Q-C-0} commutes with the restriction
map $k_{m} \surj k_{m-1}$ for all $m \ge 1$.
It follows therefore from  \lemref{lem:M-Q-C-0} and the snake lemma that
$\wt{\theta}^n_{k_{m+1}}$ restricts to an isomorphism
\begin{equation}\label{eqn:Chow-Milnor**-0}
  \wt{\theta}^n_{k_{m+1}} \colon  \Omega^{n-1}_k \otimes_k t^{m+1}k_{m+1} \xrightarrow{\cong} F^n_{m+1};
  \end{equation}
  \[
    \wt{\theta}^n_{k_{m+1}}((d\log(b_1) \wedge \cdots \wedge d\log(b_{n-1})) \otimes at^{m+1})
    = \{1 + at^{m+1}, b_1, \ldots , b_{n-1}\}.
    \]

    We now consider the diagram
    \begin{equation}\label{eqn:Chow-Milnor**-1}
      \xymatrix@C.8pc{
        0 \ar[r] & \Omega^{n-1}_k \ar[r]^-{V_{m+1}}
        \ar[d]_-{\wt{\theta}^n_{k_{m+1}} \circ \eta^k_{m+1}} &
        \W_{m+1}\Omega^{n-1}_k \ar[d]^-{cyc^M_k} \ar[r] & \W_m\Omega^{n-1}_k \ar[r] \ar[d]^-{cyc^M_k} & 0 \\
        0 \ar[r] & F^n_{m+1} \ar[r] &  \wt{K}^M_n(k_{m+1}) \ar[r] &  \wt{K}^M_n(k_{m}) \ar[r] & 0,}
      \end{equation}
      where the horizontal arrows on the right in both rows are the restriction maps.
      In particular, the square on the right is commutative.

      We next note that $V_{m+1}(ad\log(b_1) \wedge \cdots \wedge d\log(b_{n-1})) =
      V_{m+1}(a) d\log(b_1) \wedge \cdots \wedge d\log(b_{n-1})$ (see
      \cite[Proposition~4.4]{RS}).
      Indeed, by iteration, it is enough to check this for $n =2$. Furthermore, we can check
      it in $\TCH^2(k,2;m)$. 
      Now, we have 
      \begin{equation}\label{eqn:Chow-Milnor**-2}
        \begin{array}{lll}
          V_{m+1}(ad\log(b_1)) & = & V_{m+1}(Z(1-ab^{-1}_1t) d(Z(1-b_1t))) \\
                             & = & V_{m+1}(Z(1 -at, y_1 - b_1)) \\
                             & = & Z(1 - at^{m+1}, y_1 - b_1),
        \end{array}
      \end{equation}
      where the last two equalities follow from the definitions of differential and Verschiebung on
      the additive higher Chow groups (see the proof of \lemref{lem:Milnor-Witt-desc}).
      On the other hand, we have
\begin{equation}\label{eqn:Chow-Milnor**-3}
        \begin{array}{lll}      
          V_{m+1}(a) d\log(b_1) & = & V_{m+1}(a) b^{-1}_1db_1 \\
                              & = & (V_{m+1}(Z(1 - at)) Z(1-b^{-1}_1t) d(Z(1-b_1t)) \\
                              & = & Z(1 - at^{m+1}) Z(1- b^{-1}_1t) d(Z(1- b_1t)) \\
                              & = & Z(1 - at^{m+1}) Z(1 - b^{-1}_1t) Z(1-b_1t, y_1 - b_1) \\
                              & = & Z(1 - at^{m+1}) Z(1 - t, y_1 - b_1) \\
                              & = & Z(1 - at^{m+1}, y_1 - b_1),
        \end{array}
      \end{equation}
      where the equalities again follow from various definitions
      (see the proof of \lemref{lem:Milnor-Witt-desc}).
A combination of ~\eqref{eqn:Chow-Milnor**-2} and
~\eqref{eqn:Chow-Milnor**-3} proves what we  had claimed.
Since the map $k \otimes K^M_{n-1}(k) \to \Omega^{n-1}_k$ (given by $a \otimes \un{b} \mapsto
ad\log(b_1) \wedge \cdots \wedge d\log(b_{n-1})$) is surjective,
it follows from ~\eqref{eqn:Milnor-surj*-3}
and ~\eqref{eqn:Chow-Milnor**-0} that the left square in  ~\eqref{eqn:Chow-Milnor**-1}
is also commutative.

The top row of ~\eqref{eqn:Chow-Milnor**-1} is well known to be exact
in characteristic zero (e.g., see 
\cite[Remark~4.2]{RS}). The bottom row is exact by the definition of $F^n_{m+1}$
and the fact that
the map $\wt{K}^M_n(k_{m+1}) \to \wt{K}^M_n(k_{m})$ is surjective.
The left vertical arrow in ~\eqref{eqn:Chow-Milnor**-1} is an isomorphism by
~\eqref{eqn:Chow-Milnor**-0} and the right vertical arrow is an isomorphism by induction.
We conclude that the middle vertical arrow is also an an isomorphism.
This completes the proof of the theorem.
\end{proof}

\begin{remk}\label{remk:k-R}
  The reader can check that the proof of \thmref{thm:Chow-Milnor} continues to work for any
  regular semi-local ring $R$ containing $k$ as long as $cyc_R$ is defined in a way it is for fields.
  In more detail, the isomorphism $\eta^R_m$ makes sense. Lemmas~\ref{lem:Milnor-Chow-TCH-0} and
  ~\ref{lem:Milnor-Witt-desc} as well as ~\eqref{eqn:Chow-Milnor**-0} are all valid for $R$.
  The final step ~\eqref{eqn:Chow-Milnor**-1} also holds for $R$.
  This observation will be used in the next section in extending \thmref{thm:Add-main} to regular
  semi-local rings.
  \end{remk}

\subsection{The final step}\label{sec:Final-0}
We shall now prove our final step in the proof of \thmref{thm:Add-main}.
Namely, we shall show that the reduced Milnor and Quillen $K$-theories of the 
truncated polynomial rings are
pro-isomorphic. This will finish the proof of \thmref{thm:Add-main}.
The desired pro-isomorphism is based on the following general result about 
pro-vanishing of 
relative cyclic homology.

Let $R$ be a ring containing $\Q$. Recall from \cite[\S~4.5, 4.6]{Loday} that 
for integers $m \ge 0$ and $n  \ge 1$, there are functorial $\lambda$-decompositions
$\wt{HH}_n(R_m) = \stackrel{n}{\underset{i =1}\oplus} \wt{HH}^{(i)}_n(R_m)$
and $\wt{HC}_n(R_m) = \stackrel{n}{\underset{i =1}\oplus} \wt{HC}^{(i)}_n(R_m)$.
Moreover, there are  isomorphisms
\begin{equation}\label{eqn:Lambda-0}
\wt{HH}^{(n)}_n(R_m) \cong \wt{\Omega}^n_{R_m} \ {\rm  and} \ 
\wt{HC}^{(n)}_n(R_m) \cong \frac{\wt{\Omega}^n_{R_m}}{d \wt{\Omega}^{n-1}_{R_m}}
\end{equation}
such that the canonical map $\wt{HH}^{(n)}_n(R_m) \to \wt{HC}^{(n)}_n(R_m)$ 
is the quotient map
$\wt{\Omega}^n_{R_m} \to \frac{\wt{\Omega}^n_{R_m}}{d \wt{\Omega}^{n-1}_{R_m}}$.
We also have the functorial isomorphisms
\begin{equation}\label{eqn:Lambda-1}
\wt{HH}^{(i)}_n(R_m) \cong \wt{D}^{(i)}_{n-i}(R_m).
\end{equation}

The key result is the following.

\begin{prop}\label{prop:Pro-vanish}
Let $R$ be a regular semi-local ring which is essentially of finite type over a 
characteristic zero
field. Let $0 < i < n$ be two integers. Then $\{\wt{HC}^{(i)}_n(R_m)\}_m = 0$.
\end{prop}
\begin{proof}
First we show that $\{\wt{HH}^{(i)}_n(R_m)\}_m = 0$. By ~\eqref{eqn:Lambda-1}, 
this is equivalent to
showing that $\{\wt{D}^{(i)}_{n-i}(R_m)\}_m = 0$.
To show this latter vanishing, we let $A = R[t^2, t^3] \subset R[t]$ be the monomial subalgebra 
generated by $\{t^2, t^3\}$.
We then know that the inclusion $A \inj R[t]$ is the normalization homomorphism 
whose conductor ideal
is $I = (t^2, t^3) \subset A$ such that $IR[t] = (t^2)$.
Since $R[t]$ is regular, we know that $D^{(i)}_{n-i}(R[t]) = 0$ for $0 < i < n$ by
\cite[Theorem~3.5.6]{Loday}. We now apply part (ii) of 
\cite[Proposition~5.2]{Krishna-0} 
with $A = R[t^2, t^3]$ and $B = R[t]$ to conclude that
$\{{D}^{(i)}_{n-i}(R_{2m-1})\}_m = \{D^{(i)}_{n-i}({R[t]}/{(t^{2m})}\}_m = 0$.
But this is same as saying that $\{{D}^{(i)}_{n-i}(R_{m})\}_m = 0$.
In particular, we get $\{\wt{D}^{(i)}_{n-i}(R_{m})\}_m = 0$.

To prove the result for cyclic homology, we first assume $n \ge 2$ and $i = n-1$.
Connes' periodicity exact sequence (see \cite[theorem~4.6.9]{Loday}) gives 
an exact sequence
\[
\wt{HH}^{(n-1)}_n(R_m) \xrightarrow{I} \wt{HC}^{(n-1)}_n(R_m) \xrightarrow{S}
\frac{\wt{\Omega}^{n-2}_{R_m}}{d \wt{\Omega}^{n-3}_{R_m}} 
\xrightarrow{B} \wt{\Omega}^{n-1}_{R_m} 
\xrightarrow{I} \frac{\wt{\Omega}^{n-1}_{R_m}}{d \wt{\Omega}^{n-2}_{R_m}}.
\]
We have shown that the first term in this exact sequence vanishes. It suffices 
therefore to show
that the map $B$ is injective. But $B$ is same as the differential map
$d \colon \frac{\wt{\Omega}^{n-2}_{R_m}}{d \wt{\Omega}^{n-3}_{R_m}} \to 
\wt{\Omega}^{n-1}_{R_m}$
by \cite[Corollary~2.3.5]{Loday}.
By \lemref{lem:M-Q-C-0}, we therefore have to show that the map
$d \colon \Omega^{n-2}_R \otimes_R tR_m \to  \wt{\Omega}^{n-1}_{R_m}$ is injective.
But this follows from \lemref{lem:Inj-*}.

We prove the general case by induction on $n \ge 2$.
The case $n =2$ is just proven above. When $n \ge 3$, we can assume that 
$1 \le i \le n-2$ by what
we have shown above. 
We now again use the periodicity exact sequence to get an exact sequence of pro-abelian groups: 
\[
\{ \wt{HH}^{(i)}_n(R_m) \}_m \xrightarrow{I} \{ \wt{HC}^{(i)}_n(R_m)\}_m  \xrightarrow{S} 
\{\wt{HC}^{(i-1)}_{n-2}(R_m)\}_m.
\]
Since $i \le n-2$, we get $i-1 \le n-3$. Hence, the induction hypothesis 
implies that the last term of this exact sequence is zero. We have shown 
above that the first term is zero.
We conclude that the middle term is zero as well. This finishes the proof.
\end{proof}

\begin{cor}\label{cor:Milnor-Quillen-main}
Let $R$ be a regular semi-local ring which is essentially of finite type 
over a characteristic zero
field. Let $n \ge 0$ be an integer. Then the canonical map
\[
\psi_{R,n} \colon \{\wt{K}^M_n(R_m)\}_m \to \{\wt{K}_n(R_m)\}_m
\]
is an isomorphism of pro-abelian groups.
\end{cor}
\begin{proof}
The case $n \le 2$ is well known (e.g, see \cite[Proposition~2]{kerz10})
and in fact an isomorphism at every level $m \ge 0$.
We can therefore assume that $n \ge 3$.
We have seen in  \S~\ref{sec:Rel-K-trun} that Goodwillie provides an
isomorphism of $\Q$-vector spaces 
\[
\tr^R_{m,n} \colon \wt{K}_n(R_m) \xrightarrow{\cong} \wt{HC}_{n-1}(R_m) \cong \
\stackrel{n-1}{\underset{i =1}\oplus} \wt{HC}^{(i)}_{n-1}(R_m).
\]

Moreover, Lemmas~\ref{lem:Adams} and~\ref{lem:M-Q-C} say that 
the map $\psi_{R_m,n} \colon \wt{K}^M_n(R_m) \to \wt{K}_n(R_m)$ is injective 
and $\tr^R_{m,n}$ maps $\wt{K}^M_n(R_m)$ isomorphically onto 
$\wt{HC}^{(n-1)}_{n-1}(R_m)$.
We therefore have to show that
$\{\wt{HC}^{(i)}_{n-1}(R_m)\}_m = 0$ for $1 \le i \le n-2$. But this 
follows from  \propref{prop:Pro-vanish}.
\end{proof}

\vskip .2cm

{\bf{Proof of \thmref{thm:Add-main}:}}
The proof is a combination of \lemref{lem:Factor}, \thmref{thm:Chow-Milnor} 
and \corref{cor:Milnor-Quillen-main}.
$\hfill\square$

\section{The cycle class map for semi-local rings}\label{sec:local}
In this section, we shall define the cycle class map for relative 0-cycles over
regular semi-local rings and prove an extension of \thmref{thm:Add-main} 
for such rings.
Let $k$ be a characteristic zero field and let $R$ be a regular 
semi-local ring which is 
essentially of finite type over $k$. We shall let $F$ denote the 
total quotient ring of $R$.
Note that $R$ is a product of regular semi-local integral domains and 
$F$ is the product of
their fields of fractions. Since all our proofs for regular semi-local 
integral domains
directly generalize to finite products of such rings, we shall assume 
throughout that $R$
is an integral domain. We shall let $\pi: \Spec(F)  \to \Spec(R)$ denote the
inclusion of generic point. We shall often write $X = \Spec(R)$ and 
$\eta = \Spec(F)$. We shall let $\Sigma$ denote the set of all 
maximal ideals of $R$.

\subsection{The sfs cycles}\label{sec:Rel-cyc}
Let $m \ge 0$ and $n \ge 1$ be two integers.
Recall from \S~\ref{sec:HCGM} and \S~\ref{sec:Add-0}
that $\TCH^n(R,n;m)$ is the defined as the middle homology of the complex
$\Tz^n(R,n+1;m) \xrightarrow{\partial} \Tz^n(R,n;m)
\xrightarrow{\partial} \Tz^{n}(R,n-1;m)$.
A cycle in $\Tz^n(R,n;m)$ has relative dimension zero over $R$. 
For this reason,
$\TCH^n(R,n;m)$ is often called the additive higher Chow group of 
relative 0-cycles on $R$.
When $R$ is a field, it coincides with the one used in the statement of 
\thmref{thm:Add-main}.

Since $\TCH^n(R,n;m)$ does not consist of 0-cycles if 
$\dim(R) \ge 1$, we can not directly apply \thmref{thm:Intro-1}
to define a cycle class map for $\TCH^n(R,n;m)$.
We have to use a different approach for 
constructing the cycle class map.
We shall use the main results of \cite{KP-3} and the case of fields to
construct a cycle class map in this case. We shall show later 
in this section that this map is
an isomorphism. 
We shall say that an extension of regular semi-local rings 
$R_1 \subset R_2$ is simple
if there is an irreducible monic polynomial 
$f \in R_1[t]$ such that $R_2 = {R_1[t]}/{(f(t))}$.

Let $Z \subset X \times \A^1_k \times \square^{n-1}$ be an irreducible 
admissible relative 0-cycle.
Recall from \cite[Definition~2.5.2, Proposition~2.5.3]{KP-2} that 
$Z$ is called an sfs-cycle if the following hold.
\begin{enumerate}
\item
$Z$ intersects $\Sigma \times \A^1_k \times F$ properly for all faces 
$F \subset \square^{n-1}$.
\item
The projection $Z \to X$ is finite and surjective.
\item
$Z$ meets no face of $X \times \A^1_k \times  \square^{n-1}$.
\item
$Z$ is closed in $X \times \A^1_k \times \A^{n-1}_k = 
\Spec(R[t, y_1, \ldots , y_{n-1}])$ 
(by (2) above) and there is a sequence of simple extensions of
regular semi-local rings
\[
R = R_{-1} \subset R_0 \subset \cdots \subset R_{n-1} = k[Z]
\]
such that $R_0 = {R[t]}/{(f_0(t))}$ and $R_i = 
{R_{i-1}[y_i]}/{(f_i(y_i))}$ for $1 \le i \le n-1$.
\end{enumerate}
Note that an sfs-cycle has no boundary by (3) above.
We let $\Tz^n _{\sfs}(R, n;m) \subset \Tz^n(R, n;m)$ be the 
free abelian group on integral sfs-cycles and
define
\begin{equation}\label{eqn:sfs}
\TCH^n _{\rm sfs} (R, n; m) = \frac{\Tz^n _{\sfs}(R, n;m)}
{\partial(\Tz^n (R, n+1;m)) \cap \Tz^n _{ \sfs} (R, n;m)}.
\end{equation}

We shall use the following result from \cite[Theorem~1.1]{KP-3}.

\begin{prop}\label{prop:sfs-lemma}
The canonical map $\TCH^n _{\rm sfs} (R, n; m) \to \TCH^n(R, n; m)$ 
is an isomorphism.
\end{prop}

\subsection{The cycle class map}\label{sec:CCM-R}
By \propref{prop:sfs-lemma}, it suffices to define the cycle class map on 
$\TCH^n _{\rm sfs} (R, n; m)$.
We can now repeat the construction of \S~\ref{sec:CCM} word by 
word to get our map.
So let $Z \subset X \times \A^1_k \times  \square^{n-1}$ be an 
irreducible sfs-cycle and let $R' = k[Z]$.
Let $f \colon Z \to X \times \A^1_k$ be the projection map.
Let $g_i \colon Z \to \square$ denote the $i$-th projection.
Then the sfs property implies that each $g_i$ defines an element of 
${R'}^{\times}$, and this in turn
gives a unique element $cyc^M_{R'}([Z]) = \{g_1, \ldots , g_{n-1}\} 
\in K^M_{n-1}(R')$.
We let $cyc_{R'}([Z])$ be its image in $K_{n-1}(R')$ under the map
$K^M_{n-1}(R') \to K_{n-1}(R')$.
Since $Z$ does not meet $X \times \{0\}$, we see that the finite map $f$ 
defines a push-forward map of spectra $f_* \colon K(R') \to K(R[t], (t^{m+1}))$.
We let $cyc_R([Z]) = f_*(cyc_{R'}([Z])) \in K_{n-1}(R[t], (t^{m+1}))$.
We extend this definition linearly to get a cycle map $cyc_R 
\colon  \Tz^n _{ \sfs} (R, n;m) \to
K_{n-1}(R[t], (t^{m+1}))$.
We can now prove our first result of this section.

\begin{thm}\label{thm:CCM-R-0}
The assignment $[Z] \mapsto cyc_R([Z])$ defines a cycle class map
\[
cyc_{R} \colon \TCH^n(R, n; m) \to K_{n-1}(R[t], (t^{m+1}))
\]
which is functorial in $R$.
\end{thm}
\begin{proof}
Let $F$ be the fraction field of $R$. We consider the diagram
  \begin{equation}\label{eqn:CCM-R-0-0}
    \xymatrix@C.8pc{
      \partial^{-1}(\Tz^n _{ \sfs} (R, n;m)) 
\ar[r]^-{\partial} \ar[d]_-{\pi^*} & 
\Tz^n _{ \sfs} (R, n;m) \ar[d]^-{\pi^*}
      \ar[r]^-{cyc_R} & K_{n-1}(R[t], (t^{m+1})) \ar[d]^-{\pi^*} \\
      \Tz^n(F, n+1;m) \ar[r]^-{\partial} & 
\Tz^n(F, n;m) \ar[r]^-{cyc_F} & K_{n-1}(F[t], (t^{m+1})).}
  \end{equation}

Assume first that this diagram is commutative. 
Then \thmref{thm:Chow-Milnor} says that 
$cyc_F \circ \partial \circ \pi^*$ is zero. Equivalently, 
$\pi^* \circ cyc_R \circ \partial = 0$. We will be therefore done if we know that
the right vertical map $\pi^*$ between the relative $K$-groups is injective.
To show this, we can replace these relative $K$-groups by the relative cyclic
homology groups by \cite{Good}.
These relative cyclic homology groups in turn can be replaced by the Hochschild
homology $HH_*(R)$ and $HH_*(F)$ by \cite[Proposition~8.1]{Hesselholt-Hand}.
Since $R$ is regular, we can go further and replace
$HH_*(R)$ and $HH_*(F)$ by $\Omega^*_R$ and $\Omega^*_F$, respectively,
by the famous Hochschild-Kostant-Rosenberg theorem.
We therefore have to show that the map $\Omega^*_R \to \Omega^*_F$ is injective.
But this follows from \lemref{lem:Inj-K}.

To show that $cyc_R$ is natural for homomorphisms of regular semi-local
rings $R \to R'$, we first observe that ~\eqref{eqn:CCM-R-0-0}
shows that $cyc_R$ is functorial for the inclusion $R \inj F$.
Since the right-most vertical arrow in ~\eqref{eqn:CCM-R-0-0}
is injective, we can replace $R$ and $R'$ by their fraction fields
to check the naturality of $cyc_R$ in general.
In this case, the naturality of $cyc_R$ follows from 
\thmref{thm:Intro-1}. 
It remains now to show that ~\eqref{eqn:CCM-R-0-0} is commutative.

  The left square is known to be commutative by the flat pull-back property 
of additive cycle complex.
  To show that the right square commutes, let $Z \subset 
\A^1_R \times_R \square^{n-1}_R$ be an
  irreducible sfs-cycle and let $R_{n-1} = k[Z]$.
  Let $R_0$ be the coordinate ring of the image of $Z$ in $\A^1_R$ 
as in the definition
of the sfs-cycles. By definition of sfs-cycles, we have the 
commutative diagram
\begin{equation}\label{eqn:finite-sfs}
  \xymatrix@C.8pc{
    R \ar[r] \ar[d]_-{\pi} & R_0 \ar[r]^-{f_0} \ar[d]^-{\pi_0} & 
R_{n-1} \ar[d]^-{\pi_{n-1}} \\
    F \ar[r] & F_0 \ar[r]^-{f_0} & F_{n-1},}
\end{equation}
where each term in the bottom row is the quotient field of the 
corresponding term on the top row.
Note also that all horizontal arrows are finite maps of 
regular semi-local integral domains.
In particular, we have $F_0 = R_0 \otimes_R F$ and 
$F_{n-1} = R_{n-1} \otimes_R F = R_{n-1} \otimes_{R_0} F_0$.

We let $f \colon Z \to \A^1_R$ be the projection and let 
$p \colon \Spec(R_0) \inj \A^1_R$
be the inclusion.
We denote the projection $\Spec(F_{n-1}) \to \A^1_F$ and inclusion 
$\Spec(F_0) \inj \A^1_F$ also by $f$
and $p$, respectively. Note that $f$ is a finite map
which has factorization 
$Z \to \A^1_R \setminus \{0\} \subset \A^1_R$. 

We let $\alpha = cyc^M_{R_{n-1}}([Z]) = 
\{g_1, \ldots , g_{n-1}\} \in K^M_{n-1}(R_{n-1})$.
We then have, by definition,
$cyc_R([Z]) = f_* \circ \psi_{R_{n-1}}(\alpha)$ and 
$cyc_F(\pi^*([Z])) = f_* \circ \psi_{F_{n-1}}
\circ \pi_{n-1}^*(\alpha)$. Using ~\eqref{eqn:finite-sfs}, we can write these as
$cyc_R([Z]) = p_* \circ f_{0  *} \circ \psi_{R_{n-1}}(\alpha)$ and $cyc_F(\pi^*([Z]))
= p_* \circ f_{0 *} \circ \psi_{F_{n-1}} \circ \pi_{n-1}^*(\alpha)$.
Since $\pi_{n-1}^* \circ \psi_{R_{n-1}} = \psi_{F_{n-1}} \circ \pi_{n-1}^*$, 
we only have to show that the diagram
  \begin{equation}\label{eqn:CCM-R-0-1}
    \xymatrix@C.8pc{
      K(R_{n-1}) \ar[r]^-{f_{0 *}} \ar[d]_-{\pi^*_{n-1}} & K(R_0) 
\ar[d]^-{\pi^*_0} \ar[r]^-{p_*}  & 
      K(R[t], (t^{m+1})) \ar[d]^-{\pi^*} \\      
  K(F_{n-1}) \ar[r]^-{f_{0 *}} & K(F_0) \ar[r]^-{p_*} & K(F[t], (t^{m+1}))}
\end{equation}
commutes.

The left square commutes by \cite[Proposition~3.18]{TT} since 
$F_{n-1} = R_{n-1} \otimes_{R_0} F_0$
and $f_0$ is finite. To see that the right square commutes, 
note that we can replace $K(R_0)$ by $K^{Z_0}(\A^1_R)$
and $K(F_0)$ by $K^{\eta_0}(\A^1_F)$, where $Z_0 = \Spec(R_0)$ and 
$\eta_0 = \Spec(F_0)$ (see \S~\ref{sec:Rel-K}).
We can do this because $R[t]$ and $R_0$ are regular. 
We are now done because the diagram
\begin{equation}\label{eqn:CCM-R-0-2}
  \xymatrix@C.8pc{
    K(\A^1_R, (m+1)\{0\}) \ar[r] \ar[d] & 
K(\A^1_R \setminus Z_0, (m+1)\{0\}) \ar[d] \\
    K(\A^1_F, (m+1)\{0\}) \ar[r] & 
K(\A^1_F \setminus \eta_0, (m+1)\{0\})}
\end{equation}
of restriction maps is commutative and the 
right square in ~\eqref{eqn:CCM-R-0-1}
is gotten by taking the homotopy fibers of the two rows 
of ~\eqref{eqn:CCM-R-0-2}.
We have now shown that both squares in 
~\eqref{eqn:CCM-R-0-0} are commutative.
This also shows that $cyc_R$ is compatible with the 
inclusion $R \inj F$. The proof of the theorem is complete.
\end{proof}

Throughout the remaining part of our discussion, we shall identify
$\TCH^n(R, n; m)$ with $\W_m \Omega^{n-1}_R$ (by ~\eqref{eqn:Witt-Chow}) and
${K}_{n-1}(R[t], (t^{m+1}))$ with $\wt{K}_n(R_{m})$ 
(via the connecting homomorphism).

\subsection{Factorization through Milnor $K$-theory}\label{sec:Fact-R}
We shall now show  that $cyc_R$ factors through the 
relative Milnor $K$-theory. The
proof is identical to the case of fields and we shall only sketch it.
We shall reduce the proof to the case of fields 
using the following result.

\begin{lem}\label{lem:Inj-K}
For $n \ge 0$ and $m \ge 1$, the map 
$\pi^* \colon \W_m \Omega^n_R \to \W_m \Omega^n_F$ is injective. 
In particular, the map
$\pi^* \colon \wt{K}^M_n(R_m) \to \wt{K}^M_n(F_m)$ is injective for all $m \ge 0$.
\end{lem}
\begin{proof}
Since $\W_m \Omega^n_R \cong (\Omega^n_R)^m$
(and also for $F$), we need to show that 
$\Omega^n_R \to \Omega^n_F$ is injective
to prove the first assertion of the lemma.
Since $\Omega^n_F \cong \Omega^n_R \otimes_R F$, it suffices to 
show that $\Omega^1_R$ is 
a free $R$-module. Since $R$ is regular, we have $D_1(R|k) = 0$ and 
$\Omega^1_{R/k}$ is a free 
$R$-module.
The Jacobi-Zariski exact sequence (see \cite[3.5.5]{Loday})
therefore tells us that
$\Omega^1_R \cong (\Omega^1_k \otimes_k R) \oplus \Omega^1_{R/k}$. 
This proves the first part.

For the second part, there is nothing to prove when $m =0$.
When $m \ge 1$, \lemref{lem:M-Q-C-0} reduces to showing that the map
$\Omega^n_R \otimes_R tR_m \to \Omega^n_F \otimes_F tF_m$ is 
injective for all $n \ge 0$.
Since $\Omega^n_F \otimes_F tF_m \cong \Omega^n_F \otimes_R tR_m$ and 
$tR_m$ is a free $R$-module,
the problem is reduced to showing that $\Omega^n_R \to \Omega^n_F$ 
is injective. We can now use the first part of the lemma.
\end{proof}

Our second main result of this section is the following.
This generalizes the main results of
\cite{EM}, \cite{NS} and \cite{Totaro} to truncated polynomial rings.

\begin{thm}\label{thm:Chow-Milnor-R}
Let $R$ and $m \ge 0, n \ge 1$ be as above. Then the cycle class map $cyc_R$ 
has a factorization
\[
 \TCH^n(R, n; m) \xrightarrow{cyc^M_R}  \wt{K}^M_{n}(R_{m}) 
\xrightarrow{\psi_{R_{m+1}, n}}  \wt{K}_{n}(R_{m}).
\]
Furthermore, $cyc^M_R$ is natural in $R$ and is an isomorphism.
\end{thm}
\begin{proof}
We shall use \propref{prop:sfs-lemma} which allows us to repeat 
the proof of the field case
(\lemref{lem:Factor}) word by word. 
When $n =1$, any sfs irreducible cycle $Z \subset \A^1_R$ is of the 
form $Z = V(f(t))$,
where $f(t)$ is an irreducible polynomial such that $f(0) \in R^{\times}$. 
We now repeat the argument
of the field case and use \lemref{lem:Elem-0} to finish
the proof. The $n \ge 2$ case follows from \lemref{lem:Milnor-surj} 
and the proof is identical to the case of fields. 
To prove that $cyc^M_R$ is an isomorphism,
we again repeat the case of fields and use Remark~\ref{remk:k-R}. 
The naturality of $cyc^M_R$ follows from \thmref{thm:CCM-R-0} since
$\psi_{R_{m+1}, n}$ is injective.
\end{proof}

Finally, we are now ready to prove \thmref{thm:Intro-2} (5).  We restate it  
again for reader's convenience.

\begin{thm}\label{thm:Chow-Milnor-Final}
Let $R$ be a regular semi-local ring which is essentially of finite type over a
characteristic zero field. Let $n \ge 1$ be an integer. Then the cycle class map
\[
cyc_R \colon \{\TCH^n(R, n; m)\}_m \to \{\wt{K}_{n}(R_{m})\}_m
\]
is an isomorphism of the pro-abelian groups.
\end{thm}
\begin{proof}
Combine \thmref{thm:Chow-Milnor-R} and \corref{cor:Milnor-Quillen-main}.
\end{proof}

\section{Appendix: Milnor vs Quillen $K$-theory}\label{sec:MQ}
In this section, we collect some results on the compatibility of various maps
between Milnor and Quillen $K$-theories of fields. They are used 
in the proofs of the main results
of this paper. We expect these results to be known to experts but
could not find their written proofs in the literature.

Let $k$ be a field and let $X$ be a regular scheme which is 
essentially of finite type over $k$.
    Let $x, y \in X$ be two points in $X$ of codimensions $p$ and $p-1$, 
respectively, such that 
$x \in \ov{\{y\}}$.
    Let
    \[
      Y = \ov{\{y\}}, \ F = k(y), \ A = \sO_{Y,x}, {\rm and} \  l = k(x).
    \]

\begin{lem}\label{lem:L1}
For any $n \ge 1$, the diagram
\begin{equation}\label{eqn:L2}
  \xymatrix@C.8pc{
    K^M_n(F) \ar[r]^-{\partial^M} \ar[d] & K^M_{n-1}(l) \ar[d] \\
    K_n(F) \ar[r]^-{\partial^Q} & K_{n-1}(l)}
\end{equation}
is commutative.
\end{lem}
\begin{proof}
  Let $B$ denote the normalization of $A$ and let $S$ denote 
the set of maximal ideals of $B$. 
Note that $B$ is
  semi-local so that $S$ is finite. Since the localization sequence for 
Quillen $K$-theory of 
coherent sheaves
  is functorial for proper push-forward, we have a commutative diagram
  \begin{equation}\label{eqn:L3}
    \xymatrix@C.8pc{
      G_n(B) \ar[r] \ar[d] & K_n(F) \ar@{=}[d] \ar[r]^-{\partial^Q} & 
{\underset{z \in S} \oplus} K_{n-1}(k(z))
      \ar[d]^-{\sum_z T_{{k(z)}/{l}}} \\
      G_n(A) \ar[r] & K_n(F) \ar[r]^-{\partial^Q} & K_{n-1}(l),}
    \end{equation}
    where $T_{{k(z)}/{l}}$ is our notation for the finite push-forward 
$K_*(k(z)) \to K_*(l)$ and 
$G_*(-)$ is
    Quillen $K$-theory of coherent sheaves functor (for proper morphisms).
Note that $K_*(B) \xrightarrow{\cong} G_*(B)$.

On the other hand, the boundary map in Milnor $K$-theory also has the property 
that the diagram
\begin{equation}\label{eqn:L4}
  \xymatrix@C.8pc{
   K^M_n(F) \ar@{=}[d] \ar[r]^-{\partial^M} & 
{\underset{z \in S} \oplus} K^M_{n-1}(k(z))
      \ar[d]^-{\sum_z N_{{k(z)}/{l}}} \\
   K^M_n(F) \ar[r]^-{\partial^M} & K^M_{n-1}(l)}
    \end{equation}  
    commutes, where the right vertical arrow is the sum of the Norm maps 
in Milnor $K$-theory 
of fields (see \cite{BT} and \cite{Kato-1}).
The lemma therefore follows if we prove Lemmas~\ref{lem:L5} and 
~\ref{lem:L6} below.
\end{proof}

\begin{lem}\label{lem:L5}
  Let $z \in S$ be a closed point as above and let $R$ denote 
the discrete valuation ring of $F$ 
associated to $z$.
  Then the diagram
  \begin{equation}\label{eqn:L5-0}
  \xymatrix@C.8pc{
    K^M_n(F) \ar[r]^-{\partial^M} \ar[d] & K^M_{n-1}(k(z)) \ar[d] \\
    K_n(F) \ar[r]^-{\partial^Q} & K_{n-1}(k(z))}
\end{equation}
is commutative for every $n \ge 1$.
\end{lem}
\begin{proof}
  It is well known and elementary to see 
(using the Steinberg relations) that $K^M_*(F)$ is 
generated by $K^M_1(F)$
  as an $K^M_*(R)$-module. Furthermore, $\partial^M$ is $K^M_*(R)$-linear 
(see \cite[\S~4, Proposition~4.5]{BT}).
  Since the localization sequence such as the one on the top of ~\eqref{eqn:L3} 
(with $B$ replaced by $R$)
  is $K_*(R)$-linear, it follows that all arrows in 
~\eqref{eqn:L5-0} are $K^M_*(R)$-linear. 
It therefore suffices
  to prove the lemma for $n =1$. But in this case, both 
$\partial^M$ and $\partial^Q$ are simply the 
valuation map of $F$ corresponding to $z$.
\end{proof}

\begin{lem}\label{lem:L6}
Let $k \inj k'$ be a finite extension of fields and $n \ge 0$ an integer. 
Then we have a commutative
diagram
\begin{equation}\label{eqn:Norm-PF-0}
\xymatrix@C.8pc{
K^M_n(k') \ar[r]^-{N_{{k'}/k}} \ar[d] & K^M_n(k) \ar[d] \\
K_n(k') \ar[r]^-{T_{{k'}/k}} & K_n(k).}
\end{equation}
\end{lem}
\begin{proof}
Assume first that $k \inj k'$ is a simple extension so that 
$k' = {k[t]}/{\fm}$ for some maximal ideal
  $\fm \subset k[t]$. Let $v_{\infty}$ be the valuation of $k(t)$ 
associated to the point $\infty \in \P^1_k$.
  Its valuation ring $R_{\infty} \subset k(t)$ has uniformizing 
parameter $t^{-1}$.
  In this case, we have the following diagram:
  \begin{equation}\label{eqn:L6-1}
    \xymatrix@C.8pc{
      0 \ar[r]& K^M_{n+1}(k) \ar[r] \ar[d] & K^M_{n+1}(k(t)) 
\ar[r]^{\partial^M=(\partial_v)_v~~~~} \ar[d]&
      {\underset{v \neq v_{\infty}}\oplus} K^M_n(k(v)) \ar[d]\ar[r] &0 \\
      0 \ar[r] & K_{n+1}(k) \ar[r] & K_{n+1}(k(t)) 
\ar[r]^{\partial^Q=(\partial_v)_v~~~~}&
      {\underset{v \neq v_{\infty}}\oplus}  K_n(k(v)) \ar[r] & 0.}
\end{equation}

The horizontal arrows on the left in both rows are induced 
by the inclusion $k \subset k(t)$.
The top row is Milnor's exact sequence 
(see \cite[Chapter~III, Theorem~7.4]{Weibel-1}).
The bottom row is the localization sequence in Quillen $K$-theory
(using the isomorphism $K_*(k) \xrightarrow{\cong} K_*(k[t])$)
and is known to be exact (see \cite[Chapter~V, Corollary~6.7.1]{Weibel-1}).
The right square commutes by \lemref{lem:L5}.

On the other hand, we have another diagram
\begin{equation}\label{eqn:L6-2}
  \xymatrix@C1.2pc{
    & K^M_n(k') \ar@{^{(}->}[d] \ar[dr]^-{N_{{k'}/k}} & \\
 K^M_{n+1}(k(t)) \ar@{->>}[r]^{\partial^M=(\partial_v)_v~~~~} \ar[d]&
     {\underset{v \neq v_{\infty}}\oplus} K^M_n(k(v)) 
\ar[d]\ar[r]^-{\sum_v N_{{k(v)}/{k}}} & K^M_n(k) \ar[d] \\
 K_{n+1}(k(t)) \ar@{->>}[r]^{\partial^Q=(\partial_v)_v~~~~}&
 {\underset{v \neq v_{\infty}}\oplus} K_n(k(v)) 
\ar[r]^-{\sum_v T_{{k(v)}/k}} & K_n(k) \\
 & K_n(k') \ar@{^{(}->}[u] \ar[ur]_-{T_{{k'}/k}}.}
\end{equation}

By the definition of the norm $N_{{k'}/k}$ in Milnor $K$-theory, the 
composition of the horizontal
arrows on the top is the map $(-1) \partial^{M}_ \infty \colon K^M_{n+1}(k(t)) \to K^M_n(k)$ 
(see \cite[Chapter~III, Definition~7.5]{Weibel-1}).
Similarly, the composite of the horizontal arrows on the bottom is the map
$(-1) \partial^{Q}_\infty \colon K_{n+1}(k(t)) \to K_n(k)$ 
(see \cite[Chapter~V, 6.12.1]{Weibel-1}).
Note that both of these assertions are another way of stating the 
Weil reciprocity formulas for the Milnor and
Quillen $K$-theories.

Since the left horizontal arrows in both rows are surjective, we 
are reduced to showing therefore
that the diagram
\begin{equation}\label{eqn:L6-3}
  \xymatrix@C.8pc{
    K^M_{n+1}(k(t)) \ar[r]^-{\partial^{M}_\infty} \ar[d] & K^M_n(k) \ar[d] \\
    K_{n+1}(k(t)) \ar[r]^-{\partial^{Q}_\infty} & K_n(k)}
\end{equation}
commutes. But this follows from \lemref{lem:L5}. This proves the 
lemma for simple extensions.

In general, we can write $k' = k(x_1, \ldots, x_r)$. Since the norm maps in 
Milnor $K$-theory and the
push-forward maps in Quillen $K$-theory satisfy the transitivity property, 
and since $k \inj k'$ is a
composite of simple extensions, the proof of the lemma follows.
\end{proof}

\vskip .4cm

\noindent\emph{Acknowledgments.}
The first author would like to thank TIFR, Mumbai for
invitation in March 2019.
This paper was written when the  second author was at Max Planck Institute for Mathematics, Bonn in 2019.
He thanks the institute for invitation and support.
The authors thank the referee 
for reading the manuscript thoroughly and providing
valuable suggestions to improve its presentation. 
They also thank the editors some of whom provided very useful comments and suggestions.

\end{document}